\documentclass{amsart}
\usepackage[dvips,dvipdf]{graphicx}
\usepackage[usenames,dvipsnames]{color}
\usepackage{subfig}
\usepackage{amsmath,amssymb,amsthm}
\usepackage{cite}
\usepackage{ulem}
\usepackage{algorithm} 
\usepackage{algpseudocode}

\newtheorem{theorem}{Theorem}[section]
\newtheorem{proposition}[theorem]{Proposition}
\newtheorem{corollary}[theorem]{Corollary}

\theoremstyle{definition}
\newtheorem{definition}[theorem]{Definition}
\theoremstyle{definition}
\theoremstyle{remark}
\newtheorem{remark}[theorem]{Remark}
\newtheorem{example}[theorem]{Example}

\def\cL{\mathcal{L}}
\def\RR{\mathbb{R}}
\def\NN{\mathbb{N}}

\newcommand{\jmp}[1]{[\![#1]\!]}                     % jump
\newcommand{\mvl}[1]{\{\!\!\{#1\}\!\!\}}             % mean value 

\newcommand{\enorm}[2]{\|\,{#1}\,\|_{#2}}       % diffusive norm

\newcommand{\change}[1]{{ #1}}

\title[Landscape Adaptive Refinement]{Adaptive refinement for
  eigenvalue problems based on an associated source problem}
\author{Stefano Giani, Jeffrey Ovall and Gabriel Pinochet-Soto}
\date{\today}

\begin{document}

\maketitle

\section{Introduction}

This paper concerns the efficient approximation of several eigenpairs
of a class of selfadjoint elliptic operators via an adaptive finite
element scheme.  We consider eigenvalue problems of the form
\begin{align}\label{Problem}
\cL\psi=\lambda\psi\mbox{ in }\Omega\;,\; \psi=0\mbox{ on }\partial\Omega\quad,\mbox{ where }
\cL w=-\nabla\cdot(A\nabla w)+Vw~,
\end{align}
and $\Omega\subset\RR^2$ is a bounded domain.  The potential $V$ and
matrix $A$ are assumed to be piecewise smooth on some partition of the
domain, $\overline\Omega=\bigcup_{k=1}^m\overline{\Omega_k}$ and
$\Omega_i\cap\Omega_j=\emptyset$ when $i\neq j$.  We further assume
that $V\geq 0$ and that $A$ is uniformly positive definite a.e. on
$\Omega$.  Since we will discretize the problem using standard finite
elements, it is convenient to assume that $\Omega$ and the subdomains
$\Omega_k$ are polygons, and that all triangulations of $\Omega$
respect the partition (i.e. each triangle is contained in precisely
one of the $\Omega_k$). The approach we will propose extends naturally
to $\RR^3$, but our numerical experiments will be in 2D, so we focus
on that case in the discussion.

Adaptive finite element schemes for eigenvalue computations have a
long history, with contributions typically mirroring (and often
following soon after) their counterparts for source problems.  In the
case of computing clusters of eigenvalues, where repeated eigenvalues
(or merely tightly clustered eigenvalues) may be present, together
with bases for their corresponding invariant subspaces is more a
challenging task, and contributions in this \change{direction are relatively}
recent in comparison with those that concern a few simple eigenvalues
near the bottom of the spectrum. We do not attempt a comprehensive
survey of the extensive literature on the topic, and instead provide
some references that more specifically concern the approximation of
several, possibly clustered, eigenvalues, together with a basis of the
corresponding invariant
subspace~\cite{Grubisic2009,Bank2013,Gallistl2015,Boffi2017,Cances2020,Giani2021,Giani2021a,Liu2022}.
A common theme of nearly all of these contributions it that, whatever
strategy one ultimately chooses for adaptively enriching the finite
element space used to approximate the eigenpairs, the resulting finite
element space should be well-suited to approximate generic functions
in the entire invariant subspace.  An exception to this pattern is the
work~\cite{Giani2021a}, in which multiple meshes/FEM spaces are
employed.

The generic template for adaptive finite element procedures, for
source problems or eigenvalue problems, is to \change{repeat the following}
steps until some convergence criterion is met:
\begin{center}
SOLVE $\longrightarrow$ ESTIMATE $\longrightarrow$ MARK 
$\longrightarrow$ REFINE
\end{center}
Global error estimates and local error indicators are obtained in the
ESTIMATE step from the approximation(s) obtained in the SOLVE step.
In the MARK step, some elements are flagged for refinement based on
the local indicators.  These elements are refined, either by
subdividing the element ($h$ refinement) or increasing the polynomial
degree in that element ($p$ refinement).  Some other elements,
typically near the marked elements, may also need to be refined to
maintain some type of conformity. In the case of eigenvalue problems
in which multiple eigenpairs are sought, elements are generally marked
for refinement based on local error indicators that consider
\textit{all} of the current computed eigenpairs---different approaches
primarily differ in \textit{how} this information is used for marking.
A distinguishing feature of the present contribution is that we
propose a marking and refinement strategy that is based solely on the
approximate solution of a single source problem, namely
\begin{align}\label{Landscape}
\cL u=1\mbox{ in }\Omega\;,\; u=0\mbox{ on }\partial\Omega~.
\end{align}
We will refer to $u$ as the \textit{landscape function} for the
operator and domain, which is what this function is typically called
in the context of eigenvector localization
(cf.~\cite{Filoche2012,Arnold2016,Arnold2019,Arnold2019a}).

\change{As will be seen in Corollary~\ref{LControl} and
  Proposition~\ref{FourierLandscape}, and in the experiments, we
  expect the landscape function to reliably encode information
  concerning where eigenvectors are singular or smooth, or where they
  exhibit other interesting features such as localization.  As such,
  we contend that refinement based on features of the landscape
  function will typically provide a family of adapted meshes that
  resolve the eigenvectors of interest at an optimal rate, and will do
  so more efficiently than the eigenvector cluster approach,
  particularly in the case of $hp$-adaptivity.}

\begin{remark}[Some caveats]\label{LRCanFail}
  It is possible to construct examples, particularly in 1D, for which
  the landscape function $u$ is a polynomial (of low degree) in some
  subdomain $R$, but for which the eigenvectors $\psi_j$ of interest
  are not polynomial on $R$.  For such an example we might rightly
  expect refinement driven by approximations of $u$ to cease refining
  within $R$ at some stage, and thereby fail to provide convergent
  eigenvector approximations.  Such situations are highly unlikely in
  practice, however, and we will provide both \change{theory and substantial
  empirical evidence} that ``landscape refinement'' is a very effective
  approach for the efficient approximation of (large) clusters of
  eigenvalues and eigenvectors.  \change{Such examples illustrate why
    a result of the form $\|\psi-\psi_{hp}\|\leq C
    \|u-u_{hp}\|$ cannot hold in general.  However, the preceding
    paragraph suggests why such a result does not need to hold for the
    approach to be effective.  In particular, although we do
    compute a posteriori estimates of $\|u-u_{hp}\|$ for some
    comparisons, these are not used as part of any stopping criterion.
  }
  In Example~\ref{DiscDiff}, we
  illustrate a different situation in which a few eigenvectors within
  a large cluster are converging more slowly under landscape
  refinement than the rest, and we indicate \change{what goes wrong in
    this case, and how to efficiently restore the proper rate of convergence}.
\end{remark}

The rest of the paper is structured as follows.  In
Section~\ref{Theory}, we provide theoretical results that link the
landscape function $u$ to eigenpairs $(\lambda_j,\psi_j)$ in a way
that makes the efficient approximation of such eigenvectors on finite
element spaces based on meshes generated by landscape refinement more
reasonable than it may appear at first glance.  Preliminary
experiments in this section add to the evidence that landscape
refinement might actually be a viable alternative to refinement based
on the (approximate) eigenpairs of interest.
Corollary~\ref{LControl}, which concerns the pointwise control of
properly scaled eigenvectors by $u$ is, by now, a well-known
result (cf.~\cite{Filoche2012,Arnold2019a,Arnold2016,David2021}).
Our more general result, Theorem~\ref{LandscapeStability}, does not
appear to be broadly known---we have not found it in the literature.
Apart from the initial proof-of-concept computations in Section~\ref{Theory}, our
numerical experiments are done using the SIPG variant of $hp$ DG
finite element discretizations.  These are briefly described in
Section~\ref{sec:DGmethods}, together with the error estimators and
local error indicators that we use in our experiments.
Section~\ref{Experiments} contains the bulk of the experiments, in
which we illustrate the effectiveness of our approach on a variety of
problems, making relevant comparisons with more traditional approaches.
Details about the algorithms used for these comparisons are given in
Section~\ref{sec:alg}, and further enhancements of our proposed
algorithm are considered in Section~\ref{AlgOpt}.  A few concluding
remarks are given in Section~\ref{Conclusions}

%%%%%%%%%%%%%%%%%%%%%%%%%%%%%%%%%%%%%%%%%%%%%%%%%%%%%%%%%%%%%%%%
%%%%%%%%%%%%%%%%%%%%%%%%%%%%%%%%%%%%%%%%%%%%%%%%%%%%%%%%%%%%%%%%

\section{Basic Properties of the Landscape Function}\label{Theory}

The assumptions on $A$ and $V$ imply that we have the classical
maximum principle for the operator $\cL$.  A straight-forward
application of the Maximum Principle yields the following refinement
of the standard stability estimate, which involves the landscape
function in a very natural way.  We provide a brief proof for
completeness.
\begin{theorem}[Pointwise Stability Estimate]\label{LandscapeStability}
If $v\in C(\overline{\Omega})$ and $\cL v\in L^{\infty}(\Omega)$, then
\begin{align}\label{LStab}
|v(x)|\leq \|v\|_{L^\infty(\partial\Omega)}+\|\cL v\|_{L^{\infty}(\Omega)}\,u(x)~,
\end{align}
for each $x\in\overline\Omega$.
\end{theorem}
\begin{proof}
  Let $w_\pm=\pm v-\|\cL v\|_{L^{\infty}(\Omega)}\,u(x)$.  By
  construction, we have $\cL w_{\pm}\leq 0$ on $\Omega$, and it
  follows by the Maximum Principle that
\begin{align*}
w_{\pm}(x)\leq \|w_{\pm}\|_{L^\infty(\partial\Omega)}=\|v\|_{L^\infty(\partial\Omega)}~.
\end{align*}
Rearranging terms yields~\eqref{LStab}.
\end{proof}
\noindent
The standard stability result replaces $u(x)$ by a constant $C$
in~\eqref{LStab}.  The inequality~\eqref{LStab} is sharp, as can be
seen by choosing $v=u$.

An immediate consequence of this stability result for the eigenvalue
problem~\eqref{Problem} is given below.
\begin{corollary}[Pointwise Control of Eigenvectors]\label{LControl}
Let $(\lambda,\psi)$ be an eigenpair of $\cL$.  It holds that
\begin{align}\label{LCont}
\frac{|\psi(x)|}{\lambda\|\psi\|_{L^\infty(\Omega)}}\leq 
u(x)\mbox{ for all }x\in\overline\Omega~.
\end{align}
\end{corollary}
\noindent
This result first appeared in~\cite{Filoche2012} (with a different
proof), and demonstrates an important connection between the landscape
function and eigenvectors lower in the spectrum: where the landscape
function is relatively small, such eigenvectors are obliged to be
relatively small; and where the landscape function is relatively
large, such eigenvectors are permitted to be relatively large.  We
illustrate Corollary~\ref{LControl} in Example~\ref{1DExample}.  The
landscape function, through this result and many others, provides a
surprisingly rich understanding of the phenomenon of eigenvector
localization, at least lower in the spectrum.

\begin{example}[1D Example]\label{1DExample}
  In order to illustrate Corollary~\ref{LControl}, we consider the 1D
  problem in which $\Omega=(0,1)$ and $\cL=-\frac{d}{dx^2}+V$.  Here,
  the potential $V$ is piecewise constant on a uniform partition of
  $\Omega$ into 32 subintervals of equal length, as pictured in
  Figure~\ref{1DFigA}.  It varies between $5274.4361$ and $98928.04$.
  The corresponding landscape function is also given in this figure,
  and its local extrema are highlighted. The landscape function in
  this case has 13 local maxima.  The largest of these is roughly
  $1.3501\times 10^{-4}$, and it occurs near $x=0.42334$; the smallest
  is roughly $1.5101\times 10^{-5}$, and it occurs near $x=0.20361$.
  The landscape theory
  from~\cite{Filoche2012,Arnold2016,Arnold2019,Arnold2019a} predicts
  that there will be a ``ground state'' eigenvector associated with
  each of the peaks of $u$, and that each such ground state is largely
  concentrated (localized) near the corresponding local maximizer.
  Such a ``ground state'' has its peak near the corresponding peak of
  $u$, and decays rapidly away from this peak, generally in such a way
  that you cannot \textit{see} any sign changes---we note that there
  is only one true ground state, associated with the smallest
  eigenvalue, and that is the only eigenvector that is purely of one
  sign on $\Omega$.  These assertions are illustrated in
  Figure~\ref{1DFigB}.  In this figure, we show scaled versions of the
  eigenvectors $\psi_{j}$, for $1\leq j\leq 9$ and $47\leq j\leq 49$.
  Among these, six correspond to what one might naturally consider a
  ground state, namely those for $j=1,2,3,6,9$ and $j=48$.  The
  eigenvector $\psi_{48}$ is associated with the smallest of the peaks
  of $u$, and is the last of the ground states.  The eigenvectors
  $\psi_4$ and $\psi_7$ might also be considered ground states,
  despite the fact that we can clearly observe a sign change, because
  their peaks are very closely related to two peaks of $u$ that are
  not otherwise ``accounted for'', in the sense that, if we did not
  count them as ground states, we would not have 13 ground states to
  match with the 13 local maxima of $u$.  The five remaining ground
  states, which are not pictured in Figure~\ref{1DFigB}, are
%\begin{align*}
  $\{\psi_{10},\psi_{13},\psi_{16},\psi_{19},\psi_{27}\}$.
%\end{align*}
\end{example}

\begin{figure}
\centering
\includegraphics[width=0.45\textwidth]{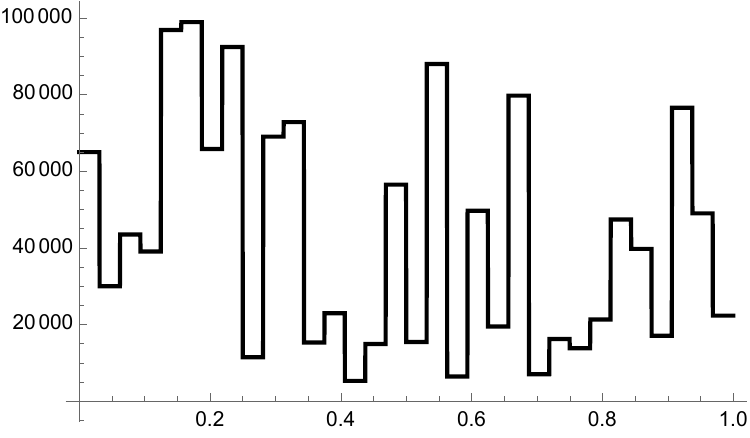}
\includegraphics[width=0.5\textwidth]{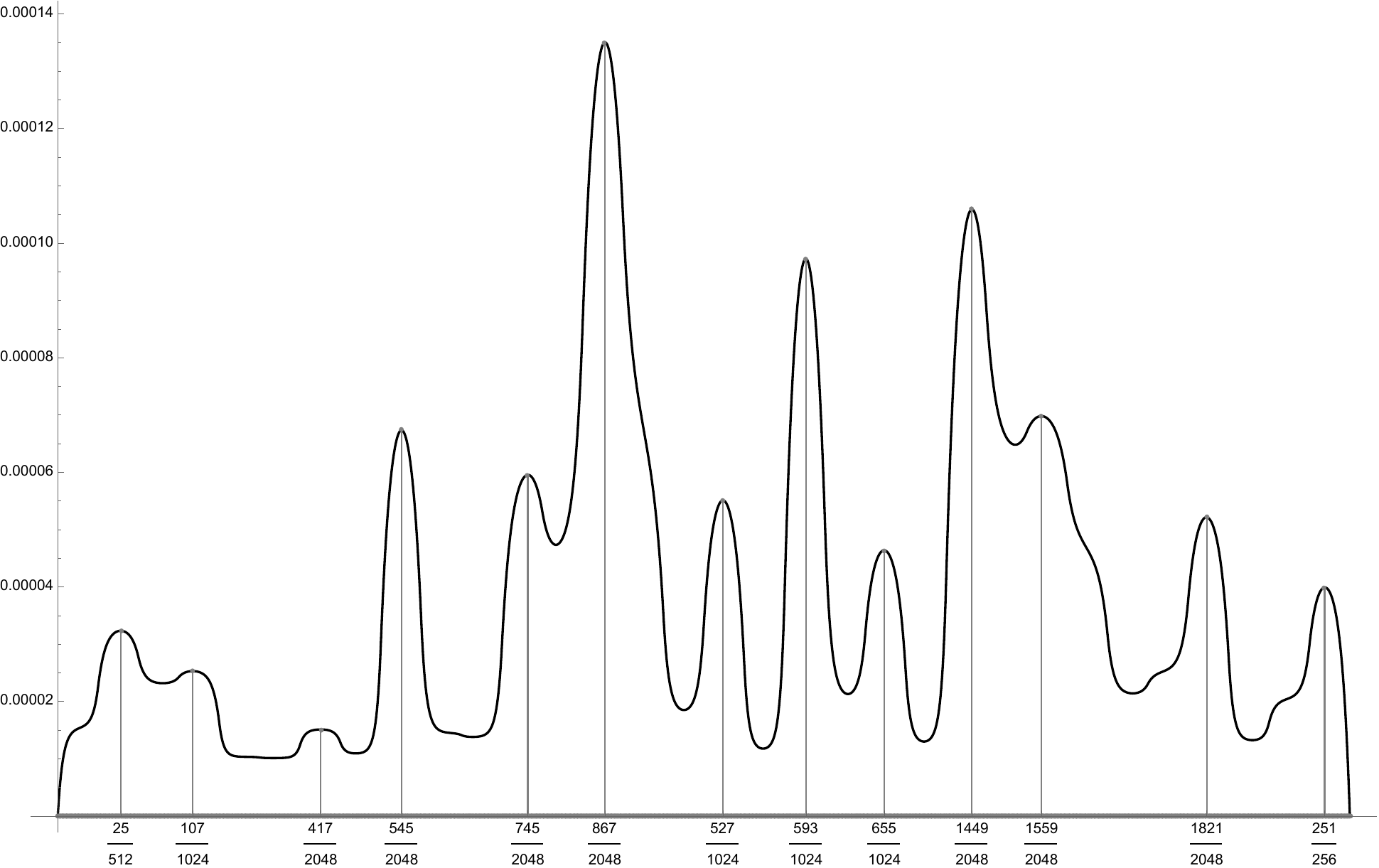}
\caption{The potential $V$ and associated landscape function $u$ for Example~\ref{1DExample}.  The local maxima of $u$ are highlighted.}\label{1DFigA}
\end{figure}

\begin{figure}
\centering                                                
\includegraphics[width=0.32\textwidth]{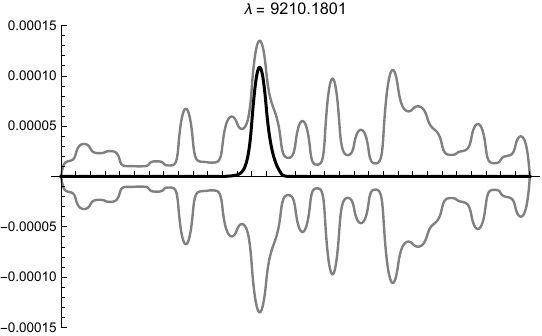}
\includegraphics[width=0.32\textwidth]{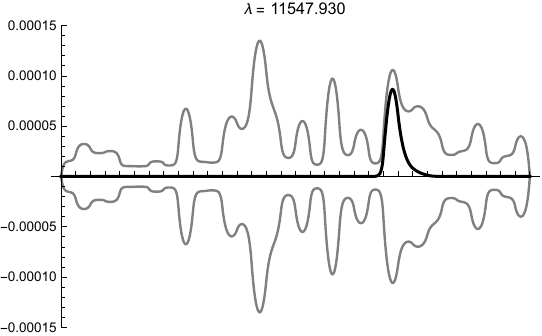}
\includegraphics[width=0.32\textwidth]{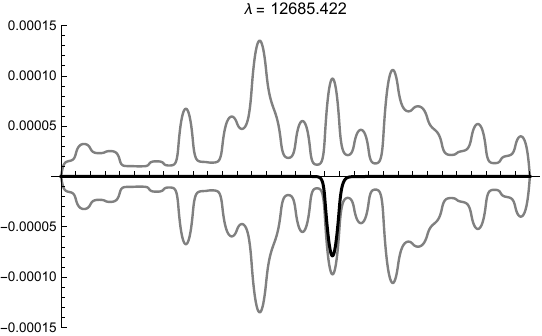}
\includegraphics[width=0.32\textwidth]{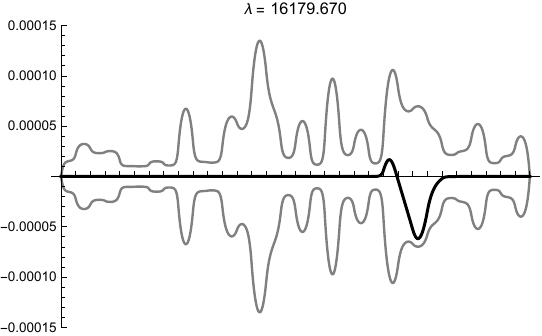}
\includegraphics[width=0.32\textwidth]{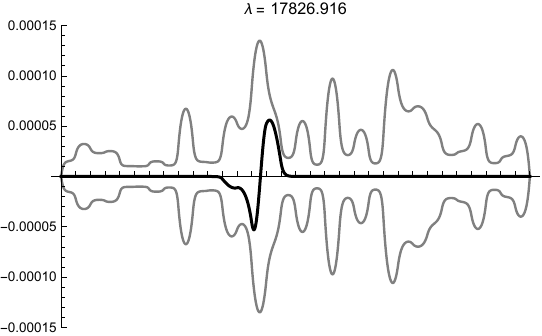}
\includegraphics[width=0.32\textwidth]{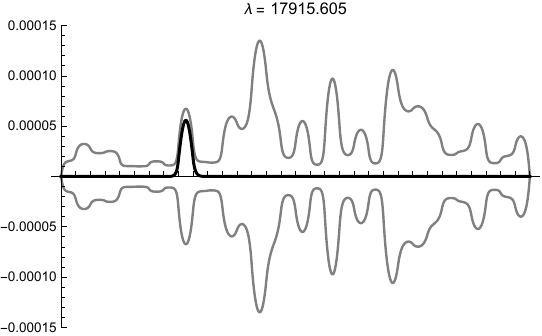}
\includegraphics[width=0.32\textwidth]{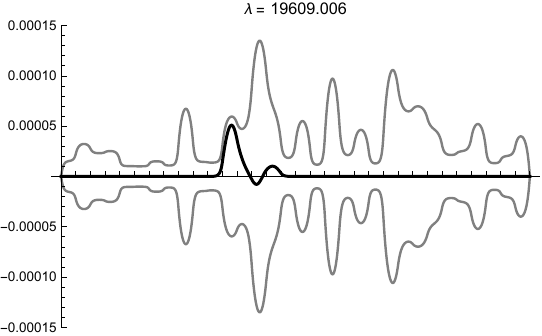}
\includegraphics[width=0.32\textwidth]{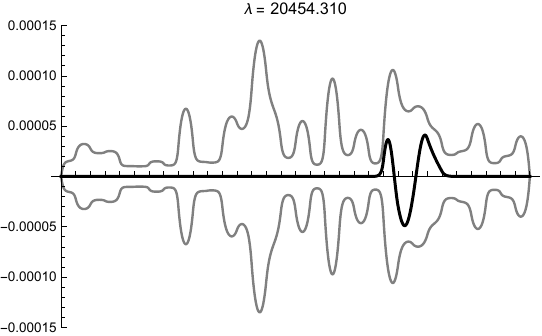}
\includegraphics[width=0.32\textwidth]{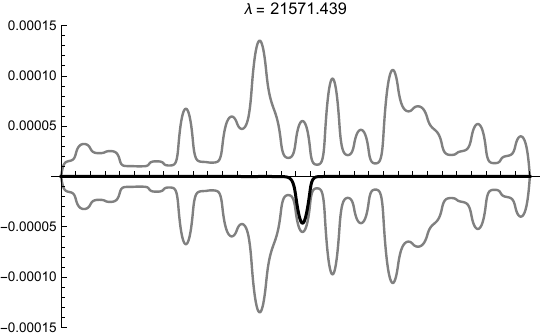}
\includegraphics[width=0.32\textwidth]{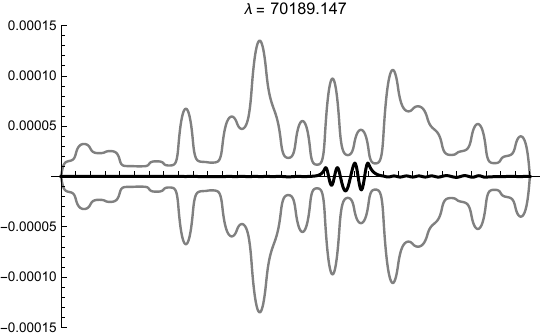}
\includegraphics[width=0.32\textwidth]{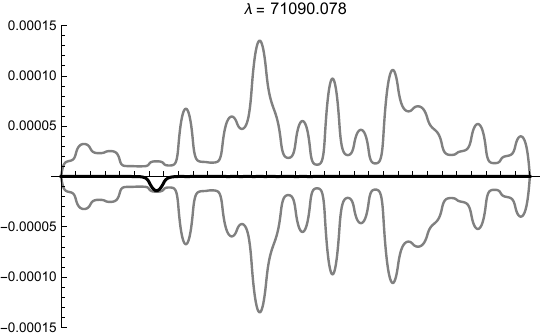}
\includegraphics[width=0.32\textwidth]{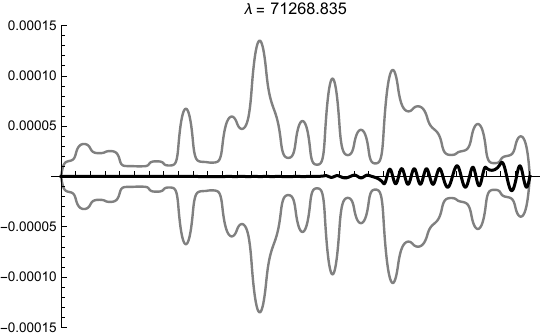}
\caption{\label{1DFigB} Several scaled eigenvectors (black) early in
  the spectrum, together with $u$ and $-u$ (grey), which act as an envelope
  of the eigenvectors.  The eigenvectors are scaled as indicated
  in~\eqref{LCont}.  Eigenvectors 1-9 and 47-49 for Example~\ref{1DExample}.}
\end{figure}

We have seen in this example how the landscape function provides
remarkably tight control of the behaviour of ground state
eigenvectors.  We have also seen that, although~\eqref{LCont} holds
for any eigenpair of $\cL$, the bound is much more meaningful lower in
the spectrum, for eigenvectors that are not highly oscillatory.
Instead of considering how the landscape function governs the
behaviour of eigenvectors low in the spectrum, we might consider how
such eigenvectors dictate the shape of the landscape function. This
perspective is captured in the following result, whose proof is
omitted.
\begin{proposition}[Fourier Expansion of the Landscape
  Function]\label{FourierLandscape}  Let $(\lambda_n,\psi_n)$,
  $n\in\NN$, be eigenpairs of $\cL$ such that $\{\psi_n:\,n\in\NN\}$ is
  an orthonormal Hilbert basis (a Fourier basis) of $L^2(\Omega)$.
  The landscape function has an expansion
  \begin{align}\label{eq:landscape_expansion}
    \change{u=\sum_{n\in\NN}c_n\,\psi_n
    \;,\; c_n=\left(\int_\Omega \psi_n\,dx\right)/\lambda_n~.}
  \end{align}
\end{proposition}
\noindent
Roughly speaking, we expect the dominant features of $u$ to look like
a finite linear combination of ground states lower in the spectrum,
because it is these vectors that should have the largest coefficients,
$c_n$, in a relative sense.  This assertion is illustrated in
Figure~\ref{1DFigC}.
\begin{figure}
\centering
\includegraphics[width=0.48\textwidth]{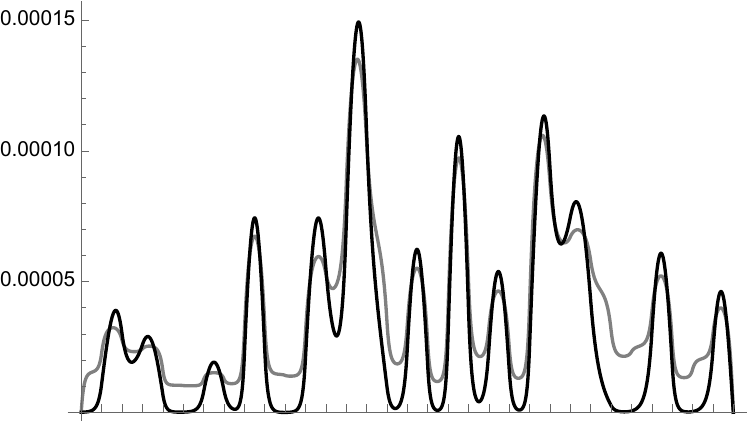}
\includegraphics[width=0.48\textwidth]{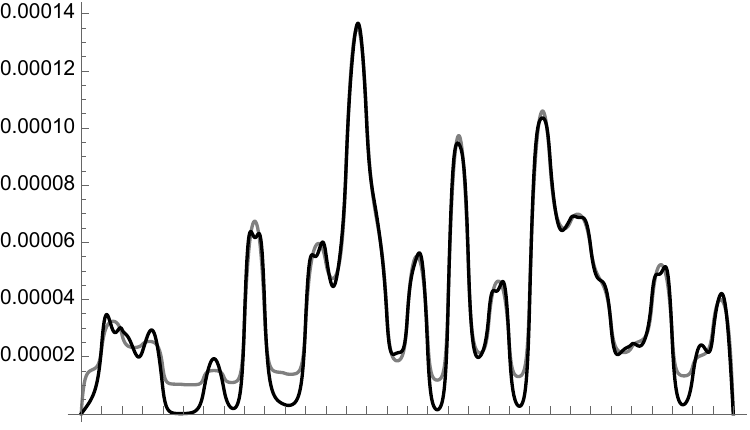}
\caption{\label{1DFigC} The landscape function $u$ (grey) for Example~\ref{1DExample} and its partial
  Fourier expansion (black) using only the 13 ground states. At right,
the partial Fourier expansion of $u$ using the first $50$ eigenvectors.}
\end{figure}

\change{The eigenbasis expansion~\eqref{eq:landscape_expansion}
  indicates how features of the eigenvectors are encoded in the
  landscape function $u$.  More specifically, this expansion
  demonstrates why we generally expect the landscape function to
  identify the locations and strengths of all singularities that
  \textit{may} occur in eigenvectors, highlighting the strongest
  singularities that can occur at each singular point.  In particular
  $u$ will have the same singularities as the ground state eigenvector
  $\psi_1$, which often exhibits the strongest singularities that can
  occur for any eigenvector.  The causes and strengths of
  singularities for the types of problems we consider are well-known
  (cf.\cite{Borsuk2006,Grisvard1985,Grisvard1992,
    Kellogg1974,Kondratiev1967,Wigley1964}).  More specifically, these
  authors identify the kinds of singularities that can occur in
  solutions source problems $\cL v = f$ under reasonable regularity
  assumptions on $f$.  This theory clearly applies to the case of
  eigenvalue problems, by taking $f=\lambda v$, so it provides a rich
  understanding of the sorts of singularities eigenvectors may
  exhibit.  Assuming that $u$ does encode the possible singularities
  of eigenvectors, and that adaptive mesh refinement for $u$ is done
  in such a way as to generate meshes/spaces that yield optimal order
  convergence of approximations of $u$, it is unsurprising that
  computed eigenvectors also converge at the optimal order, \textit{at
    least asymptotically}.  This is precisely the behavior we observe
  in nearly all of our experiments---Example~\ref{DiscDiff}, briefly
  discussed below, offers a counter-example.
  
  It is possible to construct problems for which the source term $f=1$
  is orthogonal to all eigenvectors having a given singular behavior,
  in which case the landscape function will fail to identify the
  singularity.  This is almost certainly what is happening in
  Example~\ref{DiscDiff}, for which we observe that the (empirically)
  most singular eigenvectors converge more slowly under landscape
  refinement than the rest.  In such a case, one might naturally
  consider using a different source problem $\cL u_f = f$ to drive
  adaptivity, possibly in conjunction with the landscape function.
  For example, if it appears that $1$ is orthogonal to all
  eigenvectors having a given singularity, then choosing a simple
  non-trivial $f$ that is orthogonal to $1$ may identify, via $u_f$,
  the correct singular behavior of the eigenvectors that was missed by
  $u$.  In fact, one expects a ``randomly chosen'' $f$ to do the job.
  Understood in that light, one might wonder why we propose $f=1$ as
  the default choice.  Theorem~\ref{LandscapeStability},
  Corollary~\ref{LControl} and Example~\ref{1DExample}, provide
  justification that this is a very sensible choice, and our
  experiments provide further support.  Furthermore, the body of work
  exemplified by the aforementioned
  contributions~\cite{Filoche2012,Arnold2019a,Arnold2016,David2021}
  demonstrates the perhaps surprising amount of information $u$
  reveals about the spectrum of Schr\"odinger operators.
  Additionally, the choice $f=1$ does not introduce any particular
  biases toward certain features that may not be present in most eigenvectors.

  % Conversely, it is possible that all eigenvectors are smooth, but
  % that the landscape function possesses (mild) singularities.  This
  % happens for the Laplacian on a square, for example, for which all
  % eigenvectors are analytic, but the landscape function has
  % $r^2\ln r$-type singularities at each corner, where $r$ denotes the
  % distance to that corner.

  We are primarily interested in employing adaptive refinement based
  on approximations of the landscape function when computing moderate
  or large collections of eigenpairs.  In that setting, it is probable
  that some (many) of the eigenvectors of interest will have the strongest singularities
  that are possible for the given domain and differential operator,
  and our adaptive scheme is expected to yield convergence of
  \textit{collective measures} of error that are not only
  asymptotically optimal, but also quite reasonable in the
  pre-asymptotic regime.  Of course, our one-mesh-fits-all approach
  will be overly aggressive in its refinement toward singular points
  in terms of the approximation of smooth(er) eigenvectors, but we do
  not see this as a disqualifying drawback.  At any rate, refinement
  based on clusters of computed eigenvectors that contain some smooth
  and some singular eigenvectors would suffer the same drawback, and
  do so at greater computational cost.
}

The discussion above was intended to establish some
connections between the eigenvectors of $\cL$ and the landscape
function $u$, as a \textit{first} indication that adaptive refinement
that is aimed at approximating the landscape function might, in fact,
yield finite element spaces in which the eigenvectors are also well
approximated.  We now provide a first illustration of this approach
and its performance on a model problem for which there is no
eigenvector localization, but for which adaptive refinement is
advisable.  In Section~\ref{Experiments}, we provide strong supporting
evidence for the ``landscape refinement''
approach \change{on a variety} of problems.

\begin{example}[L-Shape Domain, First Look]\label{hAdaptLShape}
  We here consider the Laplacian $\cL=-\Delta$ on the domain
  $\Omega=(-1,1)\times(-1,1)\setminus[0,1)\times[0,1)$.  Most
  eigenvalues/vectors in this case are not known explicitly, though some highly
  accurate approximations of a few eigenvalues are given in the
  literature (cf.~\cite{Trefethen2006}).  It is known is that some
  eigenvectors have an $r^{2/3}$-singularity near the origin
  (cf.~\cite{Grisvard1985,Grisvard1992}).  This is true of the
  ground state eigenvector $\psi_1$, for instance.  So adaptive
  refinement is natural in this setting.

  We approximate the first five eigenpairs of $\cL$, all of which are
  simple, via an $h$-adaptive conforming finite element method on
  triangular meshes, with fixed local polynomial degree $p=2$.
  Without elaborating on the details of local error indicators and
  marking and refinement strategies (they are standard), we provide
  the final meshes generated by two different refinement strategies in
  Figure~\ref{fig:hAdaptLShapeMeshes}.  The first is driven by the
  computed eigenvectors, and the second is driven by the computed
  landscape function. Both start from the same quasi-uniform coarse
  mesh.  We use (CR) and (LR) respectively to denote
  ``cluster refinement'' and ``landscape refinement''.

  The reference values
  \begin{align*}
    &\lambda_1=9.639723844\;,\; \lambda_2=15.19725193\;,\;
      \lambda_3=19.73920880\\
    &\lambda_4=29.52148111\;,\;
    \lambda_5=31.91263596~,
  \end{align*}
  which were computed by the method of particular
  solutions~\cite{Betcke2005}, and are correct in all digits shown,
   were used to obtain the errors in our eigenvalue approximations.
  We note that $\lambda_3=2\pi^2$ and $\psi_3=\sin(\pi x)\sin(\pi y)$
  (up to scaling).
  The two final meshes are
  qualitatively very different, but the eigenvalue approximation
  errors are comparable; see Figure~\ref{hAdaptLShape}.   In both
  cases, the errors $e_j=|\hat\lambda_j-\lambda_j|$ in the
  approximations $\hat\lambda_j$ exhibit the optimal convergence rates,
  namely $\mathcal{O}(\mathrm{DOF}^{-2})$.
  The
  eigenvalue errors on the finest meshes are
\begin{center}
 \begin{tabular}{|ccccccc|}\hline
    &DOF&$e_1$&$e_2$&$e_3$&$e_4$&$e_5$\\\hline   
   CR&118240&1.38e-07&1.86e-07&3.64e-07&1.04e-06&1.37e-06\\
   LR&119586&1.81e-07&4.19e-07&1.32e-06&6.11e-06&4.32e-06\\\hline
  \end{tabular}  
\end{center}
Although CR wins in each head-to-head comparison of error, the errors
are comparable and decrease at the correct rate in both cases, and LR
is computationally cheaper, so LR might be considered a viable
alternative.  In later experiments, \change{in which a larger collection
of eigenvalues is desired}, we will see that LR becomes an even more
attractive alternative.

\begin{figure}
  \centering
   \includegraphics[width=0.45\textwidth]{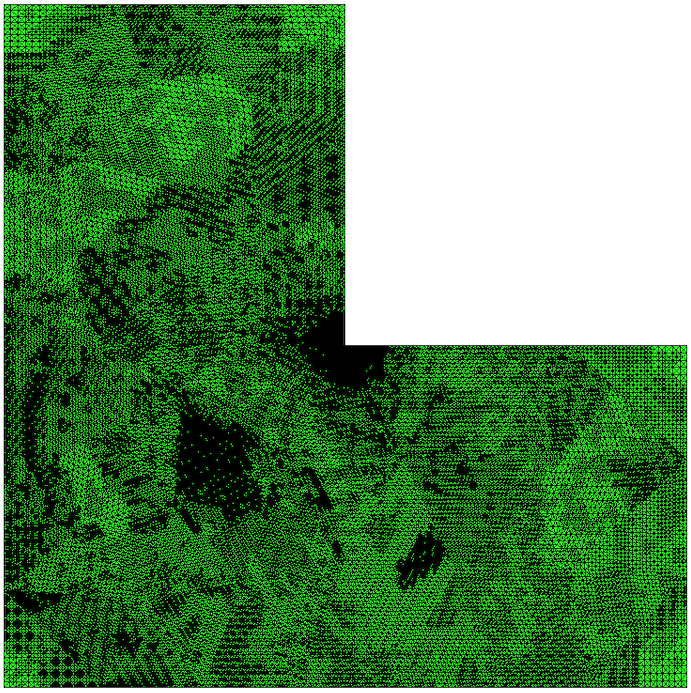}
  \includegraphics[width=0.45\textwidth]{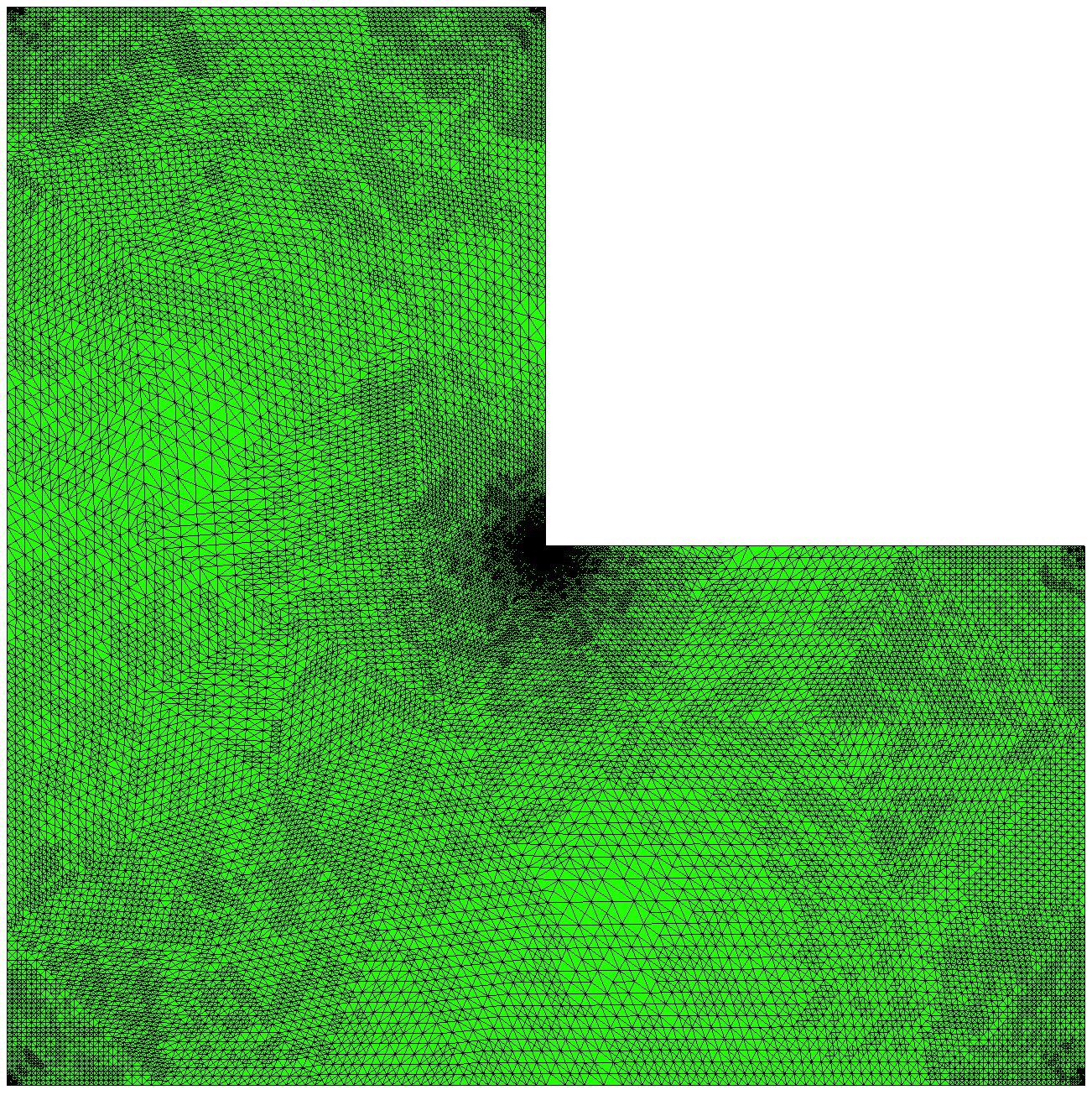}\\
  \includegraphics[width=0.45\textwidth]{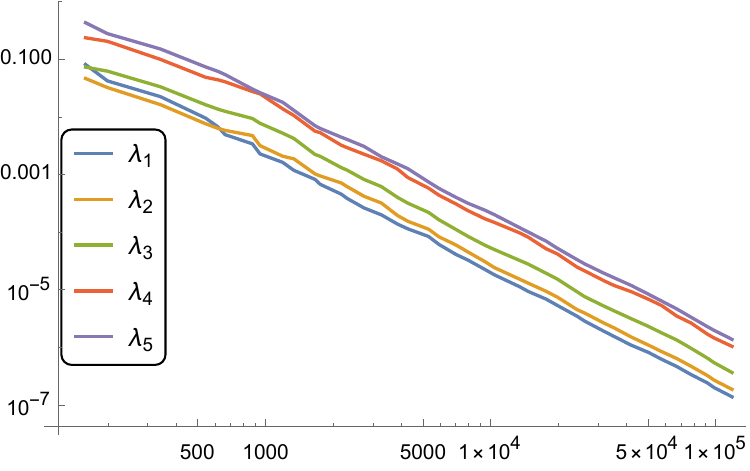}
  \includegraphics[width=0.45\textwidth]{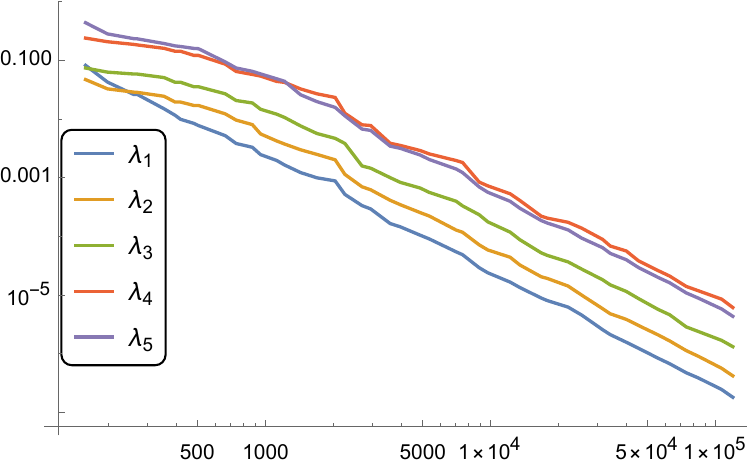} 
   \caption{Example~\ref{hAdaptLShape}. Top row: two meshes having roughly the same
    number of degrees of freedom, obtained by CR
    (left) and LR for the first five
    eigenvalues.  Bottom row: Error plots for CR (left) and LR.
    \label{fig:hAdaptLShapeMeshes}}
\end{figure}

We next consider a cluster of
eigenpairs a bit higher in the spectrum.  The interval
$[50\pi^2-10,\change{50\pi^2}+10]$ contains seven eigenvalues of $\cL$; four
are simple, but $\lambda=50\pi^2$ has a three-dimensional eigenspace
\begin{align*}
S=\mathrm{span}\{\sin(7\pi x)\sin(\pi y)\,,\, \sin(5\pi x)\sin(5\pi y) \,,\, \sin(\pi x)\sin(7\pi y)\}~.
\end{align*}
Reference values for these eigenvalues, obtained via the
method of particular solutions as before, are
\begin{align*}
  \lambda_{101}&=485.71752463708\;,\;
   \lambda_{102}=490.15998172598\;,\;
    \lambda_{103,104,105}=493.48022005447\\
    \lambda_{106}&=499.24106145290\;,\;
    \lambda_{107}=502.30119419396
\end{align*}

In this case, we consider three different refinement strategies:
landscape refinement (LR), eigenvector refinement (CR), and a mixed
strategy (MR) that begins with LR and then shifts to CR when the
dimension of the finite element space reaches some threshold.
The plots in Figure~\ref{fig:hAdaptLShapeHighSeven} contain the
computed eigenvalues in the interval $[50\pi^2-10, \change{50\pi^2}+10]$ for
both the LR and CR refinement strategies, starting from the same
coarse mesh as before.  Early in the adaptive cycle, neither approach
has produced a finite element space that is sufficient to resolve the
invariant subspace of interest---for the first several refinement
cycles, neither approach even \change{computes the correct number of
eigenvalues in the interval interest}!  In the case of CR, this means
that refinement during this phase is being at least partially driven
by approximations of that are, at best, related to eigenvectors of
$\cL$ for eigenvalues outside of $[50\pi^2-10, \change{50\pi^2}+10]$.  If CR is
generating useful meshes during this phase, it essentially doing so
``on accident''.  In fact, CR does seem to be generating useful meshes
\change{for approximating the invariant subspace} of interest, and it does so a
bit earlier than LR refinement, based on when we see a discernable
pattern of convergence toward the correct eigenvalues; for CR this
happens somewhere around 3000 or 4000 DOF, and for LR it is nearer to
8000 DOF.

\begin{figure}
  \centering
  \includegraphics[width=0.49\textwidth]{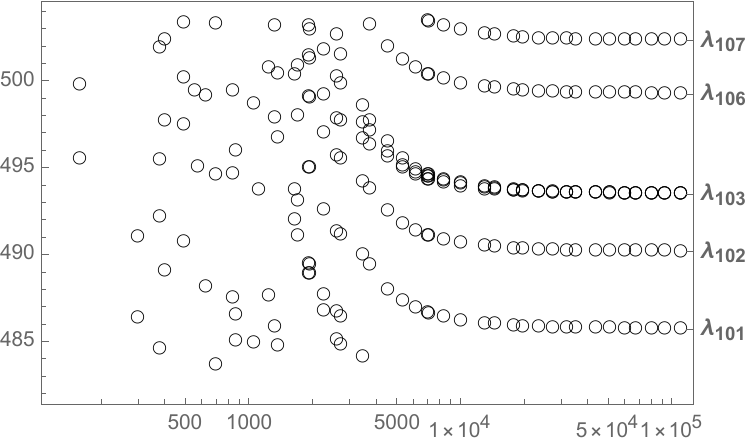}
  \includegraphics[width=0.49\textwidth]{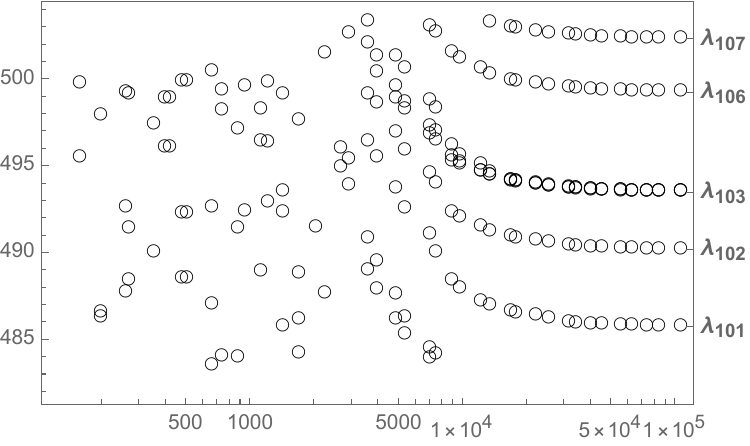}
  \caption{Example~\ref{hAdaptLShape}. Computed approximations of the
    (seven) 
    eigenvalues in $[50\pi^2-10,50\pi^2+10]$ over a sequence of
    adaptively refined meshes using CR (left) and LR.  Note that $\lambda_{103}=\lambda_{104}=\lambda_{105}=50\pi^2$.
    \label{fig:hAdaptLShapeHighSeven}}
\end{figure}

Table~\ref{FinalErrorsThreeRefinementStrategies} below shows the
error, $e_j=|\lambda_{j}-\hat\lambda_{j}|$, in the seven computed
eigenvalues $\hat\lambda_j$ for the first mesh exceeding 100,000 DOF
for each of the three refinement strategies.  For MR, we show the
results of three variants: MR1 switches from LR to CR as soon as
DOF $> 4000$; MR2 switches when DOF $ > 8000$; MR3 switches when
DOF $ > 10000$.  We first note that, although LR yields the largest
final errors among the five options all of these approaches exhibit
convergence consistent with the optimal rate $\mathcal{O}(\mathrm{DOF}^{-2})$,
despite the fact that the meshes used here for LR are precisely those
used for the smallest five eigenvalues.  We also remark that, although
Figure~\ref{fig:hAdaptLShapeHighSeven} indicates that CR begins to
accurately resolve the eigenvalues a bit sooner than LR, the mixed
approach MR3 that transitions from LR to CR latest actually performs
best.

\begin{table}
  \caption{\label{FinalErrorsThreeRefinementStrategies}
    Example~\ref{hAdaptLShape}.  Errors $e_j=|\lambda_{j}-\hat\lambda_{j}|$ on the finest mesh for
    approximating the seven eigenvalues in $[50\pi^2-10, \change{50\pi^2}+10]$
    using several adaptive refinement strategies.}
\begin{small}
\begin{center}
 \begin{tabular}{|ccccccccc|}\hline
   &DOF&$e_{101}$&$e_{102}$&$e_{103}$&$e_{104}$&
               $e_{105}$&$e_{106}$&$e_{107}$\\\hline   
   CR    &110123&4.30e-03&4.29e-03&4.02e-03&4.51e-03&5.11e-03&5.27e-03&5.18e-03\\
   LR    &106638&2.57e-02&2.40e-02&2.36e-02&2.42e-02&2.70e-02&2.61e-02&2.30e-02\\
   MR1&104095&4.81e-03&4.70e-03&4.37e-03&4.99e-03&5.67e-03&5.76e-03&5.77e-03\\
   MR2&115258&3.81e-03&3.77e-03&3.57e-03&4.03e-03&4.67e-03&4.67e-03&4.76e-03\\   
   MR3&104432&1.96e-03&2.07e-03&1.92e-03&2.06e-03&3.12e-03&2.42e-03&2.59e-03\\
   \hline
  \end{tabular}  
\end{center}
\end{small}
\end{table}
\end{example}

%%%%%%%%%%%%%%%%%%%%%%%%%%%%%%%%%%%%%%%%%%%%%%%%%%%%%%%%%%%%%%%%
%%%%%%%%%%%%%%%%%%%%%%%%%%%%%%%%%%%%%%%%%%%%%%%%%%%%%%%%%%%%%%%%
\section{Discontinuous Galerkin discretization}
\label{sec:DGmethods}

We will be using the $hp$-version SIPG finite element
method~\cite{arnold_unified_2002,giani_posteriori_2012,giani_posteriori_2018}
for the discretization of \eqref{Problem} in
Section~\ref{Experiments}, so we describe it in this section, together
with error estimators and the local error indicators that will be used
to drive our adaptive algorithms.
Throughout, we assume that the computational domain~$\Omega$ can be
partitioned into a shape-regular mesh $ \mathcal{T}$, and
that the mesh elements are affine quadrilaterals.  The diameter of an
element $K$ is denoted by $h_K$.  In order to be able to deal with
irregular meshes we need to define the edges of a mesh ${\mathcal
  T}$. We refer to~$e$ as an interior mesh edge of~${\mathcal T}$ if
$e=\partial K\cap\partial K^\prime$ for two neighboring elements
$K, K^\prime\in{\mathcal T}$ whose intersection has a positive surface
measure. The set of all interior mesh edges is denoted by
$\mathcal{E}_I(\mathcal{T})$. Analogously, if the intersection
$e=\partial K \cap \partial\Omega$ of the boundary of an
element~$K\in{\mathcal T}$, we refer to~$e$ as a boundary mesh edge of
${\mathcal T}$. The set of all boundary mesh edges of~${\mathcal T}$
is denoted by~$\mathcal{E}_\Gamma(\mathcal{T})$ and we set
$\mathcal{E}(\mathcal{T})=\mathcal{E}_I(\mathcal{T})\cup
\mathcal{E}_{\Gamma}(\mathcal{T})$. The diameter of an edge $e$ is
denoted by~$h_e$.  We allow for 1-irregularly refined meshes
$\mathcal{T}$, i.e. we allow for at most one hanging node per edge.
We also assume that the discontinuities of $A$ and $V$ are aligned
with the edges of the elements.

Next, let us define the jumps and averages of piecewise smooth
functions across edges of the mesh ${\mathcal T}$. To that end, let
the interior edge~$e\in{\mathcal E}_I({\mathcal T})$ be shared by two
neighboring elements $K_1$ and $K_2$. For a piecewise smooth
function~$v$, we denote by $v|_{1,e}$ the trace on $e$ taken from
inside $K_1$, and by $v|_{2,e}$ the one taken from inside $K_2$. The
jump of $v$ across the edge~$e$ is then defined as
\begin{equation*}
   \jmp{v}=v|_{1,e}\,\,
  \underline{n}_{K_1} + v|_{2,e}\,\, \underline{n}_{K_2}.
\end{equation*}
Here, $\underline{n}_{K_1}$ and $\underline{n}_{K_2}$ denote the unit outward normal
vectors on the boundary of elements $K_1$ and $K_2$, respectively.
The average
across~$e$ of $A\nabla v$ is defined as in \cite{ern_discontinuous_2009} for the scalar case:
\begin{equation*}
  \mvl{A\nabla v}=\omega_2(A\nabla v)|_{2,e}+\omega_1(A\nabla v)|_{1,e},
\end{equation*}
where
$$
\omega_2 =
\frac{\underline{n}_{K_2}^tA|_{1,e}\underline{n}_{K_2}}{\underline{n}_{K_2}^tA|_{2,e}\underline{n}_{K_2}+\underline{n}_{K_2}^tA|_{1,e}\underline{n}_{K_2}}\
, \quad \omega_1 =
\frac{\underline{n}_{K_1}^tA|_{2,e}\underline{n}_{K_1}}{\underline{n}_{K_1}^tA|_{2,e}\underline{n}_{K_1}+\underline{n}_{K_1}^tA|_{1,e}\underline{n}_{K_1}}\
.
$$
On a boundary edge $e\in{\mathcal E}_\Gamma({\mathcal T})$, we
accordingly set $\mvl{A\nabla v}= A\nabla v$ and
$\jmp{v}=v\underline{n}$, with $\underline{n}$ denoting the unit
outward normal vector on~$\Gamma$.

%%%%%%%%%%%%%%%%%%%%%%%%%%%

On each element $K$, we construct a polynomial space of order $p$
like:
\begin{equation}
  \label{eq:Qp} \mathcal{Q}_{p}(K)= \{\,v\, :\, K\to \mathbb{R} \,:\,
  v\circ T_K \in \mathcal{Q}_{p}(\widehat{K})\,\},
\end{equation}
with $\mathcal{Q}_{p}(\widehat{K})$ denoting the set of tensor product
polynomials on the reference element $\widehat{K}$ of degree less than
or equal to $p$ in each coordinate direction on $\widehat{K}$.  In the
$hp$- setting, elements in the same mesh might be associated to
polynomial spaces of different orders.  We assume the order of
polynomials to be of bounded local variation, that is, there is a
constant $\varrho\geq 1$, independent of the mesh $\mathcal{T}$, such
that
\begin{equation}
  \label{eq:bounded-variation-p} \varrho^{-1}\leq p_K/p_{K^\prime}\leq
  \varrho
\end{equation}
where $p_K$ and $p_{K^\prime}$ are the orders of the elements $K$ and
$K^\prime$, for any pair of neighboring elements
$K, K^\prime\in \mathcal{T}$. For a mesh edge
$F\in\mathcal{E}(\mathcal{T})$, we introduce the edge polynomial
degree $p_e$ by
\begin{equation}
  \label{eq:face-polynomial}
  \begin{aligned}
    p_e&=\begin{cases}\ \max\{ p_{K_1}, p_{K_2}\},& \qquad \text{if
      }e=\partial K_1 \cap \partial K_2\in{\mathcal E}_I({\mathcal
        T}),\\[0.1cm]
      \ p_K, &\qquad \text{if } e=\partial K\cap \partial\Omega\in
      \mathcal{E}_B(\mathcal{T}).
    \end{cases}
  \end{aligned}
\end{equation}
We also denote by $h_e$ the length of the edge $e$.
For a partition $\mathcal{T}$ of $\Omega$, we define the $hp$-version
DG finite element space by
\begin{equation}
  \label{eq:hp-space} S(\mathcal{T}) =\{\,v\in
  L^2(\Omega) \, : \, v|_K \in \mathcal{Q}_{p_K}(K), \ K\in{\mathcal
    T}\,\}.
\end{equation}

%%%%%%%%%%%%%%%%%%%%%%%%%%%

The standard $L^2$ norm on the domain $\Omega$ is denoted by
$\|\cdot\|_{0,\Omega}$ and when restricted to an element or edge is
denoted by $\|\cdot\|_{0,K}$ and $\|\cdot\|_{0,e}$, respectively.
We shall also need the following DG norm:
\begin{definition}[DG norm]\label{def:dg_norm}
For any $u\in S(\mathcal{T})$
\begin{equation}
  \label{eq:energynorm} \enorm{u}{S(\mathcal{T})}^2 =\sum_{K\in\mathcal
    T}\|A\nabla u-Vu\|^2_{0,K}+\sum_{e\in\mathcal{E}({\mathcal
      T})}\frac{\gamma  cp_e^2}{h_e} \|\jmp{u}\|^2_{0,e},
\end{equation}
where 
$$
c =
2\frac{\underline{n}_{K_1}^tA|_{1,e}A|_{2,e}\underline{n}_{K_1}}{\underline{n}_{K_1}^tA|_{1,e}\underline{n}_{K_1}+\underline{n}_{K_1}^tA|_{2,e}\underline{n}_{K_1}},
$$
with $A|_{1,e}$ and $A|_{2,e}$ the values of $A$ in the two elements
$K_1$ and $K_2$ sharing $e$. If $e$ is on the boundary,
$c=\underline{n}^tA\underline{n}$.  The penalty parameter $\gamma>0$
appears in the SIPG discretization described below.
\end{definition}

%%%%%%%%%%%%%%%%%%%%%%%%%%%

The SIPG discrete version of the eigenvalue problem~\eqref{Problem}
is: find $(\lambda_{h},\phi_{h})\in \mathbb{R}\times S( \mathcal{T})$
such that
\begin{equation}
  \label{eq:dgfem_eig} a(\phi_{h}, v_{h})=\lambda_{h}\ b(\phi_{h},v_{h})\qquad
  \forall\, v_{h}\in S( \mathcal{T}),
\end{equation}
where $b(\cdot,\cdot)$ is the $L^2$ inner product and with
$\|\phi_{h}\|_{0,\Omega}=1$.  The bilinear form $a(u,v)$ is given by
\begin{equation}
\label{eq:def_a}
  \begin{split}
    a(u, v)&=\sum_{K\in{\mathcal T}}\int_K\,A\nabla u\cdot \nabla
    v+Vuv\,dx\\[1mm]
    &-\sum_{e\in\mathcal{E}({\mathcal
        T})}\int_e\,\Big(\mvl{A\nabla u}\cdot\jmp{v}+\mvl{A\nabla
      v}\cdot\jmp{u}\Big)\,ds\\[0.1cm]&+\sum_{e\in{\mathcal
        E}(\mathcal{T})} \frac{\gamma c p_e^2}{h_e} \int_e\,
    \jmp{u}\cdot\jmp{v}\,ds,
  \end{split}
\end{equation}
where the gradient operator $\nabla$ is defined elementwise and the
parameter $\gamma>0$ is the interior penalty parameter. For source
problems the SIPG is known to be a stable and consistent method for
sufficiently large values of $\gamma$
\cite{arnold_unified_2002,prudhomme_review_2000}.

The SIPG discrete version of the landscape source problem~\eqref{Landscape} is: find
$u_h\in \mathbb{R}\times S( \mathcal{T})$ such that
\begin{equation}
  \label{eq:dgfem} a(u_{h}, v_{h})=b(1,v_{h})\qquad
  \forall\, v_{h}\in S( \mathcal{T}).
\end{equation}
%%%%%%%%%%%%%%%%%%%%%%%%%%%%%%%%%%%%%%%%%%%%%%%%%%%%%%%%%%%%%%%%

Two residual-based error estimators are used in the simulations, one
for the landscape solution $u$ \eqref{Landscape} which is based on the
error estimator for linear problems, and one for the eigenpairs of the
eigenvalue problem \eqref{Problem}. An a posteriori error estimator is
a necessary ingredient for adapting the mesh and the finite element
space because it is capable of estimating the error in each element of
the computed solution using only the computed solution and given data for
the problem.

The error estimator for the landscape function $\eta_{\mathrm{land}}$
is based on \cite{houston_energy_2007,giani_reliable_2018} and is
computed using an approximation $u_h$ of $u$ computed on
$S(\mathcal{T})$:
\begin{equation}\label{Landscape_err_est}
  \eta_{\mathrm{land}}^2 := \sum_{K \in \mathcal{T} }
  \eta_{\mathrm{land},K}^2\; ,\;
  \eta_{\mathrm{land},K}^2:= \eta_{R_{\mathrm{land},K}}^2(u_h)+\eta_{F_K}^2(u_h)+\eta_{J_K}^2(u_h)
\end{equation}
where
\begin{align*}
  \eta_{R_{\mathrm{land},K}}^2(u_h)\ :=&\ A_{\mathrm{min},
  K}^{-1}p_K^{-2}h_K^2\|1+\nabla\cdot(A\nabla u_h)-V u_h\|^2_{0,K}\
                                         ,\\
 \eta_{F_K}^2(u_h) :=&
\frac{1}{2}\sum_{e\in\mathcal{E}_{I}(\mathcal{T})} A_{\mathrm{min},
  e}^{-1}\,p_e^{-1}h_e\|\jmp{A\nabla u_h}\|^2_{0,e}\ ,\\ 
  \eta_{J_K}^2(u_h)\ := &\frac{1}{2}\sum_{e\in\mathcal{E}_{I}(\mathcal{T})}
                           \Bigg(\frac{h_e}{A_{\mathrm{min}, e}\,p_e}+\frac{\gamma^2 A_{\mathrm{max}, e}\,p_e^3}{h_e}\Bigg)\|\jmp{u_h}\|^2_{0,e}\\
                         &+\sum_{e\in\mathcal{E}_{\Gamma}(\mathcal{T})}
                           \Bigg(\frac{h_e}{A_{\mathrm{min}, e}\,p_e}+\frac{\gamma^2 A_{\mathrm{max}, e}\,p_e^3}{h_e}\Bigg)\|u_h\|^2_{0,e}\ ,
\end{align*}
where $A_{\mathrm{min}, e}$ and $A_{\mathrm{max}, e}$ are respectively
the minimum and the maximum eigenvalue of $A$ among the elements
sharing the edge $e$ and where $A_{\mathrm{min}, K}$ is the minimum
eigenvalue of $A$ in $K$.  When $A\equiv 1$ everywhere, the term
$\frac{h_e}{p_e}\lesssim \frac{p_e^3}{h_e}$ in
$\eta_{J_K}^2(u_h)$. Therefore, the term $\eta_{J_K}^2(u_h)$ reduces
to:
\begin{align*}
\eta_{J_K}^2(u_h)\ :=\ &\frac{1}{2}\sum_{e\in\mathcal{E}_{I}(\mathcal{T})}
\frac{\gamma^2 p_e^3}{h_e}\|\jmp{u_h}\|^2_{0,e}
+\sum_{e\in\mathcal{E}_{\Gamma}(\mathcal{T})}
\frac{\gamma^2 p_e^3}{h_e}\|u_h\|^2_{0,e}\ .
\end{align*}

From \cite{houston_energy_2007} we have that $\eta_{\mathrm{land}}$ is
reliable:
\begin{align}\label{reliability_land}
\|u-u_h\|_{S(\mathcal{T})} \lesssim \eta_{\mathrm{land}}\ .
\end{align}

The error estimator $\eta_{\mathrm{eig},j}$ for the eigenpairs of the
eigenvalue problem \eqref{Problem} is based on
\cite{Giani2016a,giani_posteriori_2012} and computed for each
eigenpair $(\lambda_j,\phi_j)$ using an approximation
$(\lambda_{j,h},\phi_{j,h})$ of the eigenpair computed on
$S(\mathcal{T})$:
\begin{equation}\label{Landscape_err_est_eig}
\eta_{\mathrm{eig},j}^2 := \sum_{K \in \mathcal{T} }
\eta_{\mathrm{eig},K,j}^2\; ,\;
\eta_{\mathrm{eig},K,j}^2  = \eta_{R_{\mathrm{eig},K}}^2(\lambda_{j,h},\phi_{j,h})+\eta_{F_K}^2(\phi_{j,h})+\eta_{J_K}^2(\phi_{j,h})
\end{equation}
where
$$
\eta_{R_{\mathrm{eig},K}}^2(\lambda_{j,h},\phi_{j,h})\ :=\
A_{\mathrm{min},
  K}^{-1}p_K^{-2}h_K^2\|\lambda_{j,h}\phi_{j,h}+\nabla\cdot(A\nabla
\phi_{j,h})-V \phi_{j,h}\|^2_{0,K}\ .
$$
The ``jump terms'' $\eta_{F_K}$ and $\eta_{J_K}$ here are defined in
the same way they were for the landscape function.
From \cite{giani_posteriori_2012} we have that $\eta_{\mathrm{eig},j}$
is reliable for both eigenvalues and eigenvectors:
\begin{align}\label{reliability_eig}
 |\lambda_j-\lambda_{j,h}|\lesssim \eta_{\mathrm{eig},j}^2 + \mathrm{h.o.t.}\ ,\\
 \mathrm{dist}(\phi_{j,h},E_1(\lambda_j))\lesssim \eta_{\mathrm{eig},j} + \mathrm{h.o.t.}\ ,
\end{align}
where $\mathrm{h.o.t.}$ are asymptotically higher-order terms that are
not numerically computed or approximated and where $E_1(\lambda_j)$ is
the span of all eigenfunctions of the eigenvalue $\lambda_j$ normalized
in the $L^2$ norm and where $\mathrm{dist}(\phi_{j,h},E_1(\lambda_j))$
is the distance in the DG norm of the computed eigenvector
$\phi_{j,h}$ from the eigenspace $E_1(\lambda_j)$.

%%%%%%%%%%%%%%%%%%%%%%%%%%%%%%%%%%%%%%%%%%%%%%%%%%%%%%%%%%%%%%%%%%%%%%%%%%%%%%
%%%%%%%%%%%%%%%%%%%%%%%%%%%%%%%%%%%%%%%%%%%%%%%%%%%%%%%%%%%%%%%%%%%%%%%%%%%%%%
\section{Numerical Results}\label{Experiments}

We will compare landscape refinement with versions of
eigenvector-based refinement on several examples having different features.
We first \change{define all the quantities} used in
presenting the results that are collected, starting from the quantities
related to the eigenvalue problem:
\begin{itemize}
\item \textbf{Error bound for single eigenpair:} for the computed
  eigenpair $(\lambda_{j,h},\phi_{j,h})$ of index $j$, the error bound
  is $\eta_{\mathrm{eig},j}^2$.
\item \textbf{Relative error bound for single eigenpair:} for the
  computed eigenpair $(\lambda_{j,h},\phi_{j,h})$ of index $j$, the
  error bound is $\eta_{\mathrm{eig},j}^2/\lambda_{j,h}$.
\item \textbf{Envelope for the true relative error for eigenvalues:}
  This quantity is only computable if the true values of the
  eigenvalues are available. For an approximation of the lowest $M$
  eigenpairs in the spectrum, it is defined as
  $e_\mathrm{max}^\mathrm{rel}= \max_{j\leq
    M}(|\lambda_j-\lambda_{j,h}|/\lambda_j)$.
\item \textbf{Envelope for the relative error bound for eigenvalues:}
  for an approximation of the lowest $M$ eigenpairs in the spectrum,
  it is defined as
  $\eta_{\mathrm{max},\mathrm{rel}}^2= \max_{j\leq
    M}(\eta_{\mathrm{eig},j}^2/\lambda_{j,h})$.
\end{itemize}
The only quantity for the landscape problem is the \textbf{error bound
  for the landscape solution} defined as $\eta_{\mathrm{land}}$.

The last quantity to introduce is CPU time which is the sum of the
time spent by all processing units on a task. For example, if two
cores of a modern CPU are working at the same time on a task for 1
minute then the CPU time for the task is 2 minutes. CPU time is a
better measure of the computational complexity of a task than real
time because it is independent of the number of cores used. Using CPU
time different codes with different levels of parallelization can be
compared in a meaningful way.

Although we will primarily compare landscape refinement with
cluster refinement, there will be a few cases in which we also
consider refinement driven by a single eigenpair.  We will refer to
this as eigenpair refinement.  In order to present the experiments
sooner, details of the particular algorithms
will be provided in Section~\ref{sec:alg}.  Here, we merely identify
the highest-level algorithms with the descriptors above:
\begin{enumerate}
\item Eigenpair refinement (ER) for a single eigenpair, Algorithm 1
\item Cluster refinement (CR) based on the first $M$ eigenpairs,
  Algorithm 2
\item Landscape refinement (LR), Algorithm 3 
\end{enumerate}
Algorithms 2 and 3 use the largest of the eigenvalue error
estimates for the cluster to determine the stopping criterion.  The
key difference between
Algorithms 2 and 3 is the strategy for marking elements and deciding
how to refine them, with Algorithm 3 using the landscape function for
these decisions and Algorithm 2 using some collective information from
the cluster.

\begin{example}[Laplace operator on the unit square]\label{UnitSquare}
We consider the operator $\cL w=-\Delta w $ on $\Omega=(0,1)^2$, for
which the pairs are well-known, $\lambda_{mn}=(m\pi)^2+(n\pi)^2$ and
$\phi_{mn}=2\sin(m\pi x)\sin(n\pi y)$ for $m,n\in\NN$.  The
eigenvectors are analytic and do not localize anywhere, but we will
see that landscape refinement is effective even in such simple situations.

In Figure~\ref{fig:square_envelopes}, we show the collective
eigenvalue error and error estimate, $e_\mathrm{max}^\mathrm{rel}$ and
$\eta_{\mathrm{max},\mathrm{rel}}^2$, for the first $M=100$
eigenpairs, under landscape refinement starting from an initial mesh
of $8\times 8$ elements and initial polynomial order $p=2$ in each element.  This figure
also includes relative error bounds for some individual eigenvalues in
the cluster.  For comparison, the error bound for the landscape
solution $\eta_{\mathrm{land}}^2$ is also shown.
The reliability of the underlying eigenvector error estimator \eqref{reliability_eig} is
reflected in the very similar convergence patters of 
 $e_\mathrm{max}^\mathrm{rel}$ and $\eta_{\mathrm{max},\mathrm{rel}}^2$.  The square of the error bound for the
landscape solution $\eta_{\mathrm{land}}^2$ converges faster and this
is because the curves for envelopes are dominated by the error from
the eigenpairs in the higher part of the computed spectrum. Such
eigenpairs have high frequency and therefore they are harder to
approximate in FEMs. On the other hand, the landscape solution has a
very low frequency so it converges faster. It is remarkable that
adapting for the landscape solution, which has low frequency, causes
high frequency eigenpairs to converge.
\begin{figure}
\centering
\includegraphics[width=0.48\textwidth]{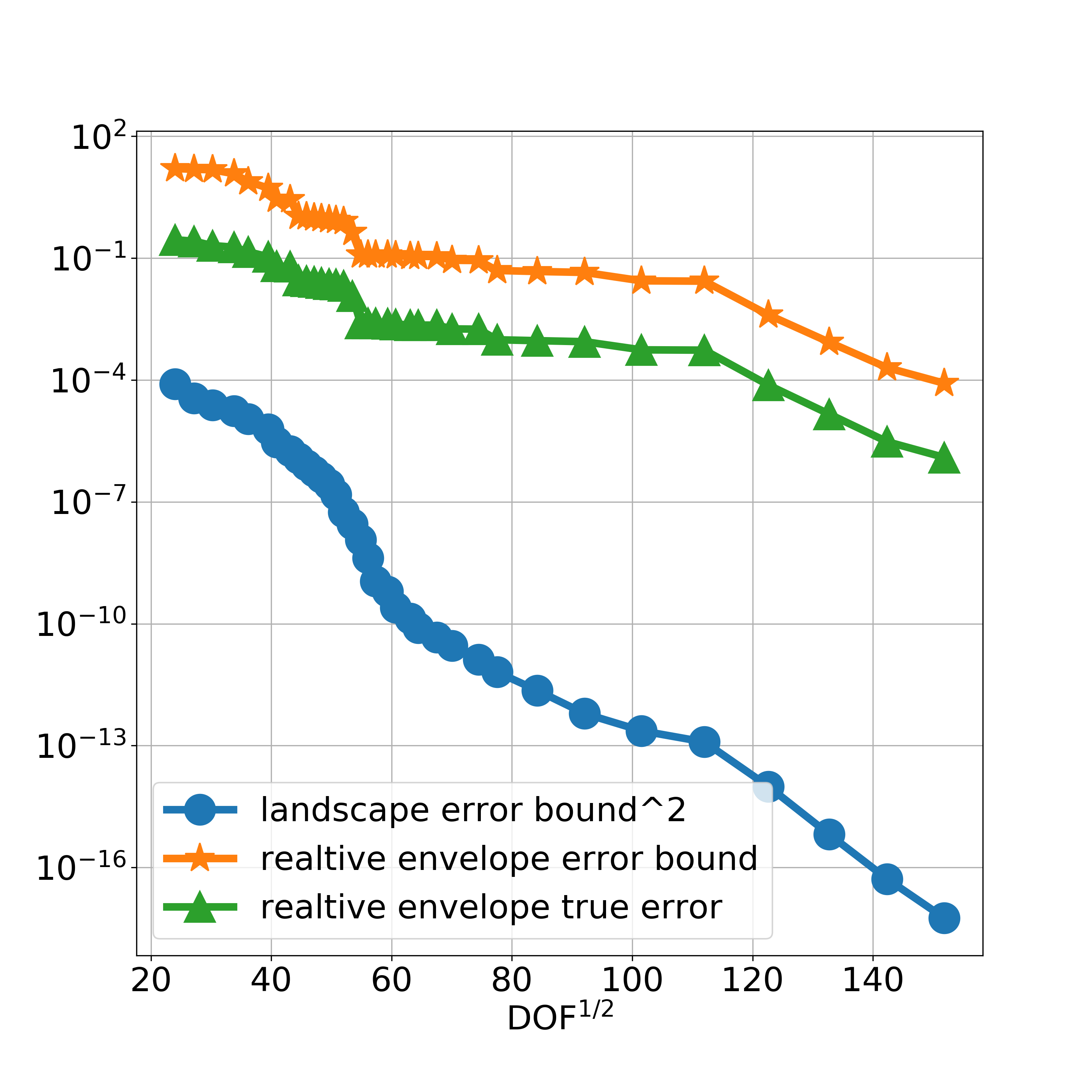}
\includegraphics[width=0.48\textwidth]{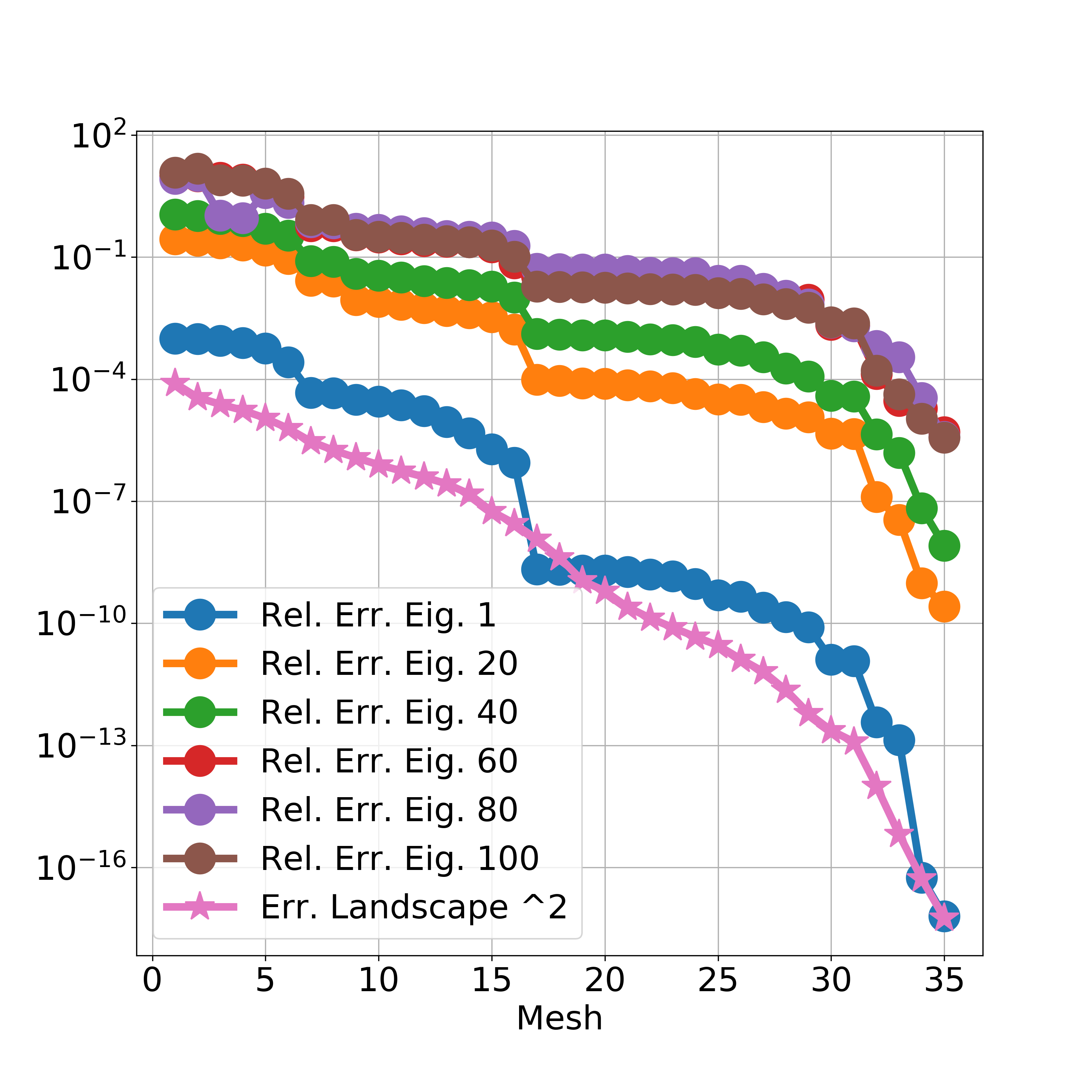}
\caption{Example~\ref{UnitSquare} Convergence plots of collective and
  individual eigenvalue errors under landscape refinement.  Left: 
  The square of the error bound for the landscape solution $\eta_{\mathrm{land}}^2$ and the envelopes for
  the relative error bound $\eta_{\mathrm{max},\mathrm{rel}}^2$ and
  the true error for the computed spectrum
  $e_\mathrm{max}^\mathrm{rel}$. Right: Relative errors for some individual eigenvalues.}\label{fig:square_envelopes}
\end{figure}

Figure~\ref{fig:square_compare} contains convergence comparisons between landscape
refinement, eigenvector refinement based on single eigenpairs, and two
variants of cluster refinement for the first $M=100$ eigenpairs.  The
two variants involve how local error and regularity indicators for
each eigenpair in the cluster are used for purposes of adaptivity.
The two approaches,  AdaptSpectrumEigSum and AdaptSpectrumEigMax, are
given in detail in Section~\ref{sec:alg} as Algorithm 7 and Algorithm
8, respectively, but the names themselves are suggestive of how they
operate.

The first picture in Figure~\ref{fig:square_compare} demonstrates that
landscape refinement can beat, or at least be competitive with,
eigenvector refinement for driving down the error in a single
eigenvalue.  We run eigenvector refinement independently three times,
targeting the eigenpairs of indices 1, 48 and 90.  These convergence
histories are plotted together with the convergence for these same
eigenvalues computed on meshes determined by landscape refinement.
For the first eigenvalue, eigenpair refinement leads to faster
convergence than landscape refinement, but both are extremely fast.
What is more interesting is that landscape refinement outperforms
eigenpair refinement for the other eigenvalues.  More specifically,
after an initial phase in which both approaches yield similar
convergence behavior, the landscape approach yields two significant
drops in estimated eigenvalue error between $40^2$ and $60^2$ DOFs
that mimic similar drops in the estimated (square) landscape error.
It may seem counterintuitive that landscape refinement can beat
eigenpair refinement for single eigenpairs, but perhaps it is not so
surprising in this case because, on coarse meshes, the approximate
eigenvectors having index 48 and 90 are not approximating the actual
48th and 90th eigenvectors well (if at all), so the refinement
dictated at those levels is based on poor information.  In contrast,
the landscape function is well-resolved even on coarse meshes, and
provides accurate information early on concerning where and how to
refine the mesh in a way that is useful even for the higher-frequency
eigenvectors (though it is a low-frequency function).
The second picture in Figure~\ref{fig:square_compare} shows that
landscape refinement convincingly beats both variants of cluster
refinement.  Based on what we observed in the single eigenvector
refinement cases, this is not a surprise, though it merits reporting.

\begin{figure}
\centering
\includegraphics[width=0.48\textwidth]{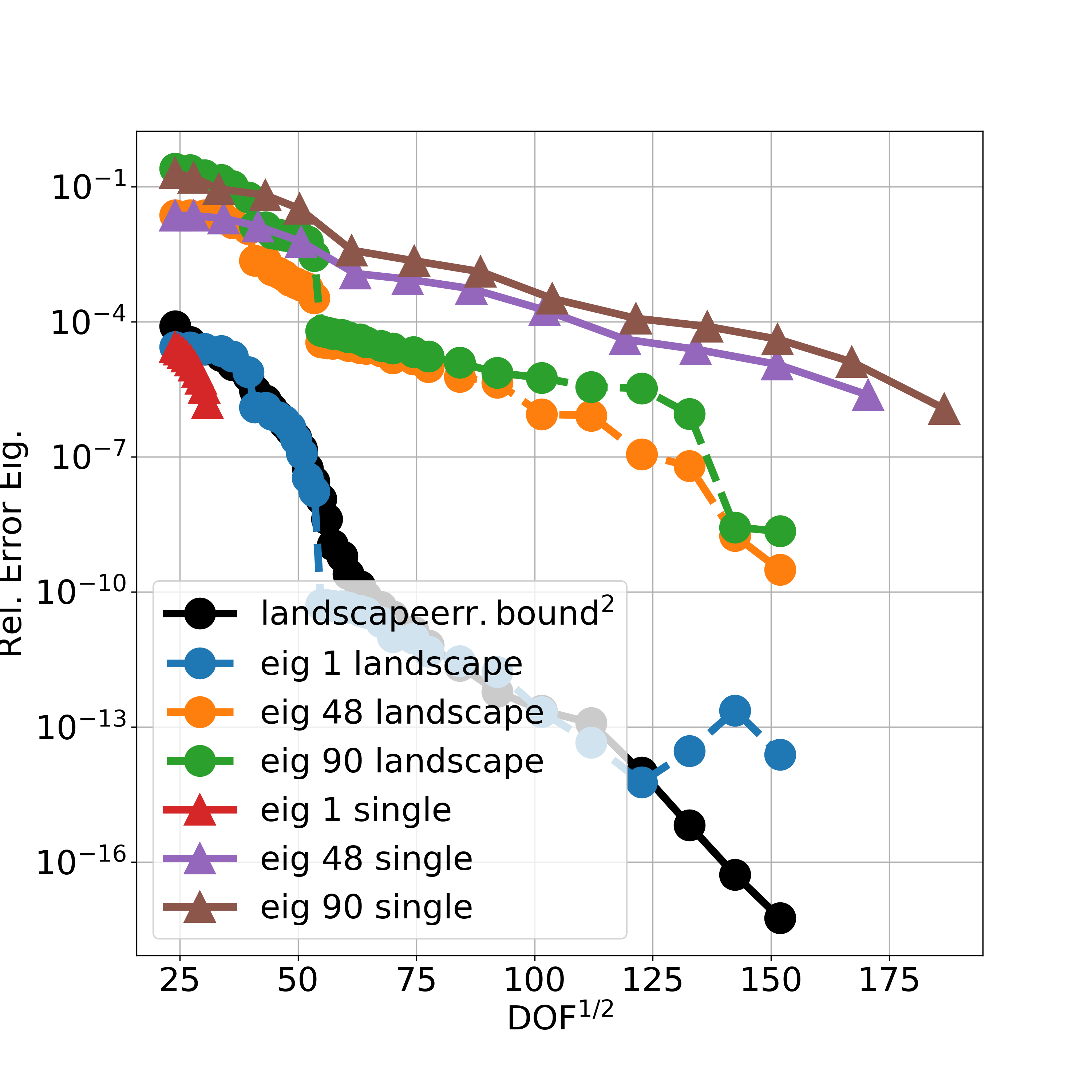}
\includegraphics[width=0.48\textwidth]{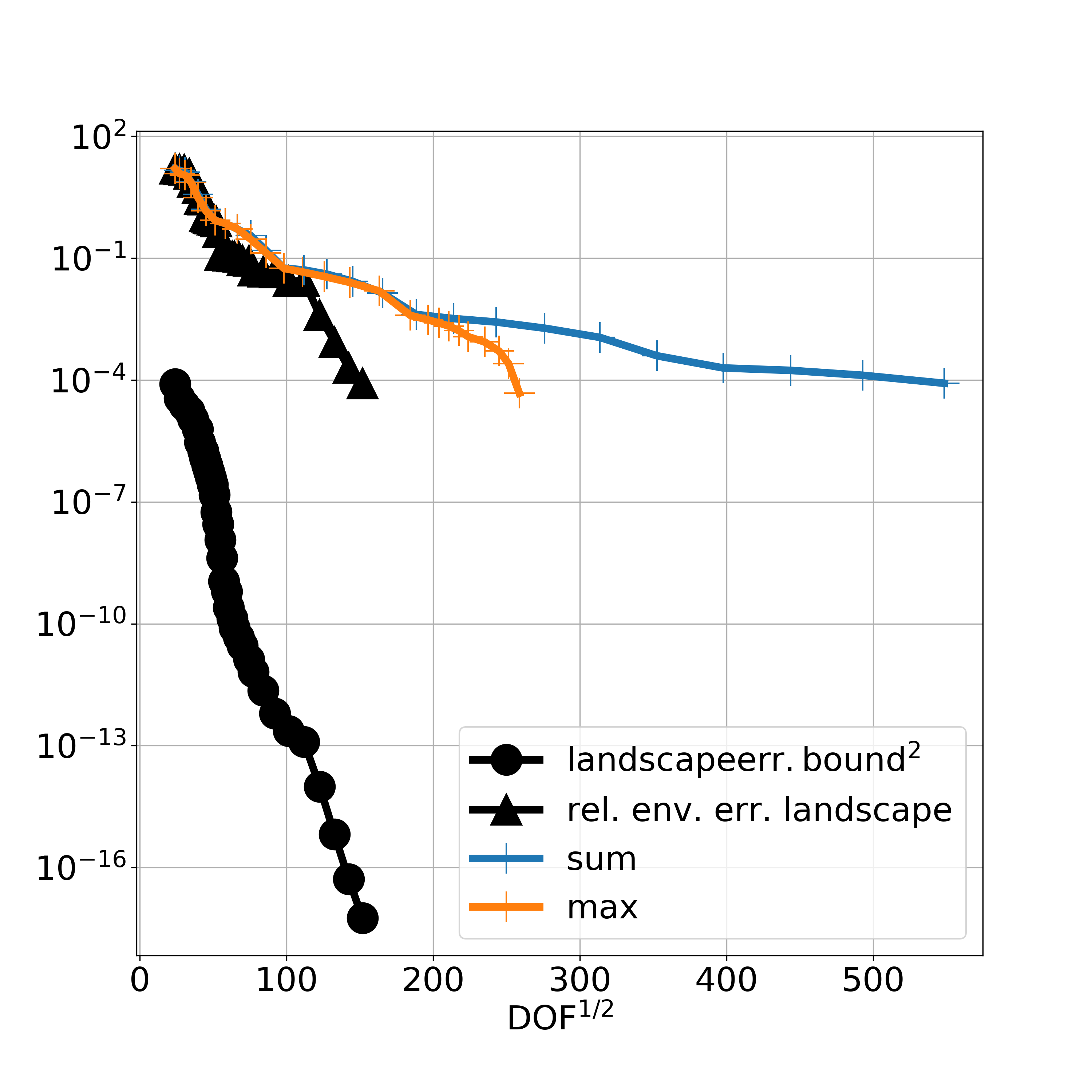}
\caption{Example~\ref{UnitSquare}. Left: Convergence plots of the relative
  eigenvalue errors $\eta_{\mathrm{eig},j}^2/\lambda_{j,h}$ under landscape refinement and single eigenpair
  refinement for $j=1,48,90$.  Right: Convergence plots of the
  relative error envelope bound $\eta_{\mathrm{max},\mathrm{rel}}^2$
  under landscape refinement \change{and two variants} of cluster refinement.
  Also included is the error bound for the landscape function
$\eta_{\mathrm{land}}^2$ under landscape refinement. }\label{fig:square_compare}
\end{figure}
\end{example}

%%%%%%%%%%%%%%%%%%%%%%%%%%%%%%%%%%
%%%%%%%%%%%%%%%%%%%%%%%%%%%%%%%%%%

\begin{example}[Laplace operator on the L-shape domain,
  revisited]\label{LShapeRevisited} We
  briefly revisit Example~\ref{hAdaptLShape}, but with $hp$ adaptivity
  using landscape refinement and cluster refinement for the first
  $M=100$ eigenpairs.  As before, we start with \change{an initial mesh} of
  $8\times 8$ elements of order $p=2$.
The convergence results for landscape refinement and cluster
refinement using AdaptSpectrumEigSum are shown in  
Figure~\ref{fig:lshape_spectrum_compare}, for the
envelope for the relative error bound
$\eta_{\mathrm{max},\mathrm{rel}}^2$.  The final meshes in each case
are given in Figure~\ref{fig:lshape_mesh}.  We see that landscape
refinement, though initially providing slightly slower convergence,
overtakes cluster refinement at around $140^2$ DOFs, and reaches an
error tolerance of $10^{-4}$ for the cluster much(!) sooner.  The
final mesh in the case of landscape refinement is also more appealing,
with $h$-refinement targeted near the corners of the domain, whereas
the final mesh for cluster refinement is far less coherent.
\begin{figure}
\centering
\includegraphics[width=0.8\textwidth]{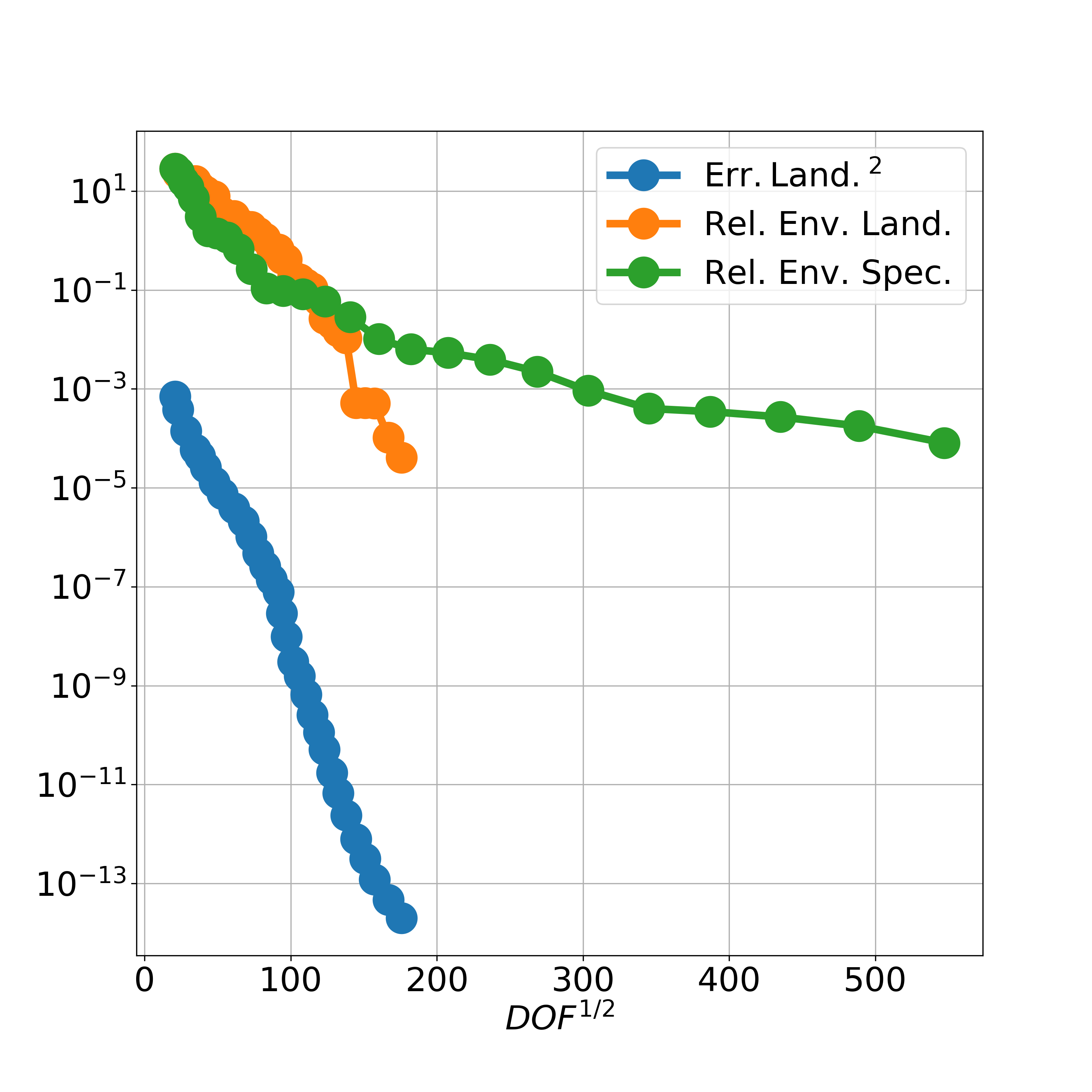}
\caption{\label{fig:lshape_spectrum_compare} Example~\ref{LShapeRevisited}. Convergence plots of the
  relative error envelope bound $\eta_{\mathrm{max},\mathrm{rel}}^2$
  under landscape refinement \change{and two variants} of cluster refinement.
   Also included is the error bound for the landscape function
$\eta_{\mathrm{land}}^2$ under landscape refinement.}
\end{figure}
\begin{figure}
\centering
\includegraphics[width=0.8\textwidth]{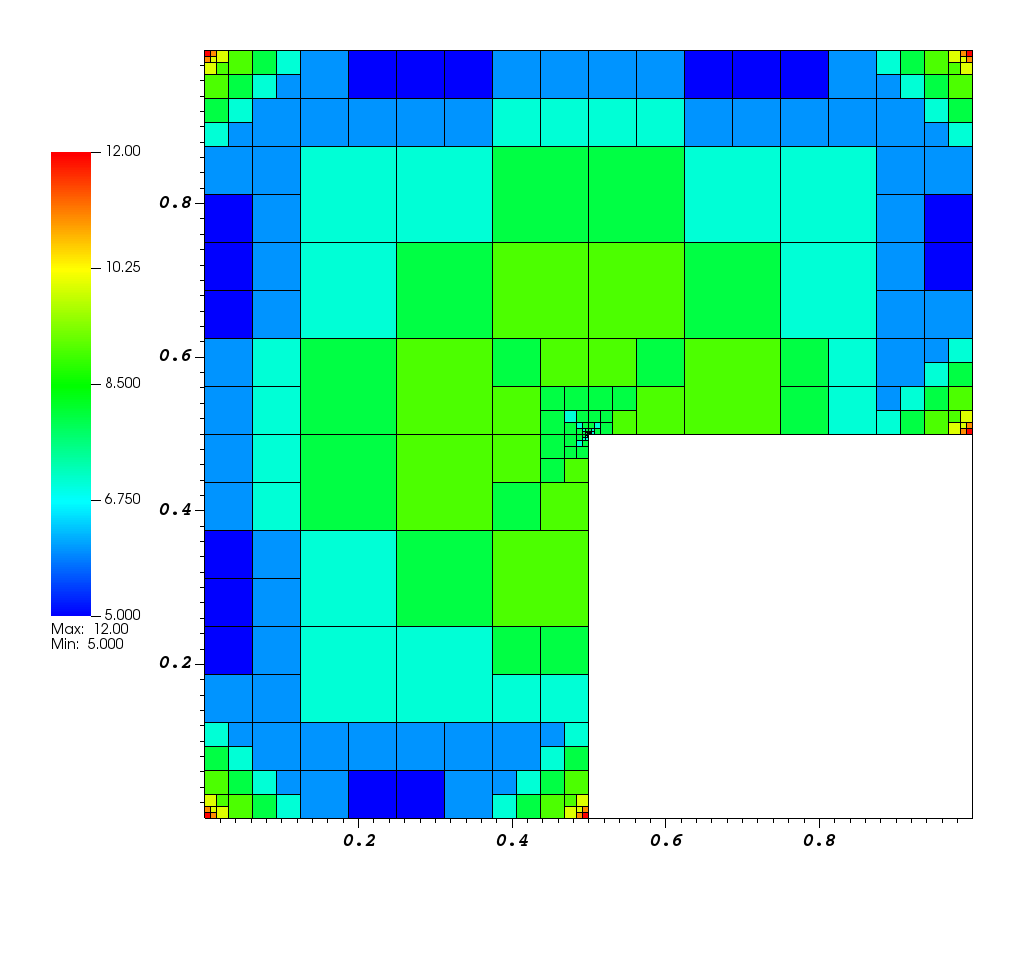}
\includegraphics[width=0.8\textwidth]{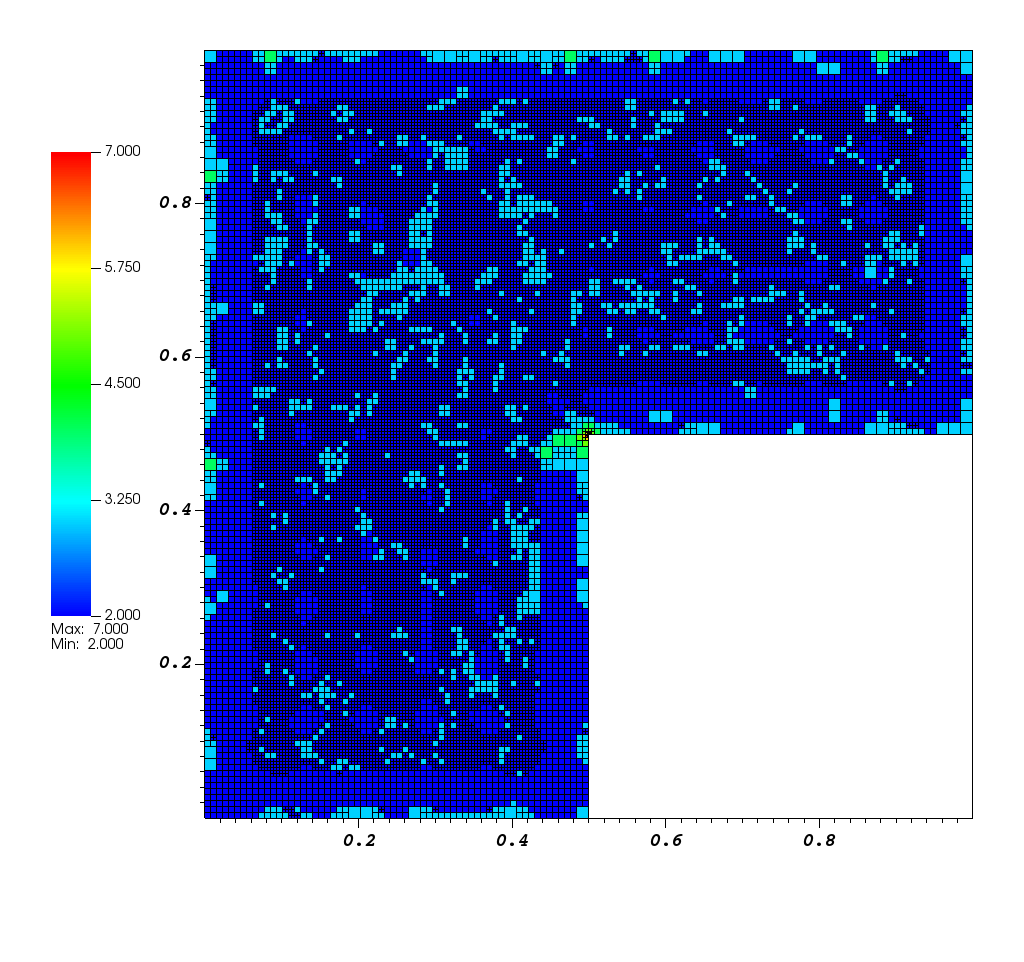}
\caption{\label{fig:lshape_mesh}
  Example~\ref{LShapeRevisited}. Final mesh computed using
  landscape refinement (top) and cluster refinement (bottom). Colors
  represent local polynomial degree.}
\end{figure}

We briefly comment on the high polynomial degrees seen in the
corners of the landscape-adapted mesh in Figure~\ref{fig:lshape_mesh},
which may seem counter-intuitive as the elements are small there.  A
wealth of empirical evidence indicates that such seeing high
polynomial degrees in regions with small elements is not at all
unusual for many $hp$-adaptive schemes,
and does not negatively impact the exponential convergence rate.  We
note the works~\cite{Canuto2017,Canuto2019} for their use of $h$
and/or $p$ coarsening during the adaptive process.  Under such a
scheme, we might not have obtained high polynomial degrees near the
corners, but the exponential convergence would not have been noticeably improved.
\end{example}

\begin{example}[Simple Schr\"odinger operator]\label{SimpleSchrodinger}
  Here we consider the Schr\"odinger operator $\cL w=-\Delta w + Vw$
  on the domain $\Omega=(0,1)^2$, where $V$ is the piecewise constant
  potential defined in Figure~\ref{fig:SimpleSchrodinger}, whose values range
  between 0 and 6400.  Again, we start with an initial mesh of
  $8\times 8$ elements and with initial order $p=2$, and consider
  landscape refinement, single eigenpair refinement, and cluster
  refinement for the first $M=100$ eigenpairs.  In the case of cluster
  refinement, we use AdaptSpectrumEigSum for marking elements and
  assessing local regularity.

In Figure~\ref{fig:SimpleSchrodinger}, we also see the convergence of
the relative error bound $\eta_{\mathrm{eig},j}^2/\lambda_{j,h}$ for
the $j$th (single) eigenpair under each of the three refinement
schemes.  Concerning the comparison between landscape refinement and
eigenpair refinement, we see similar behavior as we did in
Figure~\ref{fig:square_envelopes} for the Laplacian on this domain.
The convergence for $j=1$ was slightly better under eigenpair refinement than
for landscape refinement, reaching $10^{-4}$ first, and cluster
refinement was the worst.  Higher in the spectrum, $j=57$ and $j=76$, eigenpair
refinement and cluster refinement performed similarly, beating
landscape refinement until around $200^2$ DOF, when landscape
refinement began to significantly overtake them.
\begin{figure}
  \centering
  \includegraphics[width=0.48\textwidth]{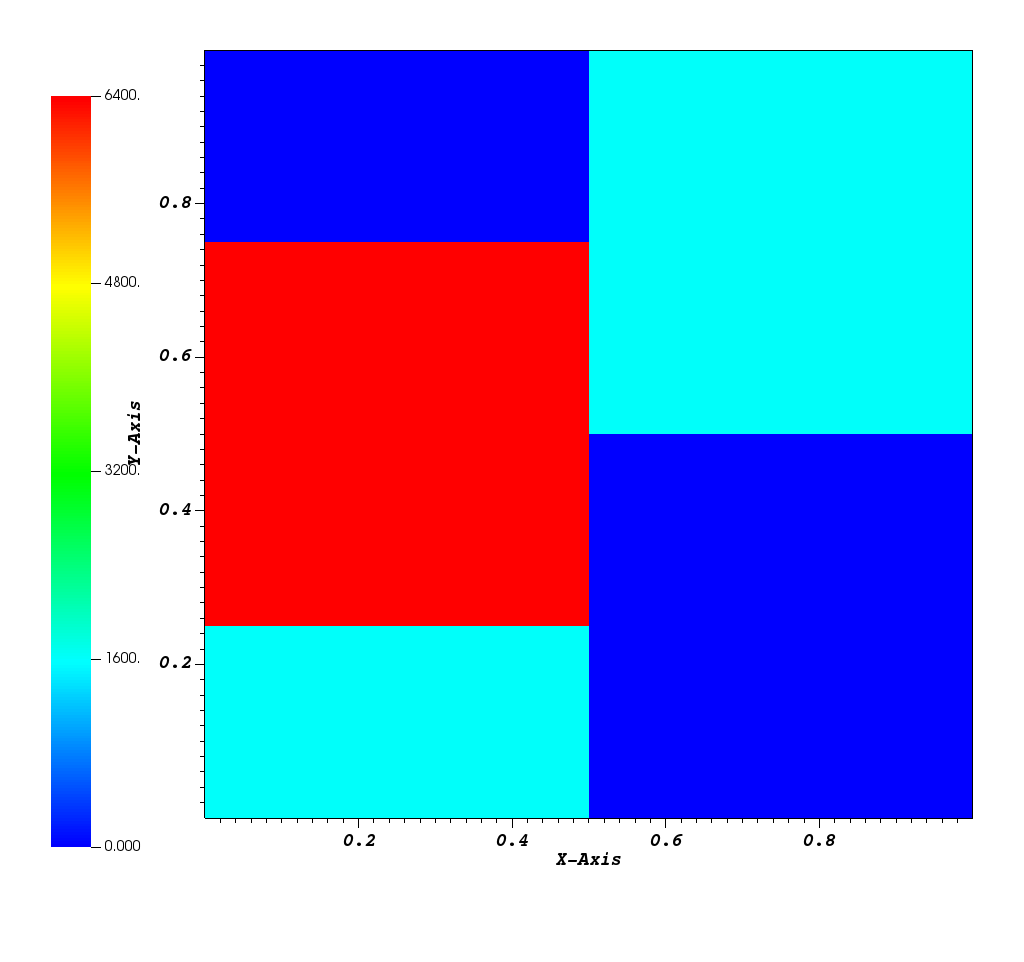}
\includegraphics[width=0.48\textwidth]{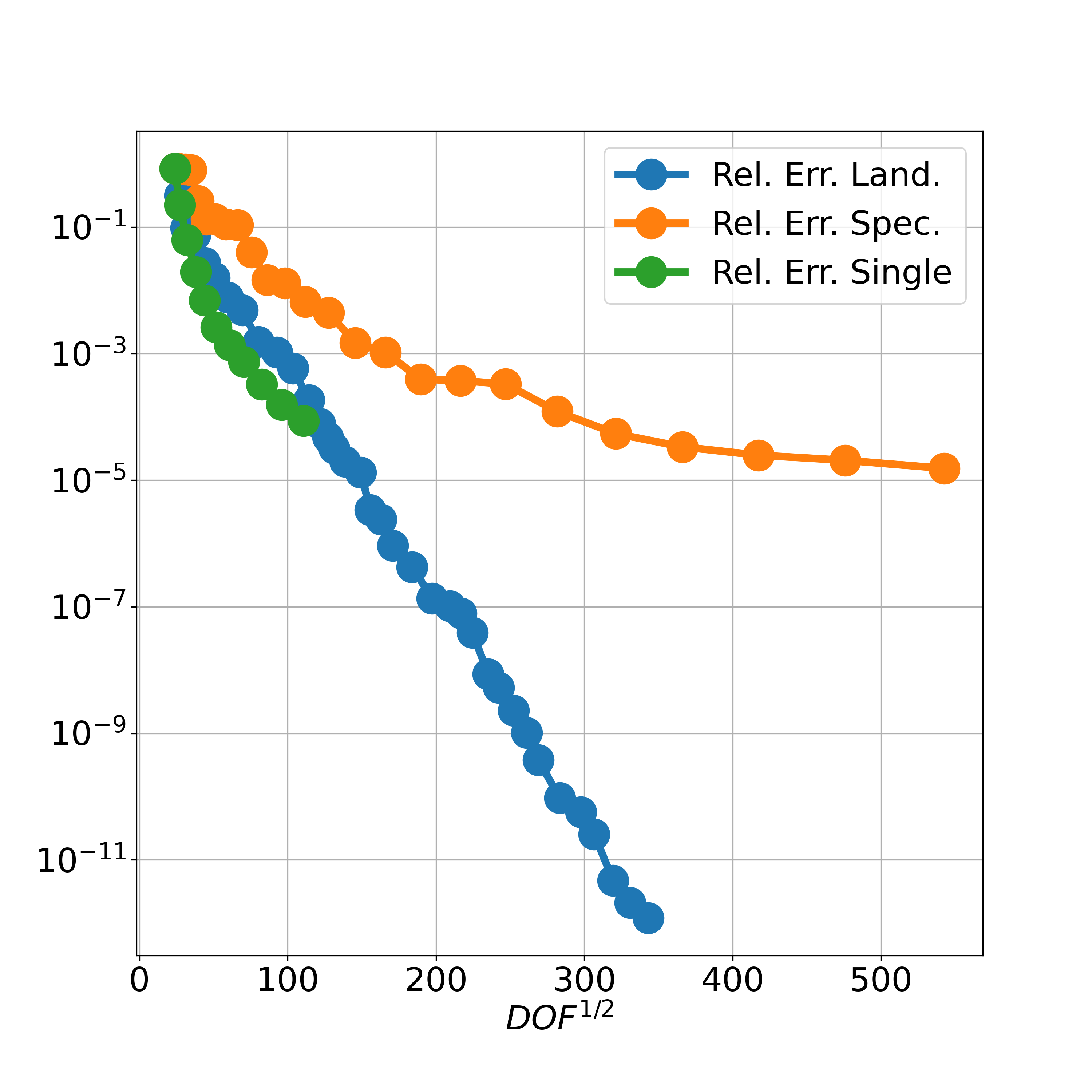}
\includegraphics[width=0.48\textwidth]{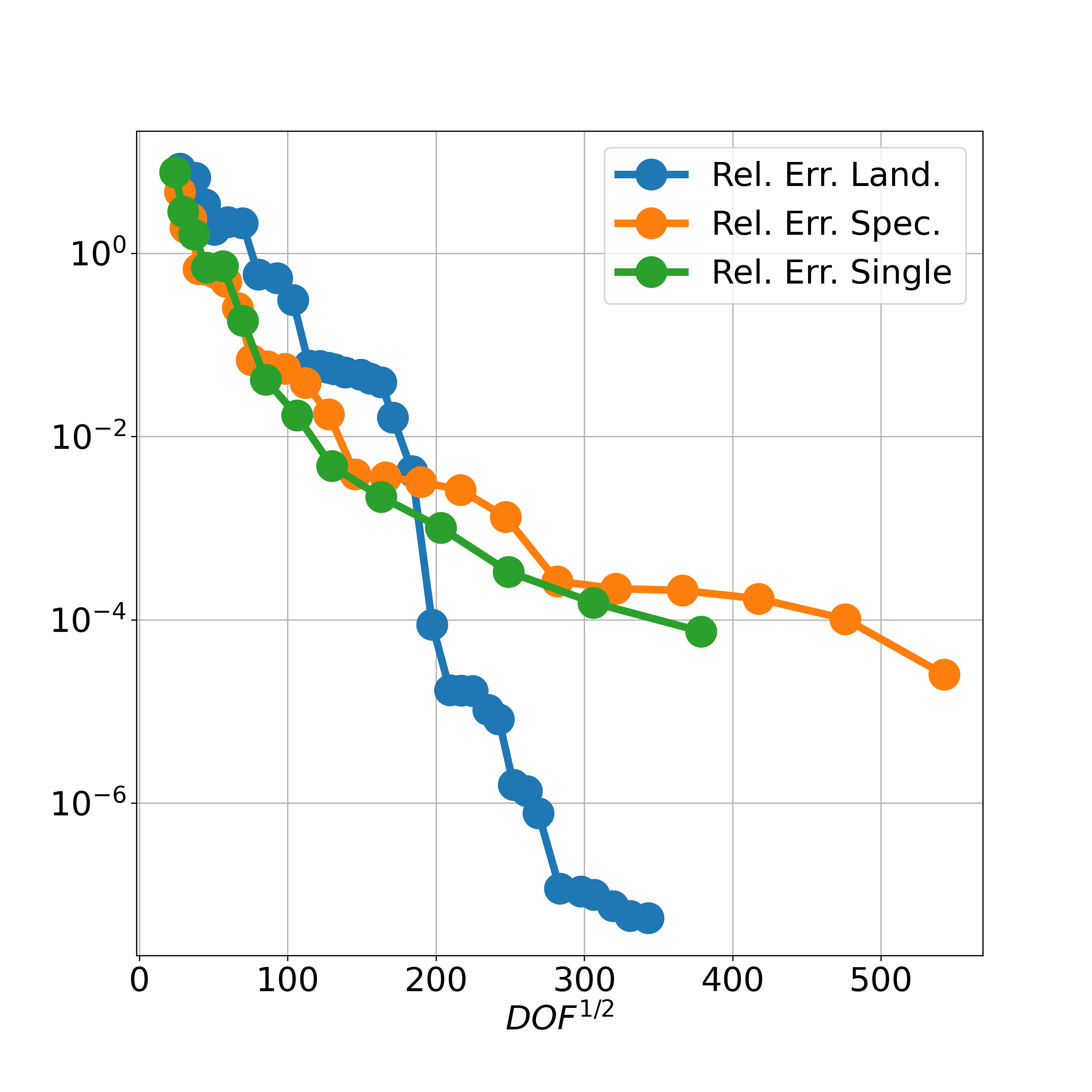}
\includegraphics[width=0.48\textwidth]{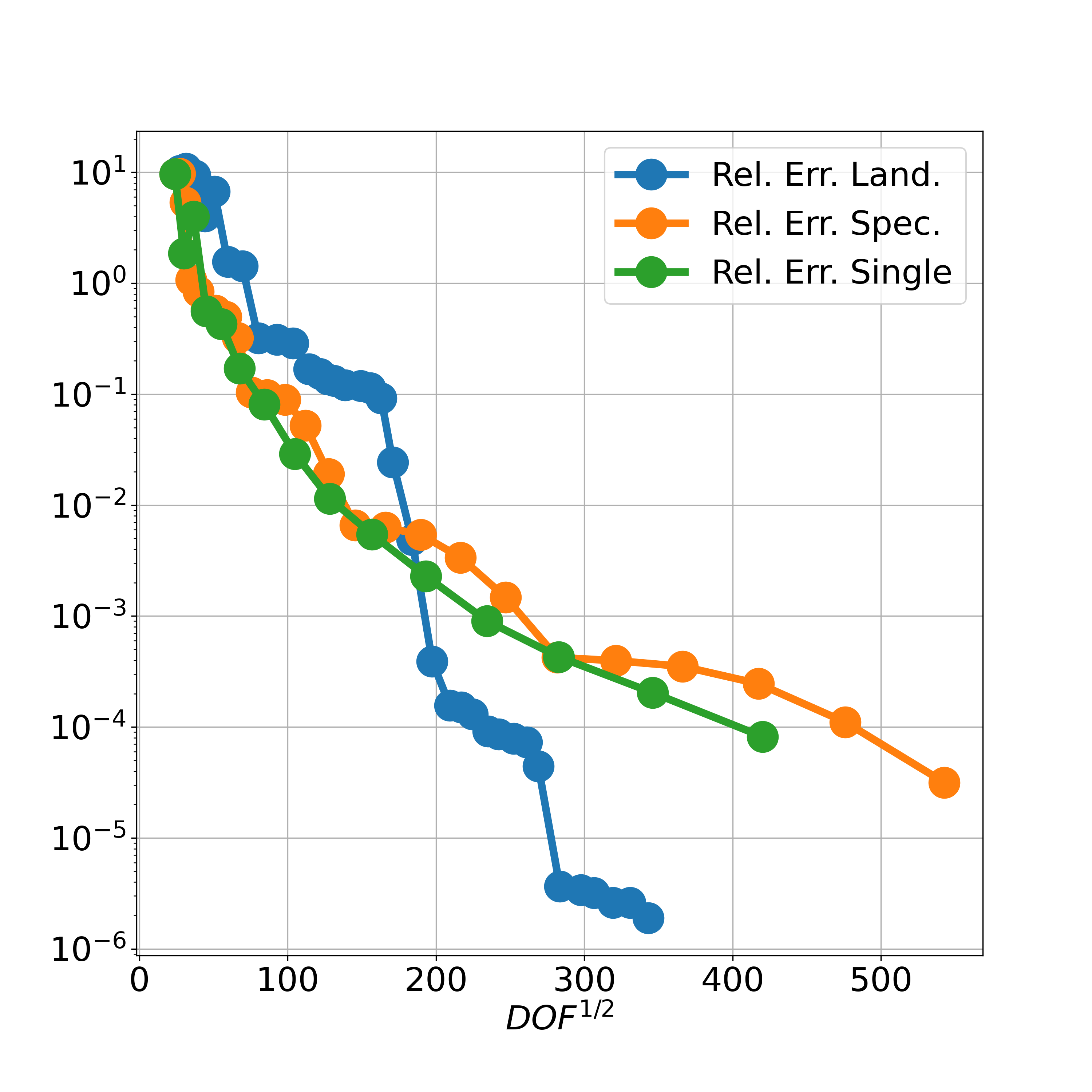}
\caption{Example~\ref{SimpleSchrodinger}.  Top left: The potential
  $V$.  Other plots: Convergence of the relative
  error bound $\eta_{\mathrm{eig},j}^2/\lambda_{j,h}$ for the $j$th eigenpair under landscape refinement,
  cluster refinement for the first $M=100$ eigenpairs, and single
  eigenvector refinement for the targeted eigenpair, $j=1,57,76$.  
 }\label{fig:SimpleSchrodinger}
\end{figure}
\end{example}

\begin{example}[Schr\"odinger Operator with Rough Potential]\label{HardSchrodinger}
We again take $\Omega=(0,1)^2$ and the Schr\"odinger operator $\cL
w=-\Delta w + Vw$, where $V$ is the piecewise constant potential
defined in Figure~\ref{fig:HardSchrodinger}, with values ranging between
about 19.23 and 7942 on a $20\times 20$ square partition of the
domain, similar to the examples presented in \cite{Filoche2012}.
These values were independently drawn from a uniform distribution on $[0,8000]$.
Also shown in this figure is the landscape function computed
in the finest space generated by landscape refinement.  The initial
mesh used order $p=2$ elements on the $20\times 20$ mesh used to
define $V$.  This was the type of problem that originally motivated our
consideration of landscape refinement, though we have already seen that
landscape refinement performs quite well on other problems for which the
landscape function does not exhibit such strong localization behavior.

Also in Figure~\ref{fig:HardSchrodinger} are convergence plots for the
relative error envelope bound $\eta_{\mathrm{max},\mathrm{rel}}^2$ for
the first $M=100$ eigenpairs, and several individual relative
eigenvalue errors $\eta_{\mathrm{eig},j}^2/\lambda_{j,h}$, under
landscape refinement.  For comparison, we also include the square of
the error bound for the landscape solution $\eta_\mathrm{land}^2$.  We
note that, although the individual eigenvectors have very different
behaviors, landscape refinement yields very similar convergence for
each of them.  Apart from scaling, this convergence tracks that of
$\eta_\mathrm{land}^2$ quite well, i.e.  all eigenpairs converge at
about the same speed as the landscape solution.

\begin{figure}
\centering
\includegraphics[width=0.48\textwidth]{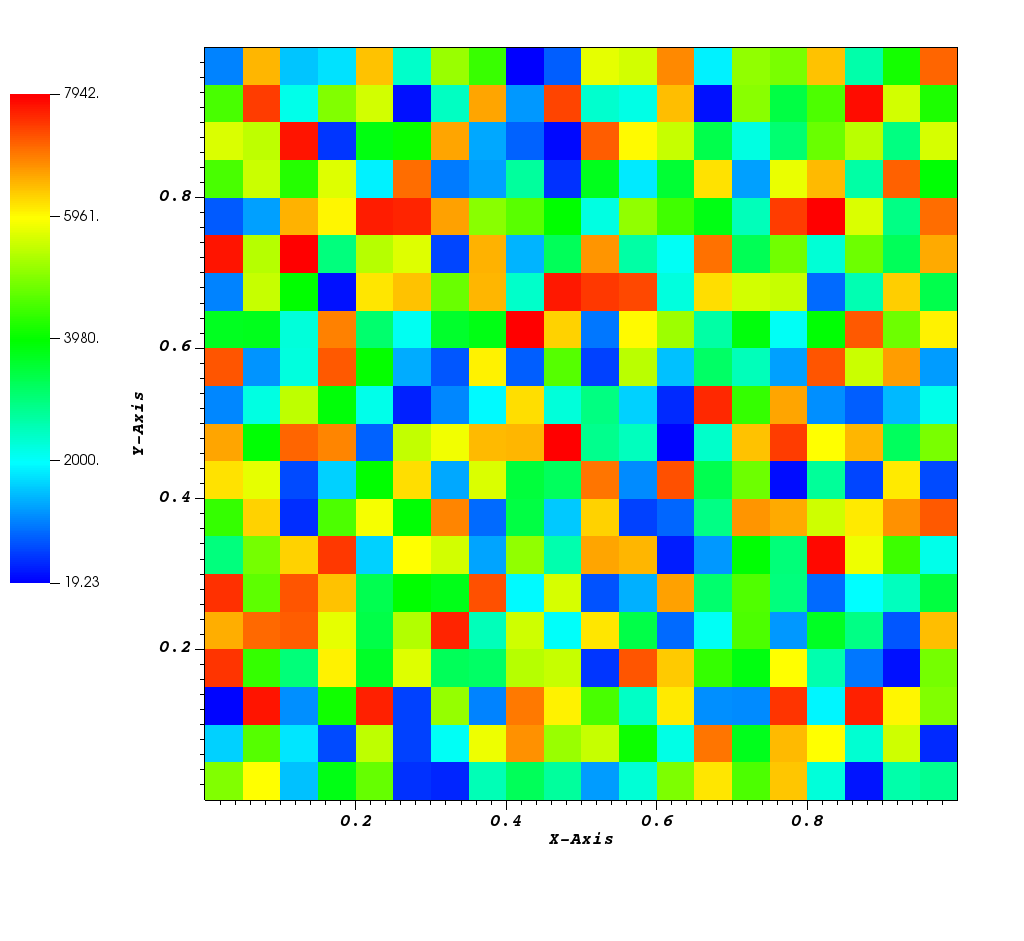}
\includegraphics[width=0.48\textwidth]{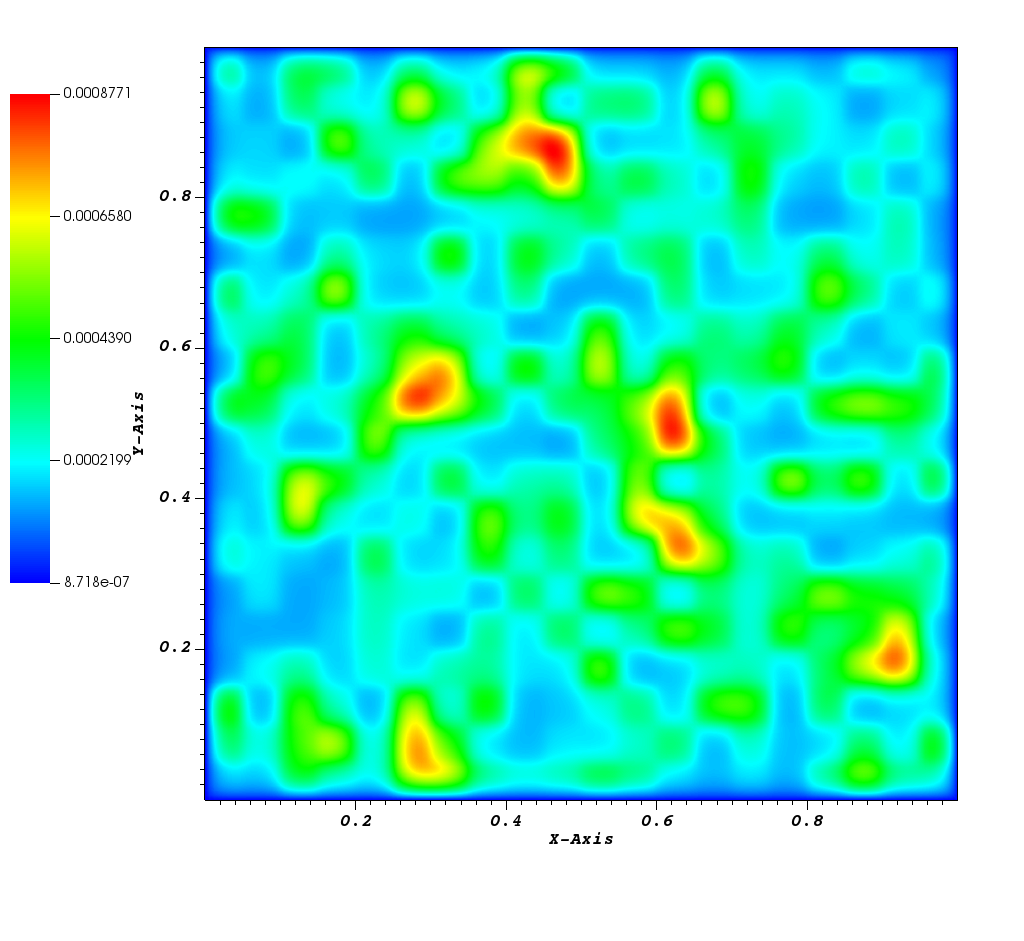}
\includegraphics[width=0.48\textwidth]{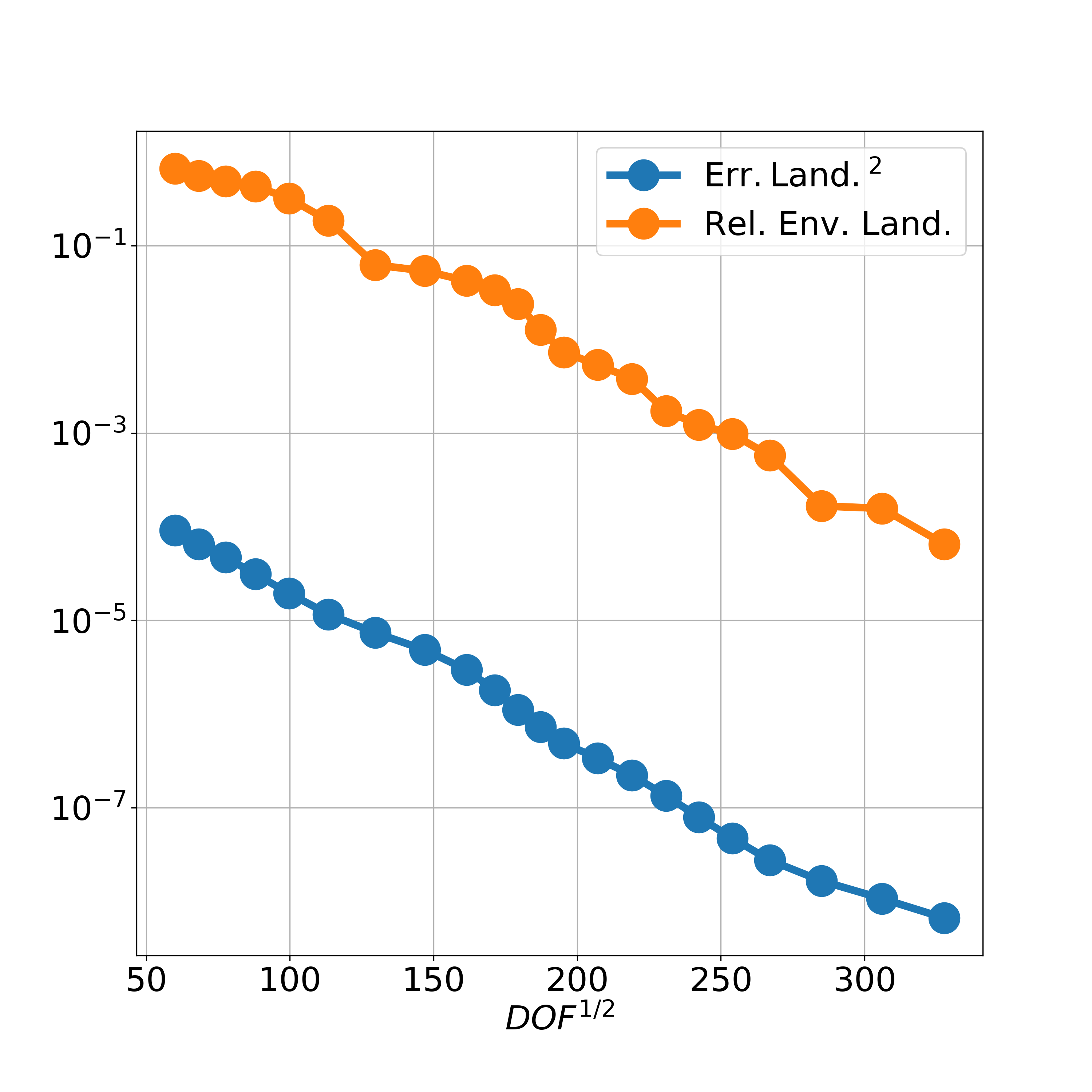}
\includegraphics[width=0.48\textwidth]{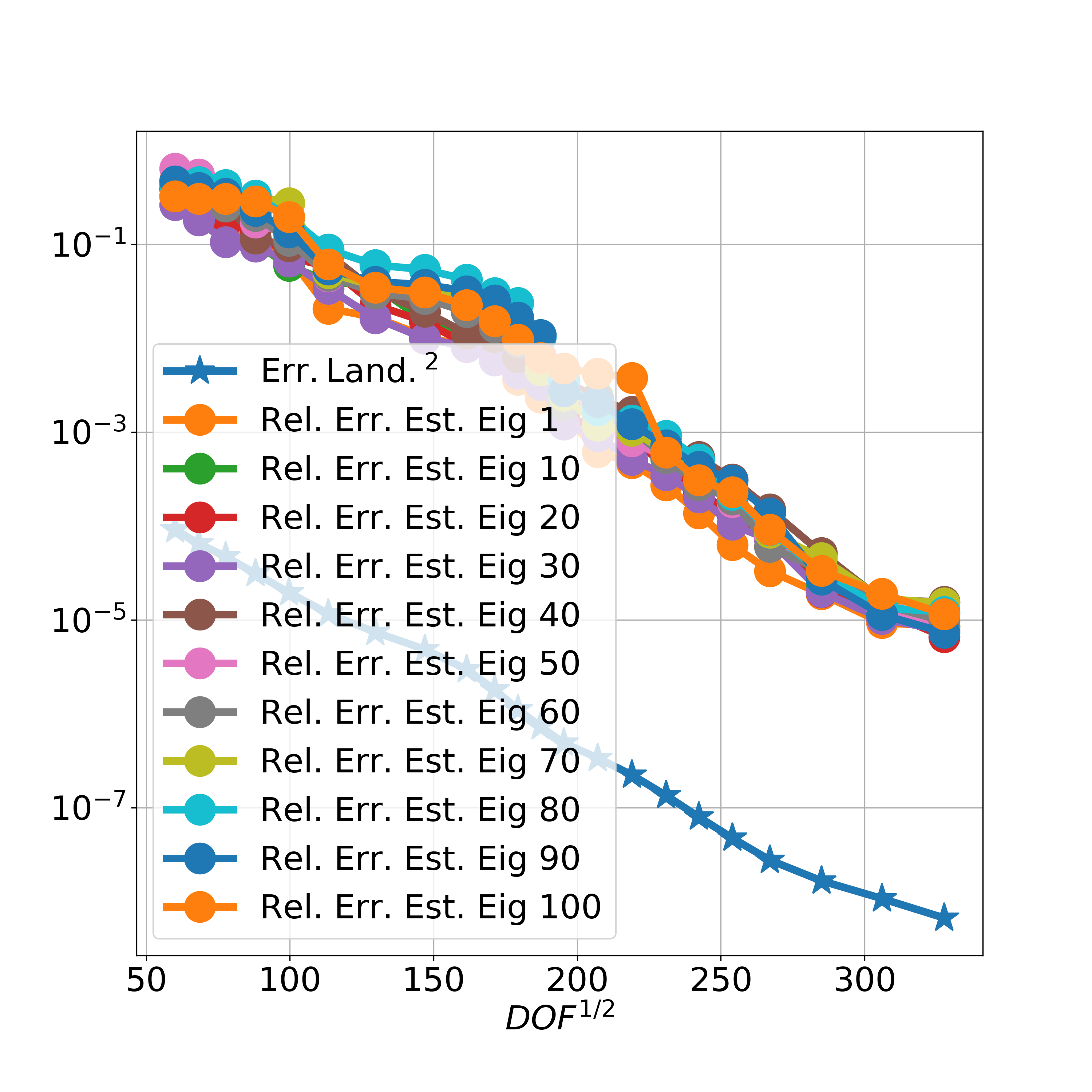}
\caption{Example~\ref{HardSchrodinger}.  Top row: The potential $V$ and the
  associated landscape function.  Bottom row: The relative error
  envelope bound $\eta_{\mathrm{max},\mathrm{rel}}^2$ (left) and
  several individual relative eigenvalue error bounds
  $\eta_{\mathrm{eig},j}^2/\lambda_{j,h}$ (right) under landscape refinement, plotted together
  with $\eta_\mathrm{land}^2$.}\label{fig:HardSchrodinger}
\end{figure}
\end{example}

\begin{example}[Discontinuous Diffusion]\label{DiscDiff}
Here, we consider the operator $\cL w=-\nabla\cdot(A\nabla w)$, where $A$ is piecewise
constant with value 1 in $(0,0.5]^2$ and $[0.5,1)^2$ and $\beta^2=10$
elsewhere.  We consider the first $M=50$ eigenpairs.
This kind of discontinuous diffusion was considered by
Kellogg~\cite{Kellogg1974} for the source problem, and the results
there suggest that some eigenvectors be singular near $(1/2,1/2)$,
with asymptotic behavior like $r^\alpha$ for some $\alpha\in(0,1)$,
where $r$ is the distance to $(1/2,1/2)$.  In this case, $\alpha$ can
be as small as $\alpha=(4/\pi)\mathrm{arccot}(\beta) \approx
0.389964$; see~\cite{Giani2016a}, where a related problem on the unit
disk is worked out analytically.

Convergence of the relative error envelope bound
$\eta_{\mathrm{max},\mathrm{rel}}^2$ and some individual relative
error bounds $\eta_{\mathrm{eig},j}^2/\lambda_{j,h}$ are given for
both landscape refinement and cluster refinement, in
Figure~\ref{fig:DiscDiff}.  For cluster refinement,
AdaptSpectrumEigSum is used.  As before, $\eta_\mathrm{land}^2$,
computed under landscape refinement, is provided for comparison.
Again, we begin with order $p=2$ elements on an $8\times 8$ square
partition of the domain.
\begin{figure}
  \centering
  \includegraphics[width=0.48\textwidth]{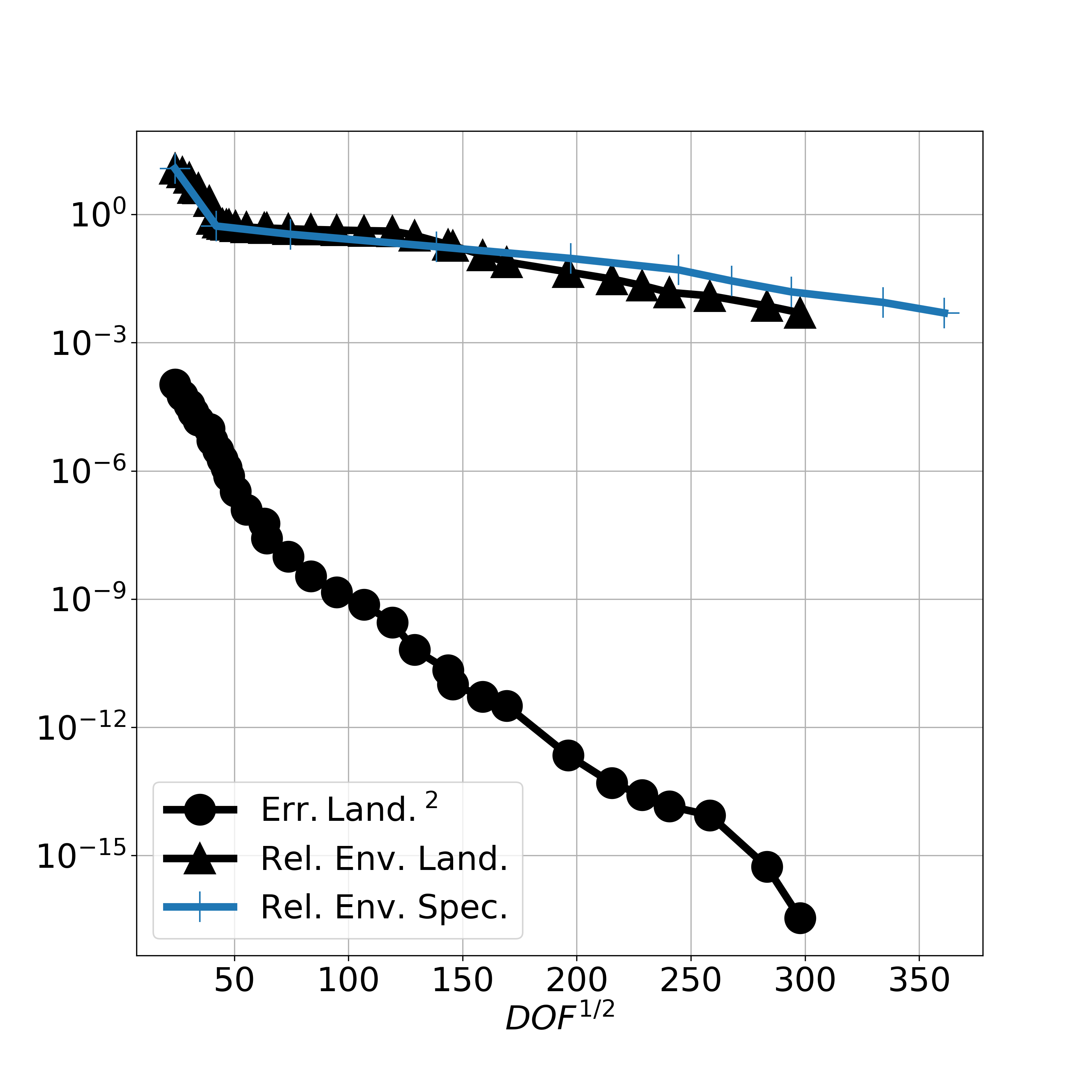}
  \includegraphics[width=0.48\textwidth]{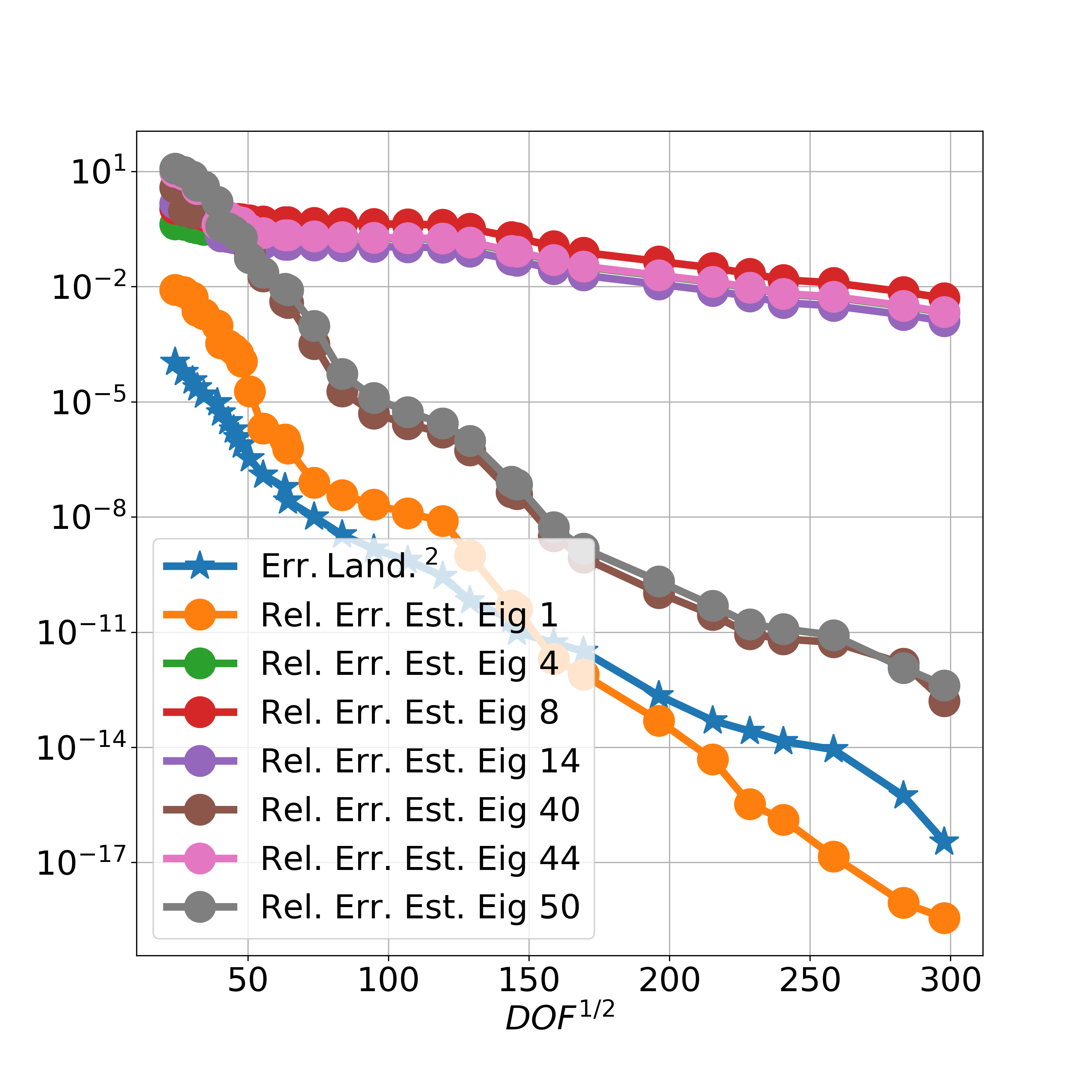}
\includegraphics[width=0.48\textwidth]{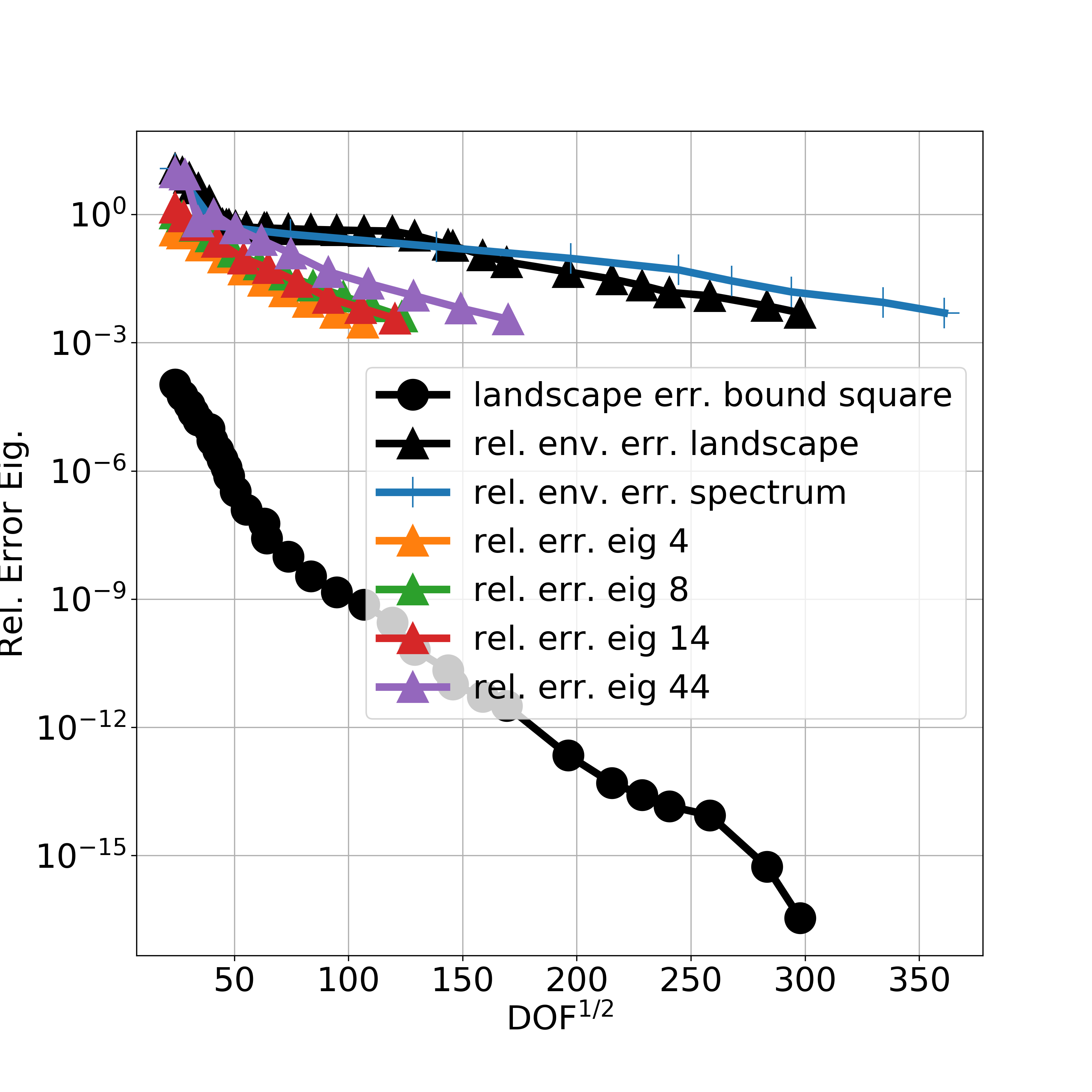}
  \includegraphics[width=0.48\textwidth]{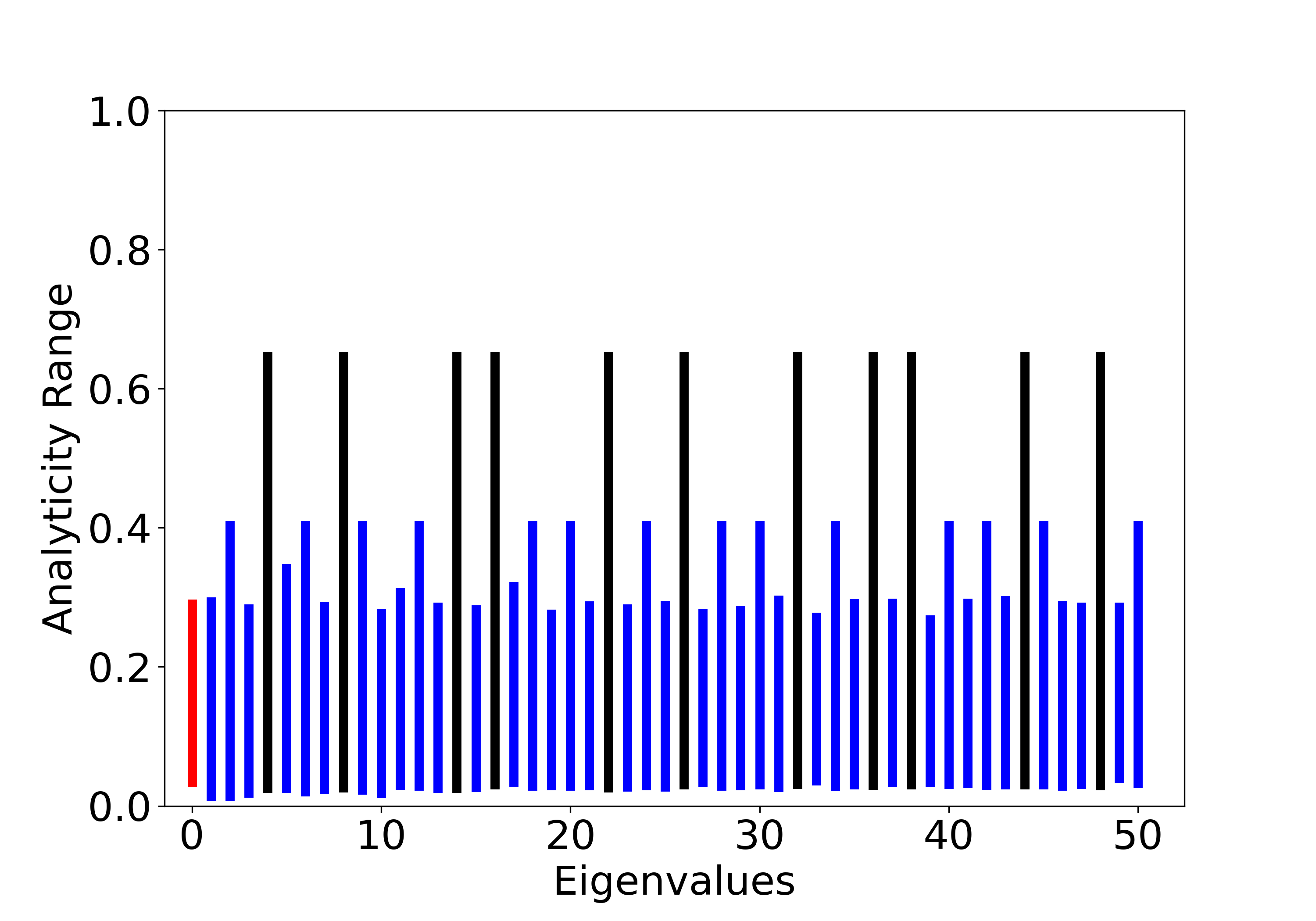}
\caption{\label{fig:DiscDiff} Example~\ref{DiscDiff}. Top left: Convergence
  plots of the square of the error bound for the landscape solution
  $\eta_{\mathrm{land}}^2$ and the envelope for the relative error
  bounds $\eta_{\mathrm{max},\mathrm{rel}}^2$ computed using landscape
  refinement and cluster refinement. Top right: Convergence rates of
  relative eigenvalue error estimates under landscape refinement for several
  eigenvalues within the cluster.  Bottom left: Convergence rates of
  eigenvectors 4, 8, 14 and 44 under single eigenpair refinement,
  superimposed over the top left plot.
Bottom right: Range of numerically
  estimated local regularity measures on the elements
  of the final mesh for each eigenpair. The landscape function is
  indicated in red, the slow-converging eigenpairs
  are in black, and the rest are in blue.}
\end{figure}
Landscape refinement overtakes cluster refinement near $150^2$ DOF,
but the convergence curves remain much closer to each other than in
previous examples, and the convergence, though still exponential, is
slower.  We also see in this figure the convergence of several
individual relative eigenvalue error estimates within the cluster.
The eigenvalues with indices 1, 40 and 50 converge like
$\eta_{\mathrm{max},\mathrm{rel}}^2$, whereas 4, 8, 14, and 44
converge much more slowly.  We highlight, however, that this slower
convergence is still exponential.

A likely explanation of this is that the
landscape function is more regular (less singular near $(1/2,1/2)$)
than these eigenvectors.  It is quite possible, due to symmetries in
the operator, that the eigenvectors with the strongest singular near
the $(1/2,1/2)$ have average value zero, so they do not appear in the
Fourier expansion of the landscape function $u$.  Based on our
computations, this appears to be the case.
Figure~\ref{fig:DiscDiff} provides an analysis of the numerically estimated
regularity for all eigenpairs in the cluster on the finest mesh.  These local regularity
measures were computed based on the algorithm described
in~\cite{houston_note_2005}, which uses the decay rates coefficients of local
Legendre expansions of the function under consideration to assess
regularity.  Without getting into detail here, we remark that these
local regularity measures are between $0$ and $1$, and that larger measures
indicate less regularity. Each vertical segment represents the ranges
of the regularity measures for the corresponding function.  We see
that eigenvectors 4, 8, 14, and 44 are among the 11 eigenvectors in
this cluster having the lowest regularity (22\% of them), and the landscape function
$u$ has the highest regularity.  Since landscape refinement uses this
local regularity indicator to determine how to refine a marked element
($h$ vs. $p$), this makes sense of the slower convergence of the less
regular eigenvectors.

Although all eigenvectors are converging exponentially under landscape
refinement, one might consider various ways of speeding up the
convergence of slower ones. The simplest thing to attempt in this
regard would be to decrease the parameter $\mathrm{tol}_\mathrm{ana}$
in the landscape refinement algorithm
(Algorithm~\ref{alg:AdaptLandscape}) to force $h$-refinement for more
of the marked elements.  The value of this parameter used in all
experiments in this section so far is
$\mathrm{tol}_\mathrm{ana}=0.25$.  We briefly summarize the results of
choosing $\mathrm{tol}_\mathrm{ana}\in\{0.1,0.001,0.0001\}$:
decreasing $\mathrm{tol}_\mathrm{ana}$ caused the convergence rates of
eigenpairs to get closer to each other, but did not really improve
individual or overall convergence rates. 
A more efficient way to deal with slow converging eigenpairs would be
to flag them either by looking at their convergent rates or their
local regularity measures, and treating these eigenvectors
differently.  There are a few natural ways of doing so.  Once the slow
converging eigenvectors are identified, one might use the worst of the
local regularity measures for these eigenvectors and the landscape
function to decide how to refine a marked element.  Alternatively, one
might ``extract'' them from the cluster and continue their
approximation individually on multiple meshes using single eigenpair
refinement, or perhaps collectively on a single mesh using cluster
refinement for those extracted eigenpairs.  We do not develop these
variants here, but instead show the effects of single eigenpair
refinement on eigenpairs 4, 8, 14 and 44 in Figure~\ref{fig:DiscDiff},
overlaid on the convergence plot for the entire cluster under
landscape and cluster refinements, to suggest how flagging slow
converging eigenpairs under landscape refinement for alternative
treatment can be beneficial.

\change{In the discussion preceding Example~\ref{hAdaptLShape} we
  indicated that is it possible that the source term $f=1$ used for
  our landscape function could be orthogonal to all eigenvectors of a
  certain type, e.g. all eigenvectors possessing the strongest point
  singularity, though that is expected to be unlikely in practice.  In
  that same discussion, we briefly mentioned the possibility of using a
  different source term $f$, and driving adaptivity based on the
  associated $u_f$.  In our discussion for the current example, we
  suggested that this orthogonality is probably what is causing the slower
  convergence of a few eigenvectors here.  The results presented in
  Figure~\ref{fig:DiscDiff2} provide strong support of this claim, and
  also provide a simple potential alternative to flagging slow-converging
  eigenvectors for special treatment.  We simply use a different
  source term $f$, as indicated above.}

\change{ Our first variation of the original problem is to move the
  singular point from $(1/2,1/2)$ to $(3/4,3/4)$ by modifying the
  diffusion coefficient $A$.  In this variant, $A=1$ on $(0,0.75]^2$
  and $[0.75,1)^2$, and $\beta^2=10$ elsewhere.  This shifting of the
  singularity results (empirically) in $f=1$ no longer being
  orthogonal to the eigenvectors of the new operator.  The top row in
  Figure~\ref{fig:DiscDiff2} provides a comparison between the
  behavior under landscape refinement for the original problem and
  this variant, and we see that the landscape function for the variant
  appears to encode the worst singular behavior in the eigenvectors.
  More specifically, the stagnating convergence observed in the
  original problem is broken at around $100^2$ DOFs, and there are no
  longer any slowly converging eigenvectors.}

\change{Our second variation of the original problem keeps the
  singular point at $(1/2,1/2)$, but uses the source term $f=1-3x$ to
  drive adaptivity based on $u_f$.  We highlight that there is nothing
  particularly special about this choice of $f$ beyond it being a
  low-degree polynomial that happens to be non-orthogonal to the most
  singular eigenvectors. The bottom plot in Figure~\ref{fig:DiscDiff2}
  compares the use of $f=1-3x$ and $f=1$.  This again corroborates our
  assertion that the culprit behind the slow convergence of some
  eigenvectors in the original example was the orthogonality of $f=1$
  and those vectors, which led to $u$ failing to encode those
  singularities.  It is also clear that a modest change to $f$
  completely fixed this issue.  The effectiveness of this alternate
  approach hints at a modification of our initial strategy of doing
  adaptive refinement based on a single source problem: choose a small
number of fixed sources, e.g. $\{1,x,y\}$, and drive adaptivity
based on the the associated solutions.  It becomes increasingly
unlikely that (singular) eigenvectors are orthogonal to even a small
set of source terms.  We do not pursue this idea further here.}
\end{example}
\begin{figure}
    \centering
\includegraphics[width=0.48\textwidth]{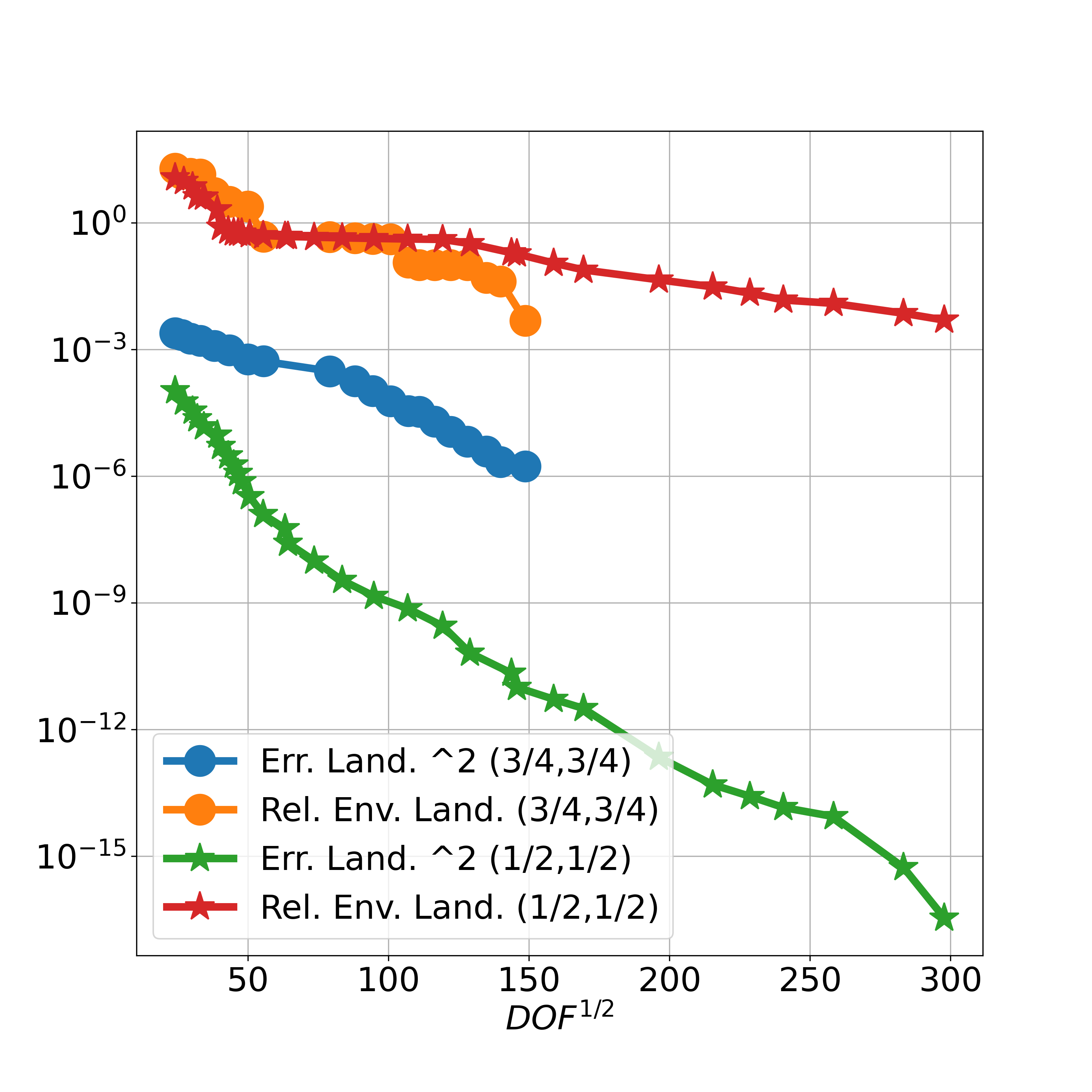}
\includegraphics[width=0.48\textwidth]{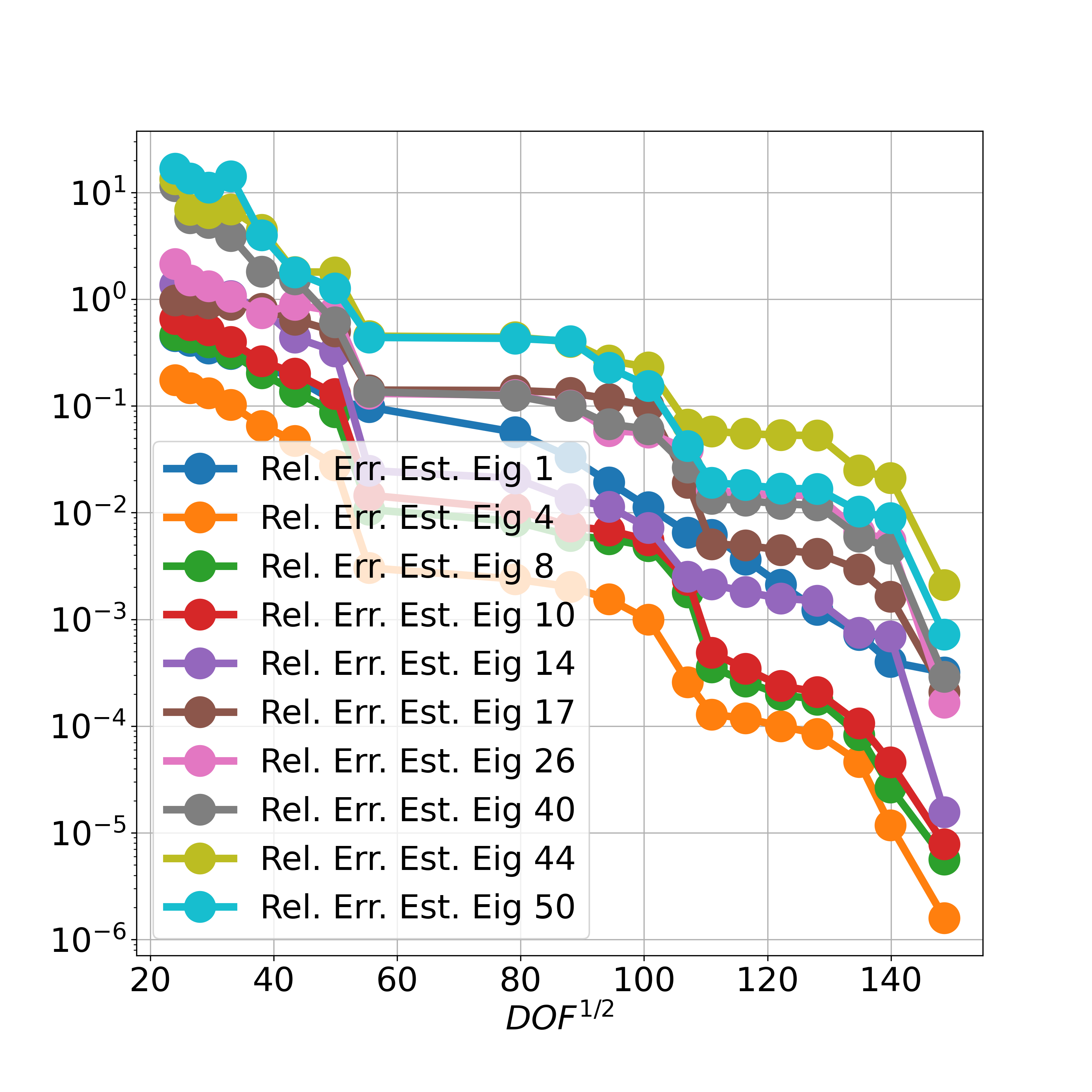}
\includegraphics[width=0.48\textwidth]{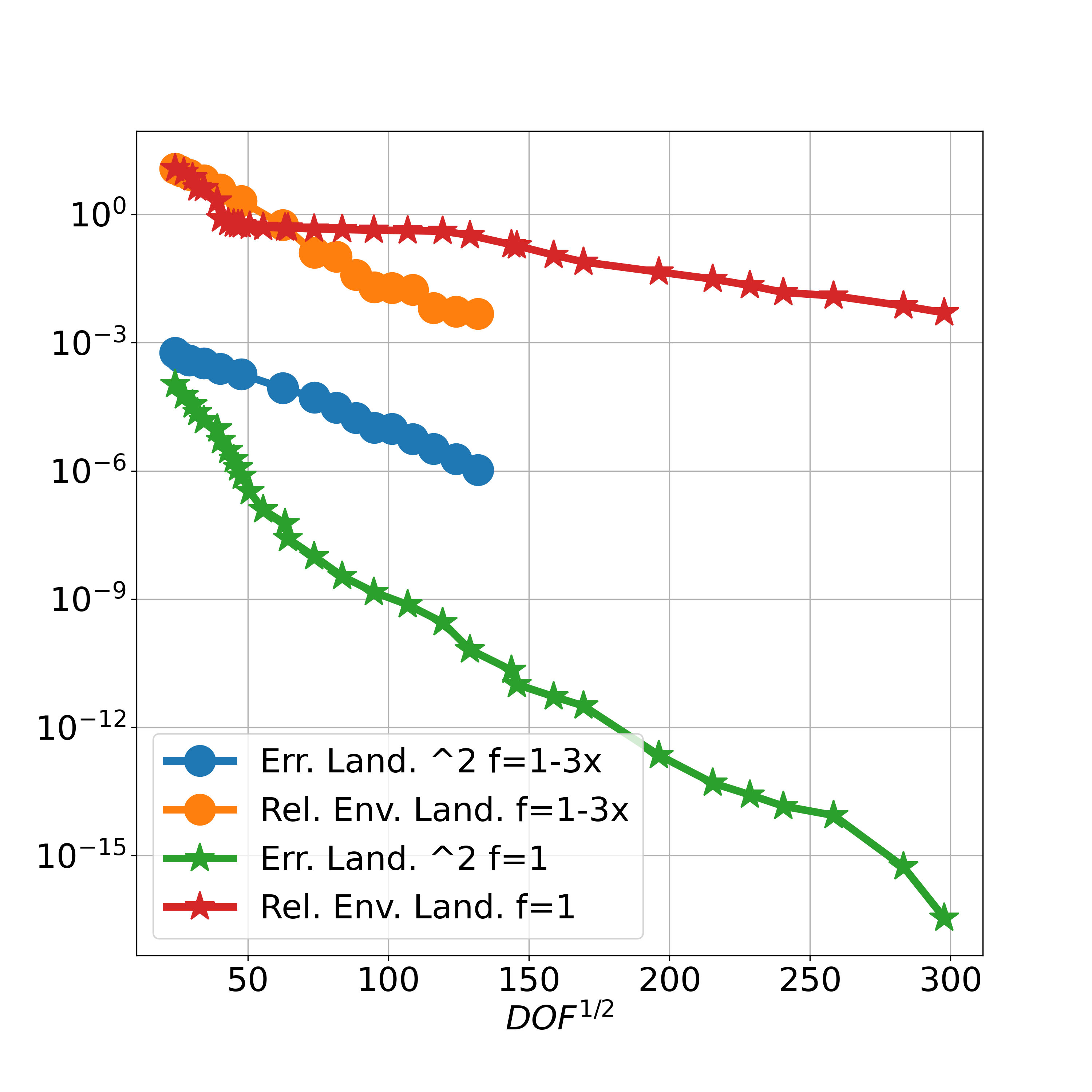}
\caption{\label{fig:DiscDiff2} Example~\ref{DiscDiff}. Top left:
  Convergence plots of the square of the error bound for the landscape
  solution $\eta_{\mathrm{land}}^2$ and the envelope for the relative
  error bounds $\eta_{\mathrm{max},\mathrm{rel}}^2$ computed using
  landscape refinement with singular point at $(1/2,1/2)$ and at
  $(3/4,3/4)$. Top right: Convergence rates of relative eigenvalue
  error estimates under landscape refinement for several eigenvalues
  within the cluster when the singular point is at $(3/4,3/4)$.
Bottom: Refinement driven by $u_f$ for the source term $f=1-3x$
compared with that of $u$ ($f=1$), with
singular point at $(1/2,1/2)$.}
\end{figure}

\change{\begin{example}[Perforated Domain]\label{SquareHoles} The
    family of domains $\Omega_m$ obtained by deleting $m^2$ squares
    having edge lengths $H=1/(2m+1)$ from the unit square, as pictured
    in Figure~\ref{fig:SquareHoles} for $m=3$, provides a setting in
    which some eigenvectors have many singularities, but other
    eigenvectors are analytic and known explicitly.  More
    specifically, if we consider the Dirichlet Laplacian on
    $\Omega_m$, many eigenvectors, including the ground state, will
    have $4m^2$ point singularities of type $r^{2/3}$, one for each of
    the non-convex corners in the domain.  The landscape function $u$
    will also have these same point singularities.  However, there are
    many analytic eigenvectors as well.  For example,
    $\psi=\sin(j\pi x/H)\sin(k\pi y/H)$ is an eigenvector with
    corresponding eigenvalue $\lambda = (j\pi/H)^2+(k\pi/H)^2$ for
    $j,k\in\NN$.  The smallest of these, which we expect to belong to
    the first analytic eigenvector, is $\lambda=2(\pi/H)^2$.

    When $m=3$, the first 100 eigenvalues are in the interval
    $[301.46,2070.55]$, and this interval only includes one of
    analytic eigenvectors described above, namely
    $\lambda=2(7\pi)^2\approx 967.22123$.  It is 41st in the spectrum,
    i.e. $\lambda=\lambda_{41}$.  The 41st computed eigenvector
    approximates $\sin(7\pi x)\sin(7\pi y)$ (up to scaling) reasonably
    well even on the coarsest space, which consists of quadratic
    elements on the mesh shown in Figure~\ref{fig:SquareHoles}.  Our
    empirical analyticity measures indicate that this is the only
    smooth eigenvector among the first 100, which aligns with our
    expectations. All other eigenvectors have analyticity measures
    similar to that of $u$, which suggests that each has a point
    singularity at at least one of the non-convex corners.  This again
    agrees with our expectations for this problem.
    % We also note
    % that the two neighboring eigenvalues $\lambda_{40}\approx
    % 955.3302$ and $\lambda_{42}\approx 973.2490$
    % 955.330248148894 967.221234091937 973.248995159117
    
    Because of the large number of singularities for $u$, one might
    naturally wonder if landscape refinement yields sub-optimal
    convergence for smooth eigenvectors such as $\psi_{41}$, because
    it would seem to concentrate much of its ``effort'' on resolving
    singularities that are not present in such eigenvectors.  As seen
    in Figure~\ref{fig:SquareHoles}, this intuition is correct.
    Refinement based on the computed approximation of $\psi_{41}$
    first achieves eigenvalue errors below $10^{-4}$ at around 40K
    DOFs, whereas this error is not achieved for landscape refinement
    until around 120K DOFs.  We note, however, that the best observed
    convergence rates for both approaches are very similar---the
    landscape approach just takes longer to achieve that rate.  This
    example provides a more dramatic instance of behavior we have seen
    in Figures~\ref{fig:square_compare} (left)
    and~\ref{fig:SimpleSchrodinger} (top right), namely that
    convergence of a single eigenpair under refinement that is
    (ideally) tailored to that specific eigenpair can outperform a
    more generic approach such as landscape refinement.  It is perhaps
    more noteworthy that landscape refinement sometimes ``wins''
    against single eigenvector refinement in this setting, and these
    same figures show that it can.  Regardless, we primarily promote
    landscape refinement as an efficient means of driving down
    collective measures of error (e.g. max error) for a cluster of
    eigenvalues, and such clusters will generically contain
    eigenvectors having singularities like those of $u$.  If one is
    willing to compute local analyticity measures, which is not
    uncommon in $hp$-adaptive algorithms, any eigenvectors that are
    flagged as being (much) smoother than $u$ could be
    extracted from a cluster for different treatment, in a similar
    fashion as we proposed in Example~\ref{DiscDiff} for eigenvectors that are
    more singular than $u$.
    \end{example}
  }
%DOF_vs_Rel_Eig_Err_Est_many_holes.png
\begin{figure}
  \centering
  \includegraphics[width=0.48\textwidth]{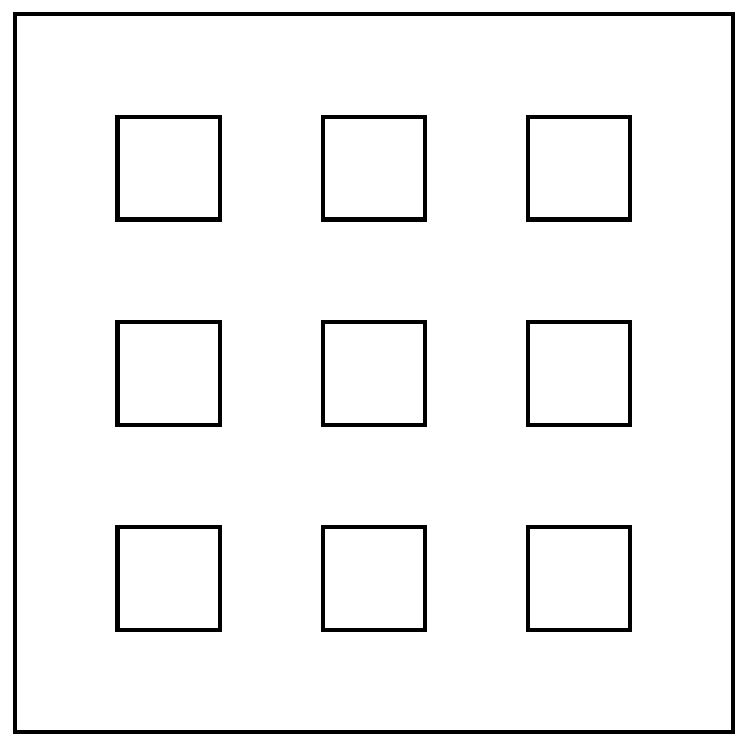}
  \includegraphics[width=0.48\textwidth]{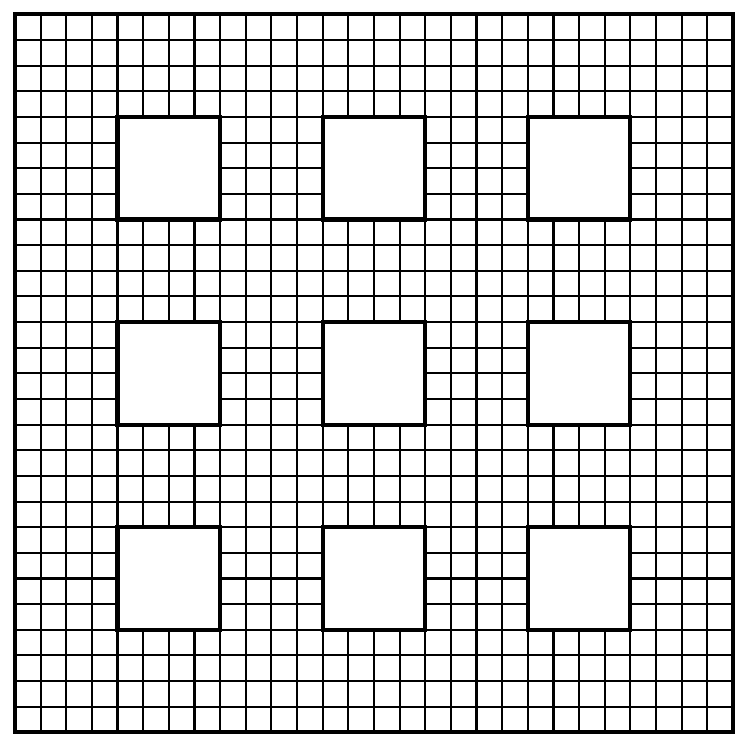}
  \includegraphics[width=0.48\textwidth]{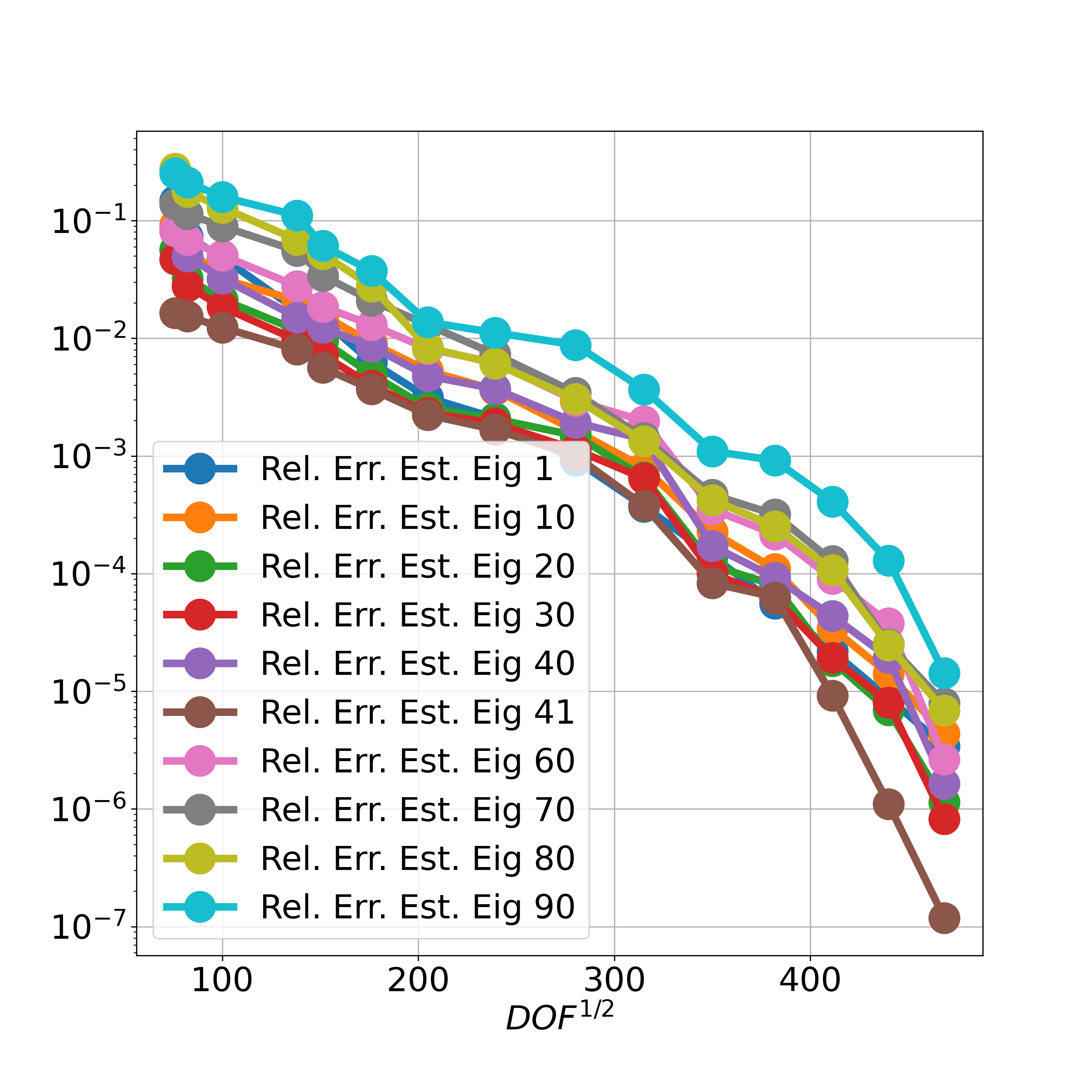}
  \includegraphics[width=0.48\textwidth]{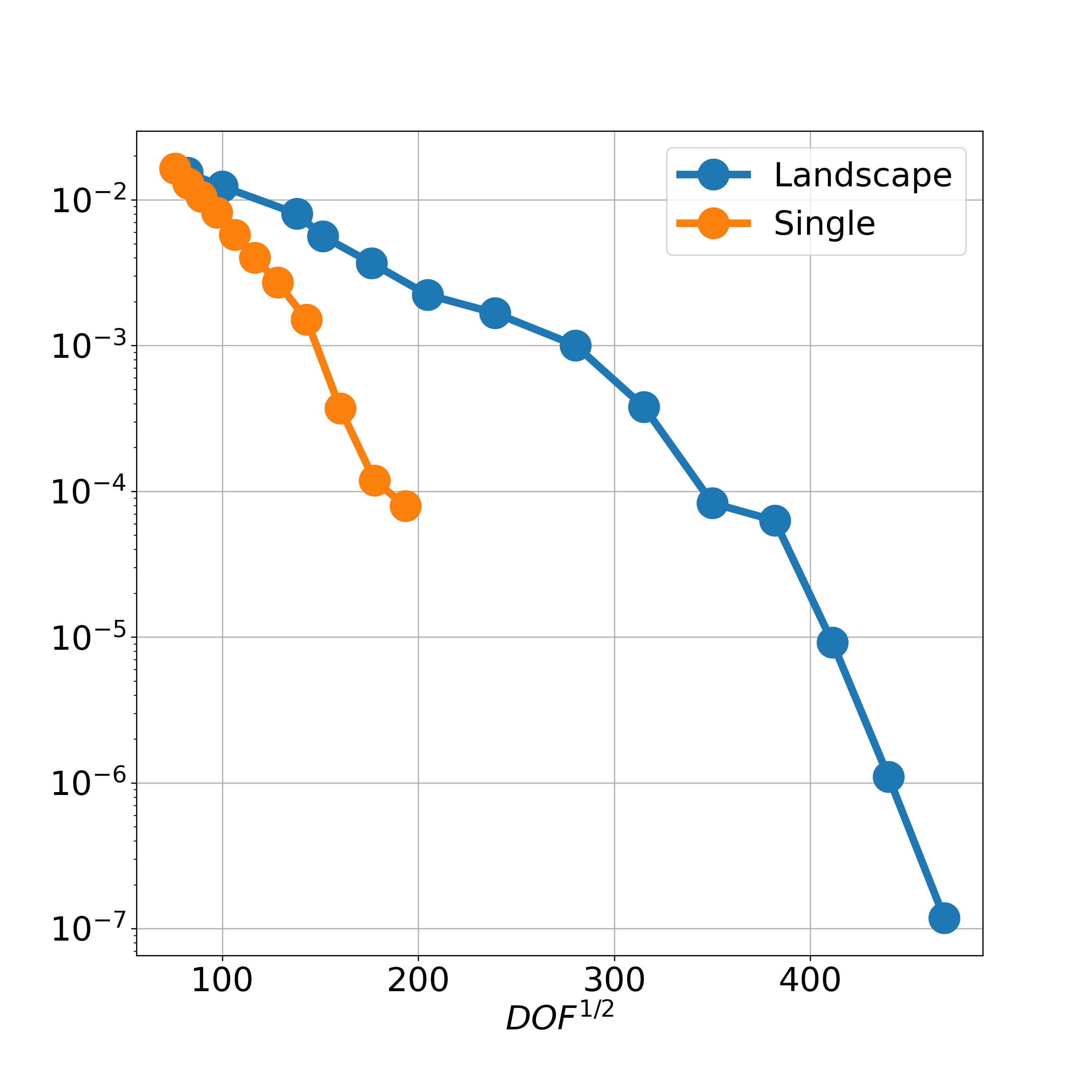}
\caption{\label{fig:SquareHoles} Example~\ref{SquareHoles}. Top row: The
  perforated unit square with $m^2$ square holes ($m=3$), and
  the initial mesh.  Bottom row:
  Eigenvalue convergence under landscape refinement for several
  eigenvalues, including $\lambda_{41}$, which has the only analytic
  eigenvector among the first 100.  Comparison of convergence for
  $\lambda_{41}$ under landscape refinement and single eigenvector
  refinement based on the computed approximation of $\psi_{41}$.}
\end{figure}

%%%%%%%%%%%%%%%%%%%%%%%%%%%%%%%%%%%%%%%%%%%%%%%%%%%%%%%%%%%%%%%%
%%%%%%%%%%%%%%%%%%%%%%%%%%%%%%%%%%%%%%%%%%%%%%%%%%%%%%%%%%%%%%%%
\section{The Algorithms}\label{sec:alg}

The algorithms are organized as follows:
\begin{enumerate}
\item Algorithm~\ref{alg:adapt_single} (single eigenpair refinement). The marking and
  refinement strategies are given in Algorithm~\ref{alg:AdaptSingleEig} (AdaptSingleEig).
\item Algorithm~\ref{alg:adapt_multi} (cluster refinement).  The
  marking and refinement strategies are given either by
  Algorithm~\ref{alg:AdaptSpectrumSum} (AdaptSpectrumEigSum) or
  Algorithm~\ref{alg:AdaptSpectrumMax} (AdaptSpectrumMax).
\item Algorithm~\ref{alg:adapt_landscape} (landscape refinement). The marking and
  refinement strategies are given in Algorithm~\ref{alg:AdaptLandscape} (AdaptLandscape).
\end{enumerate}
Algorithm~\ref{alg:adapt_single} is standard, and will not be
discussed in any detail here, apart from clarifying that, although it
targets the $j$th eigenpair, the first $j$ eigenpairs are
computed to ensure that we really have the $j$th in hand to drive the
refinement.  Both Algorithms~\ref{alg:adapt_multi}
and~\ref{alg:adapt_landscape} target the first $M$ eigenpairs of the
operator, and both use the same criterion for convergence.  They
differ (significantly) in how they go about \change{marking elements for
refinement} and deciding how to refine them.  In the case of
Algorithm~\ref{alg:adapt_landscape}, these decisions are based on the
current approximation of the landscape function $u_h$, as described in
Algorithm~\ref{alg:AdaptLandscape}.  For
Algorithm~\ref{alg:adapt_multi}, all of the computed eigenpairs are
used in some way to determine which elements are to be refined, and
how, and two natural strategies are given in
Algorithms~\ref{alg:AdaptSpectrumSum} and~\ref{alg:AdaptSpectrumMax}.
The two key parameters in
Algorithms~\ref{alg:AdaptSingleEig}-\ref{alg:AdaptLandscape} are $r$,
which is the \change{percentage of elements to be marked} for refinement, and
$\mathrm{tol}_\mathrm{ana}$, which dictates the choice of $h$ versus
$p$ refinement for a marked element based on a computed measure of the
local regularity of a function in that element.  The default settings
for these parameters
for our experiments, are $r=10$\% and
$\mathrm{tol}_\mathrm{ana}=0.25$.
The local regularity measure is computed using the method presented in
\cite{houston_note_2005}, which is based on the coefficient decay of the local
Legendre series expansion.  This assessment takes place on 
line 4 of Algorithm~\ref{alg:AdaptSingleEig},
line 6 of Algorithm~\ref{alg:AdaptSpectrumSum}, line 5 of 
Algorithm~\ref{alg:AdaptSpectrumMax}, and line 4 of Algorithm~\ref{alg:AdaptLandscape}.

\begin{algorithm}\caption{Adaptive algorithm for a single eigenpair}\label{alg:adapt_single}
 \hspace*{\algorithmicindent} \textbf{Input} \\
 \hspace*{\algorithmicindent} \hspace*{\algorithmicindent} Initial FE space $S(\mathcal{T})$.\\
 \hspace*{\algorithmicindent} \hspace*{\algorithmicindent} Index $j$ of the eigenpair to compute.\\
 \hspace*{\algorithmicindent} \hspace*{\algorithmicindent} Maximum number of adaptive steps $n_\mathrm{max}$.\\
 \hspace*{\algorithmicindent} \hspace*{\algorithmicindent} Relative tolerance $\mathrm{tol}$.\\
 \hspace*{\algorithmicindent} \textbf{Output} \\
 \hspace*{\algorithmicindent} \hspace*{\algorithmicindent} Approximated eigenpair $(\lambda_h,\phi_h)$.\\
\begin{algorithmic}[1]
\State{$n:=1$}
\While {$n<n_\mathrm{max}$}
	\State{Compute an approximation of $(\lambda_{j,h},\phi_{j,h})$ computing the first $j$ eigenpairs at the bottom of the spectrum for the problem.}
	\State{Compute the error estimator $\eta_{\mathrm{eig},j}$.}
    \If{$\eta_{\mathrm{eig},j}^2/\lambda_{j,h}<\mathrm{tol}$}
    	\State{$(\lambda_h,\phi_h):=(\lambda_{j,h},\phi_{j,h})$ and exit.}
  	\Else
    	\State{Adapt $S(\mathcal{T})$ using $\eta_{\mathrm{eig},j}$ and $\phi_h$ by calling AdaptSingleEig.}
  	\EndIf
  	\State{$n:=n+1$}
\EndWhile{}
\end{algorithmic}
\end{algorithm}

\begin{algorithm} \caption{Adaptive algorithm for a portion of the spectrum}\label{alg:adapt_multi}
 \hspace*{\algorithmicindent} \textbf{Input} \\
 \hspace*{\algorithmicindent} \hspace*{\algorithmicindent} Initial FE space $S(\mathcal{T})$.\\
 \hspace*{\algorithmicindent} \hspace*{\algorithmicindent} Index $M$ of the highest eigenpair to compute.\\
 \hspace*{\algorithmicindent} \hspace*{\algorithmicindent} Maximum number of adaptive steps $n_\mathrm{max}$.\\
 \hspace*{\algorithmicindent} \hspace*{\algorithmicindent} Relative tolerance $\mathrm{tol}$.\\
 \hspace*{\algorithmicindent} \textbf{Output} \\
 \hspace*{\algorithmicindent} \hspace*{\algorithmicindent} Approximated eigenpairs $(\lambda_{j,h},\phi_{j,h})$, for all $j=1,\dots,M$.\\
\begin{algorithmic}[1]
\State{$n:=1$}
\While {$n<n_\mathrm{max}$}
	\State{Compute an approximation of $(\lambda_{j,h},\phi_{j,h})$ computing the first $M$ eigenpairs at the bottom of the spectrum for the problem.}
	\State{Compute the error estimator $\eta_{\mathrm{eig},j}$, for all $j=1,\dots,M$.}
    \If{$\max_{j\leq M}(\eta_{\mathrm{eig},j}^2)<\mathrm{tol}$}
    	\State{Return $(\lambda_{j,h},\phi_{j,h})$, for all $j=1,\dots,M$, and and exit.}
  	\Else
    	\State{Adapt $S(\mathcal{T})$ using $\eta_{\mathrm{eig},j}$ and $\phi_{j,h}$, for all $j=1,\dots,M$ by calling either AdaptSpectrumEigSum orAdaptSpectrumEigMax.}
  	\EndIf
  	\State{$n:=n+1$}
\EndWhile{}
\end{algorithmic}
\end{algorithm}

\begin{algorithm} \caption{Adaptive algorithm using the landscape function}\label{alg:adapt_landscape}
 \hspace*{\algorithmicindent} \textbf{Input} \\
 \hspace*{\algorithmicindent} \hspace*{\algorithmicindent} Initial FE space $S(\mathcal{T})$.\\
 \hspace*{\algorithmicindent} \hspace*{\algorithmicindent} Index $M$ of the highest eigenpair to compute.\\
 \hspace*{\algorithmicindent} \hspace*{\algorithmicindent} Maximum number of adaptive steps $n_\mathrm{max}$.\\
 \hspace*{\algorithmicindent} \hspace*{\algorithmicindent} Relative tolerance $\mathrm{tol}$.\\
 \hspace*{\algorithmicindent} \textbf{Output} \\
 \hspace*{\algorithmicindent} \hspace*{\algorithmicindent} Approximated eigenpairs $(\lambda_{j,h},\phi_{j,h})$, for all $j=1,\dots,M$.\\
\begin{algorithmic}[1]
\State{$n:=1$}
\While {$n<n_\mathrm{max}$}
	\State{Compute an approximation of $(\lambda_{j,h},\phi_{j,h})$ computing the first $M$ eigenpairs at the bottom of the spectrum for the problem.}
	\State{Compute an approximation $u_h$ of the solution of the landscape function problem.}
	\State{Compute the error estimator $\eta_{\mathrm{eig},j}$, for all $j=1,\dots,M$.}
	\State{Compute the error estimator $\eta_\mathrm{land}$.}
    \If{$\max_{j\leq M}(\eta_{\mathrm{eig},j}^2)<\mathrm{tol}$}
    	\State{Return $(\lambda_{j,h},\phi_{j,h})$, for all $j=1,\dots,M$, and and exit.}
  	\Else
    	\State{Adapt $S(\mathcal{T})$ using $\eta_\mathrm{land}$ and $u_h$ by calling AdaptLandscape.}
  	\EndIf
  	\State{$n:=n+1$}
\EndWhile{}
\end{algorithmic}
\end{algorithm}

\begin{algorithm} \caption{AdaptSingleEig: adaptation of the mesh using a single eigenpair}\label{alg:AdaptSingleEig}
 \hspace*{\algorithmicindent} \textbf{Input} \\
 \hspace*{\algorithmicindent} \hspace*{\algorithmicindent} Current FE space $S(\mathcal{T})$.\\
 \hspace*{\algorithmicindent} \hspace*{\algorithmicindent} Percentage $r$ of elements to refine.\\
 \hspace*{\algorithmicindent} \hspace*{\algorithmicindent} Error estimator $\eta_{\mathrm{eig},j}$\\
 \hspace*{\algorithmicindent} \hspace*{\algorithmicindent} Computed solution $\phi_h$.\\
 \hspace*{\algorithmicindent} \hspace*{\algorithmicindent} Threshold $\mathrm{tol}_\mathrm{ana}$ for the smoothness criteria.\\
 \hspace*{\algorithmicindent} \textbf{Output} \\
 \hspace*{\algorithmicindent} \hspace*{\algorithmicindent} Adapted FE space $S(\mathcal{T})$.\\
\begin{algorithmic}[1]
\State{Order the elements of $\mathcal{T}$ in decreasing order with respect to $\eta_{\mathrm{eig},K,j}^2$. }
\State{Mark $r$-percent of the elements from the top of the sorted list. Create a set $\mathcal{M}$ of the marked elements.}
\ForAll{$K\in\mathcal{M}$}                    
	\State {Estimate the smoothness of $\phi_h$ in $K$.}
	\If{The smoothness is below $\mathrm{tol}_\mathrm{ana}$ }
    	\State{Mark the element $K$ to be refined in $p$.}
  	\Else
    	\State{Mark the element $K$ to be refined in $h$.}
  	\EndIf
\EndFor
\State{Enforce local properties for marked elements.}
\State{Adapt $S(\mathcal{T})$.}
\end{algorithmic}
\end{algorithm}

\begin{algorithm} \caption{AdaptSpectrumEigSum: adaptation of the mesh using the sum of all computed eigenpairs}\label{alg:AdaptSpectrumSum}
 \hspace*{\algorithmicindent} \textbf{Input} \\
 \hspace*{\algorithmicindent} \hspace*{\algorithmicindent} Current FE space $S(\mathcal{T})$.\\
 \hspace*{\algorithmicindent} \hspace*{\algorithmicindent} Percentage $r$ of elements to refine.\\
 \hspace*{\algorithmicindent} \hspace*{\algorithmicindent} Error estimators $\eta_{\mathrm{eig},j}$, for all $j=1,\dots,M$\\
 \hspace*{\algorithmicindent} \hspace*{\algorithmicindent} Computed eigenvectors $\phi_{j,h}$, for all $j=1,\dots,M$.\\
 \hspace*{\algorithmicindent} \hspace*{\algorithmicindent} Threshold $\mathrm{tol}_\mathrm{ana}$ for the smoothness criteria.\\
 \hspace*{\algorithmicindent} \textbf{Output} \\
 \hspace*{\algorithmicindent} \hspace*{\algorithmicindent} Adapted FE space $S(\mathcal{T})$.\\
\begin{algorithmic}[1]
\State{For each element $K$, calculate $\eta_{\mathrm{sum},K}^2=\sum_{j=1}^M \eta_{\mathrm{eig},K,j}^2$.}
\State{Sum all the computed eigenpairs $\phi_{\mathrm{sum},h}=\sum_{j=1}^M \phi_{j,h}$.}
\State{Order the elements of $\mathcal{T}$ in decreasing order with respect to $\eta_{\mathrm{sum},K}^2$. }
\State{Mark $r$-percent of the elements from the top of the sorted list. Create a set $\mathcal{M}$ of the marked elements.}
\ForAll{$K\in\mathcal{M}$}                    
	\State {Estimate the smoothness of $\phi_{\mathrm{sum},h}$ in $K$.}
	\If{The smoothness is below $\mathrm{tol}_\mathrm{ana}$ }
    	\State{Mark the element $K$ to be refined in $p$.}
  	\Else
    	\State{Mark the element $K$ to be refined in $h$.}
  	\EndIf
\EndFor
\State{Enforce local properties for marked elements.}
\State{Adapt $S(\mathcal{T})$.}
\end{algorithmic}
\end{algorithm}

\begin{algorithm} \caption{AdaptSpectrumEigMax: adaptation of the mesh using the max of all computed eigenpairs}\label{alg:AdaptSpectrumMax}
 \hspace*{\algorithmicindent} \textbf{Input} \\
 \hspace*{\algorithmicindent} \hspace*{\algorithmicindent} Current FE space $S(\mathcal{T})$.\\
 \hspace*{\algorithmicindent} \hspace*{\algorithmicindent} Percentage $r$ of elements to refine.\\
 \hspace*{\algorithmicindent} \hspace*{\algorithmicindent} Error estimators $\eta_{\mathrm{eig},j}$, for all $j=1,\dots,M$\\
 \hspace*{\algorithmicindent} \hspace*{\algorithmicindent} Computed eigenvectors $\phi_{j,h}$, for all $j=1,\dots,M$.\\
 \hspace*{\algorithmicindent} \hspace*{\algorithmicindent} Threshold $\mathrm{tol}_\mathrm{ana}$ for the smoothness criteria.\\
 \hspace*{\algorithmicindent} \textbf{Output} \\
 \hspace*{\algorithmicindent} \hspace*{\algorithmicindent} Adapted FE space $S(\mathcal{T})$.\\
\begin{algorithmic}[1]
\State{For each element $K$, calculate $\eta_{\mathrm{max},K}^2=\max_{j=1}^M \{\eta_{\mathrm{eig},K,j}^2\}$.}
\State{Order the elements of $\mathcal{T}$ in decreasing order with respect to $\eta_{\mathrm{max},K}^2$. }
\State{Mark $r$-percent of the elements from the top of the sorted list. Create a set $\mathcal{M}$ of the marked elements.}
\ForAll{$K\in\mathcal{M}$}                    
	\State {Estimate the smoothness of $\phi_{i,h}$ in $K$ where $i=\mathrm{argmax}_{j=1}^M \{\eta_{\mathrm{eig},K,j}^2\}$.}
	\If{The smoothness is below $\mathrm{tol}_\mathrm{ana}$ }
    	\State{Mark the element $K$ to be refined in $p$.}
  	\Else
    	\State{Mark the element $K$ to be refined in $h$.}
  	\EndIf
\EndFor
\State{Enforce local properties for marked elements.}
\State{Adapt $S(\mathcal{T})$.}
\end{algorithmic}
\end{algorithm}

\begin{algorithm} \caption{AdaptLandscape: adaptation of the mesh for linear problems}\label{alg:AdaptLandscape}
 \hspace*{\algorithmicindent} \textbf{Input} \\
 \hspace*{\algorithmicindent} \hspace*{\algorithmicindent} Current FE space $S(\mathcal{T})$.\\
 \hspace*{\algorithmicindent} \hspace*{\algorithmicindent} Percentage $r$ of elements to refine.\\
 \hspace*{\algorithmicindent} \hspace*{\algorithmicindent} Error estimator $\eta_\mathrm{land}$\\
 \hspace*{\algorithmicindent} \hspace*{\algorithmicindent} Computed solution $u_h$.\\
 \hspace*{\algorithmicindent} \hspace*{\algorithmicindent} Threshold $\mathrm{tol}_\mathrm{ana}$ for the smoothness criteria.\\
 \hspace*{\algorithmicindent} \textbf{Output} \\
 \hspace*{\algorithmicindent} \hspace*{\algorithmicindent} Adapted FE space $S(\mathcal{T})$.\\
\begin{algorithmic}[1]
\State{Order the elements of $\mathcal{T}$ in decreasing order with respect to $\eta_{\mathrm{land},K}^2$. }
\State{Mark $r$-percent of the elements from the top of the sorted list. Create a set $\mathcal{M}$ of the marked elements.}
\ForAll{$K\in\mathcal{M}$}                    
	\State {Estimate the smoothness of $u_h$ in $K$.}
	\If{The smoothness is below $\mathrm{tol}_\mathrm{ana}$ }
    	\State{Mark the element $K$ to be refined in $p$.}
  	\Else
    	\State{Mark the element $K$ to be refined in $h$.}
  	\EndIf
\EndFor
\State{Enforce local properties for marked elements.}
\State{Adapt $S(\mathcal{T})$.}
\end{algorithmic}
\end{algorithm}

%%%%%%%%%%%%%%%%%%%%%%%%%%%%%%%%%%%%%%%%%%%%%%%%%%%%%%%%%%%%%%%%%%%%%%%%%%%
%%%%%%%%%%%%%%%%%%%%%%%%%%%%%%%%%%%%%%%%%%%%%%%%%%%%%%%%%%%%%%%%%%%%%%%%%%%

\section{Optimisation of the algorithm}\label{AlgOpt}
Considering Example~\ref{SimpleSchrodinger} as a model problem for
illustrative purposes, we see in Table~\ref{tab:cumulative_pie}
that almost 80\% of the CPU time is spent on the
eigensolver. Moreover, in Figure~\ref{fig:timing}, it is clear that
the CPU needed by the eigensolver is growing faster than any other
part of the code.  In the next two subsections, different ways to
optimise and speed up Algorithm~\ref{alg:adapt_landscape}, the
landscape refinement algorithm, are
presented.
\begin{table}
\begin{center}
\begin{tabular}{ c c }

  Algorithm Parts   & CPU Time (\%)  \\
\hline\\
             Eigensolver    &      77.6\\
             Solver     &     10.7\\
             Error Estimator Eigenproblem   &      9.2\\
             Error Estimator Landscape    &      0.2\\
             Other    &      2.3

\end{tabular}
\end{center}\caption{Cumulative CPU time for different parts of
  Algorithm~\ref{alg:adapt_landscape} on Example~\ref{UnitSquare}.}\label{tab:cumulative_pie}
\end{table}

\begin{figure}
\centering
\includegraphics[width=0.8\textwidth]{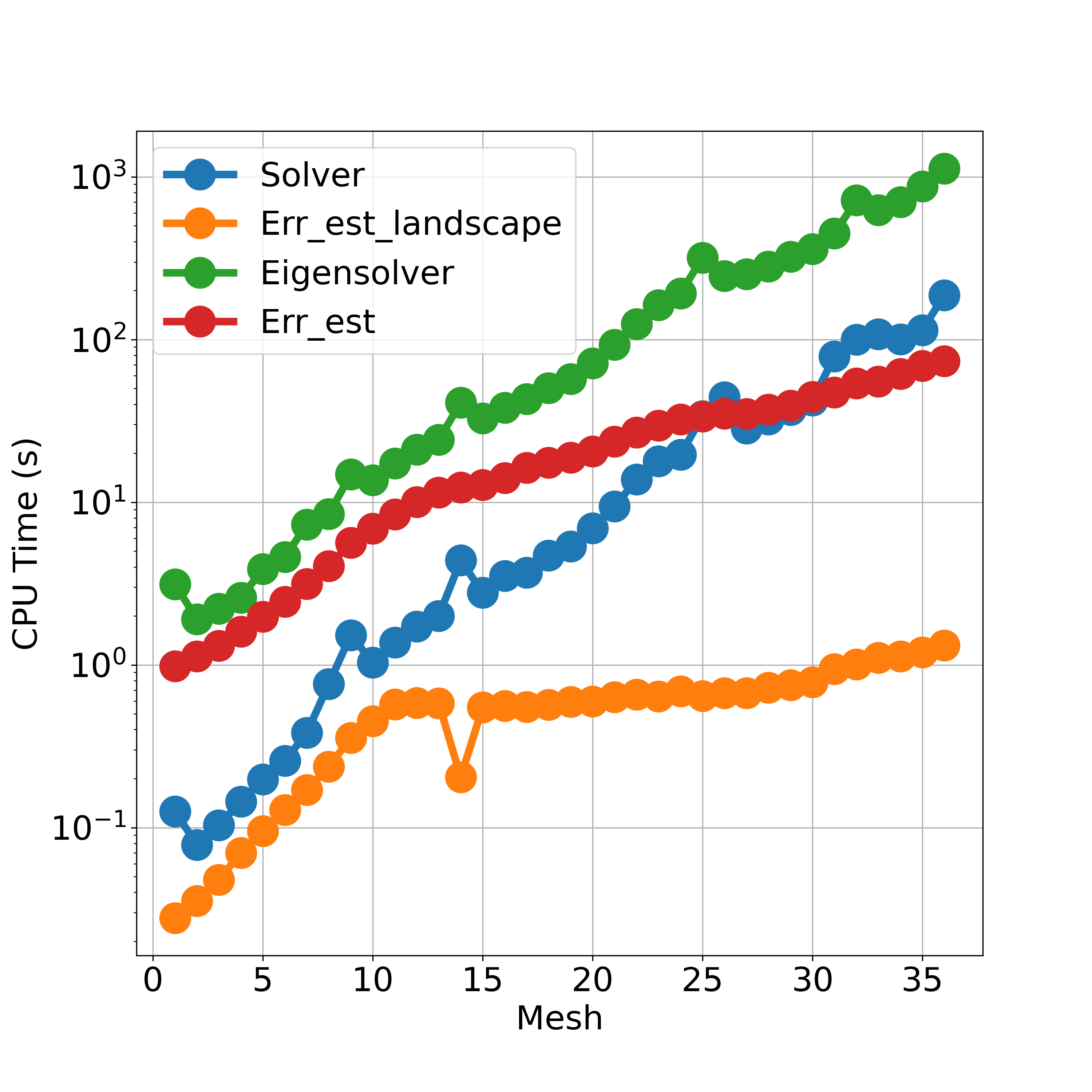}
\caption{ CPU time of different parts of the code during the adaptation of the mesh for Algorithm~\ref{alg:adapt_landscape}.} \label{fig:timing}
\end{figure}

\subsection{Pausing the eigensolver}

Since the computed eigenpairs are not used to refine the mesh, but
only in the stopping criterion in Algorithm~\ref{alg:adapt_landscape},
it makes sense to not call the eigensolver in every iteration.
Algorithm~\ref{alg:adapt_landscape_pause} is a modification of
Algorithm~\ref{alg:adapt_landscape} where a pause is introduced for
the eigensolver. The length of the pause is controlled by
$\ell_\mathrm{pause}\ge 0$ and measured in the number of
iterations. The eigensolver is always called during the first
iteration.  The algorithm cannot terminate while the eigensolver is
paused because the stopping criterion is based on the computed
spectrum.

\begin{algorithm} \caption{Pausing adaptive algorithm using the landscape function }\label{alg:adapt_landscape_pause}
 \hspace*{\algorithmicindent} \textbf{Input} \\
 \hspace*{\algorithmicindent} \hspace*{\algorithmicindent} Initial FE space $S(\mathcal{T})$.\\
 \hspace*{\algorithmicindent} \hspace*{\algorithmicindent} Index $M$ of the highest eigenpair to compute.\\
 \hspace*{\algorithmicindent} \hspace*{\algorithmicindent} Maximum number of adaptive steps $n_\mathrm{max}$.\\
 \hspace*{\algorithmicindent} \hspace*{\algorithmicindent} Relative tolerance $\mathrm{tol}$.\\
 \hspace*{\algorithmicindent} \hspace*{\algorithmicindent} Length of the pause $\ell_\mathrm{pause}$.\\
 \hspace*{\algorithmicindent} \textbf{Output} \\
 \hspace*{\algorithmicindent} \hspace*{\algorithmicindent} Approximated eigenpairs $(\lambda_{j,h},\phi_{j,h})$, for all $j=1,\dots,M$.\\
\begin{algorithmic}[1]
\State{$n:=1$}
\State{$\mathrm{ct}_\mathrm{pause}:=\ell_\mathrm{pause}$}
\While {$n<n_\mathrm{max}$}
	\State{Compute an approximation $u_h$ of the solution of the landscape function problem.}
	\State{Compute the error estimator $\eta_\mathrm{land}$.}
        \If{$\mathrm{ct}_\mathrm{pause}=\ell_\mathrm{pause}$}
		\State{Compute an approximation of $(\lambda_{j,h},\phi_{j,h})$ computing the first $M$ eigenpairs at the bottom of the spectrum for the problem.}
		\State{Compute the error estimator $\eta_{\mathrm{eig},j}$, for all $j=1,\dots,M$.}
		 \If{$\max_{j\leq M}(\eta_{\mathrm{eig},j}^2)<\mathrm{tol}$}
    			\State{Return $(\lambda_{j,h},\phi_{j,h})$, for all $j=1,\dots,M$, and and exit.}
  			\EndIf
		\State{$\mathrm{ct}_\mathrm{pause}:=0$}
	\Else
		\State{$\mathrm{ct}_\mathrm{pause}:=\mathrm{ct}_\mathrm{pause}+1$}
	\EndIf
	
    	\State{Adapt $S(\mathcal{T})$ using $\eta_\mathrm{land}$ and $u_h$ by calling AdaptLandscape.}
  	\State{$n:=n+1$}
\EndWhile{}
\end{algorithmic}
\end{algorithm}

In Table~\ref{tab:pause}, Algorithm~\ref{alg:adapt_landscape_pause} is
used to solve the problem already solved in
Example~\ref{SimpleSchrodinger} using different values for the pause
length. The simulation for $\ell_\mathrm{pause}= 0$ is equivalent to
using Algorithm~\ref{alg:adapt_landscape} and it is used as a
reference to measure the reduction in CPU time.  Since
Algorithm~\ref{alg:adapt_landscape_pause} can only stop during an
iteration with the eigensolver active, the total number of meshes
varies depending on the value of $\ell_\mathrm{pause}$. For the same
reason, also the DOF number on the final mesh varies with
$\ell_\mathrm{pause}$.  For any value of $\ell_\mathrm{pause}> 0$ the
reduction is great---up to 70\%.

\begin{table}
\begin{center}
\begin{tabular}{ c c c c c }

  Pause Length   & N Meshes   &  DOF  &  CPU Time &   Reduction(\%)\\
\hline\\
             0    &      36 & 117805  &  30986.7     &      0\\
             1     &     37 & 128366  &  19291.4    &      37.7428\\
             2    &      37 & 128366  &  14276       &     53.9284\\
             3    &      37 & 128366  &  11672.5    &      62.3306\\
             4    &      36 & 117805  &   8940.47   &      71.1473\\
             5    &      37 & 128366  &   9055.94   &      70.7747

\end{tabular}
\end{center}\caption{Comparison of performances for different values for the pause length.}\label{tab:pause}
\end{table}

Since the FE space in Algorithm~\ref{alg:adapt_landscape_pause} is
adapted just using the landscape solution, it might seem a good idea
to postpone any call to the eigensolver after a certain number of mesh
refinements instead of pausing the eigensolver at
intervals. Unfortunately, this approach does not reduce the total CPU
time by much because the CPU time used by the eigensolver on the
initial meshes is very moderate, see Figure~\ref{fig:timing}.

\subsection{Iterative solvers}

In this section, we further reduce the CPU time by exploring different
eigenvalue solvers. So far we used the ARPACK package
\cite{lehoucq_arpack_1998} which is based on the Arnoldi
method. ARPACK is very fast to compute multiple eigenpairs, however,
in an adapting method, not all eigenpairs reach the desired tolerance
on the same mesh. Therefore, it is common that on some meshes only a
\change{few of the eigenpairs} have not converged yet. ARPACK computes
approximations of all eigenpairs, even of the ones already
approximated well enough on previous meshes.  In
\cite{Giani2021a,solin_iterative_2012}, iterative eigenvalue
solvers based on Picard's and Newton's methods are described. Even if
these solvers might be slower than ARPACK in computing several
eigenpairs, they \change{can be used} to just compute a subset of the eigenpairs
in the computed \change{spectrum that are not necessarily contiguous} in the spectrum.

In Figure~\ref{fig:conv_eig}, the number of not converged eigenpairs
are reported for all computed meshes for the problem solved in
Example~\ref{SimpleSchrodinger}. As can be seen, on the last few
computed meshes, only a few eigenpairs have not converged yet but
ARPACK is anyway computing the entire spectrum of 100
eigenpairs. These calls to the eigensolver are also the most expensive
ones, as can be seen in Figure~\ref{fig:timing}.
Also in Figure~\ref{fig:conv_eig}, the CPU time for three
different eigensolvers are tested from mesh 20 and onwards. ARPACK is
still the fastest method until the last few meshes when the other
methods are faster because only a small number of eigenpairs are
computed. This suggests switching to an iterative solver when the
number of not converged eigenpairs is small enough.

\begin{figure}
\centering
\includegraphics[width=0.48\textwidth]{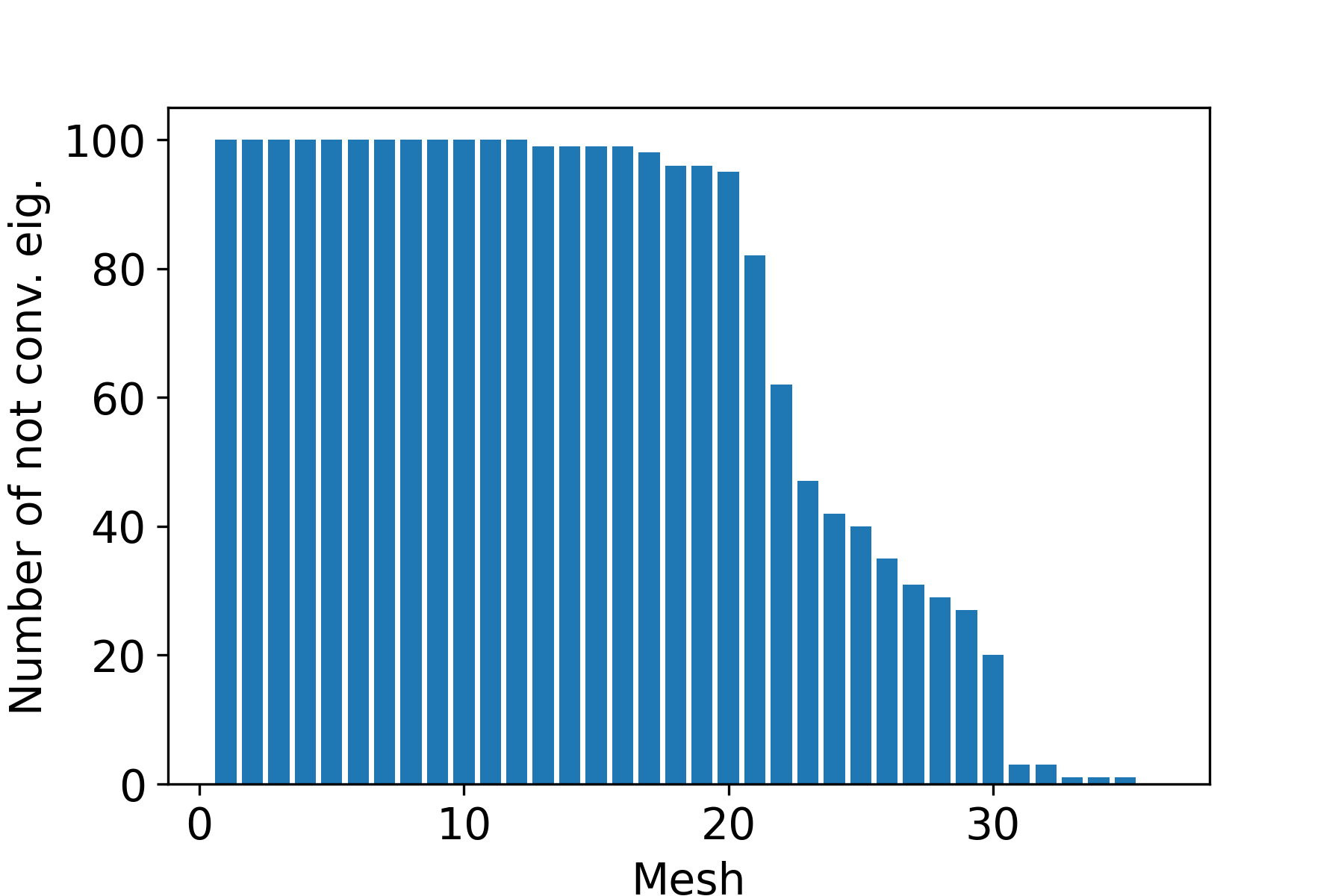}
\includegraphics[width=0.48\textwidth]{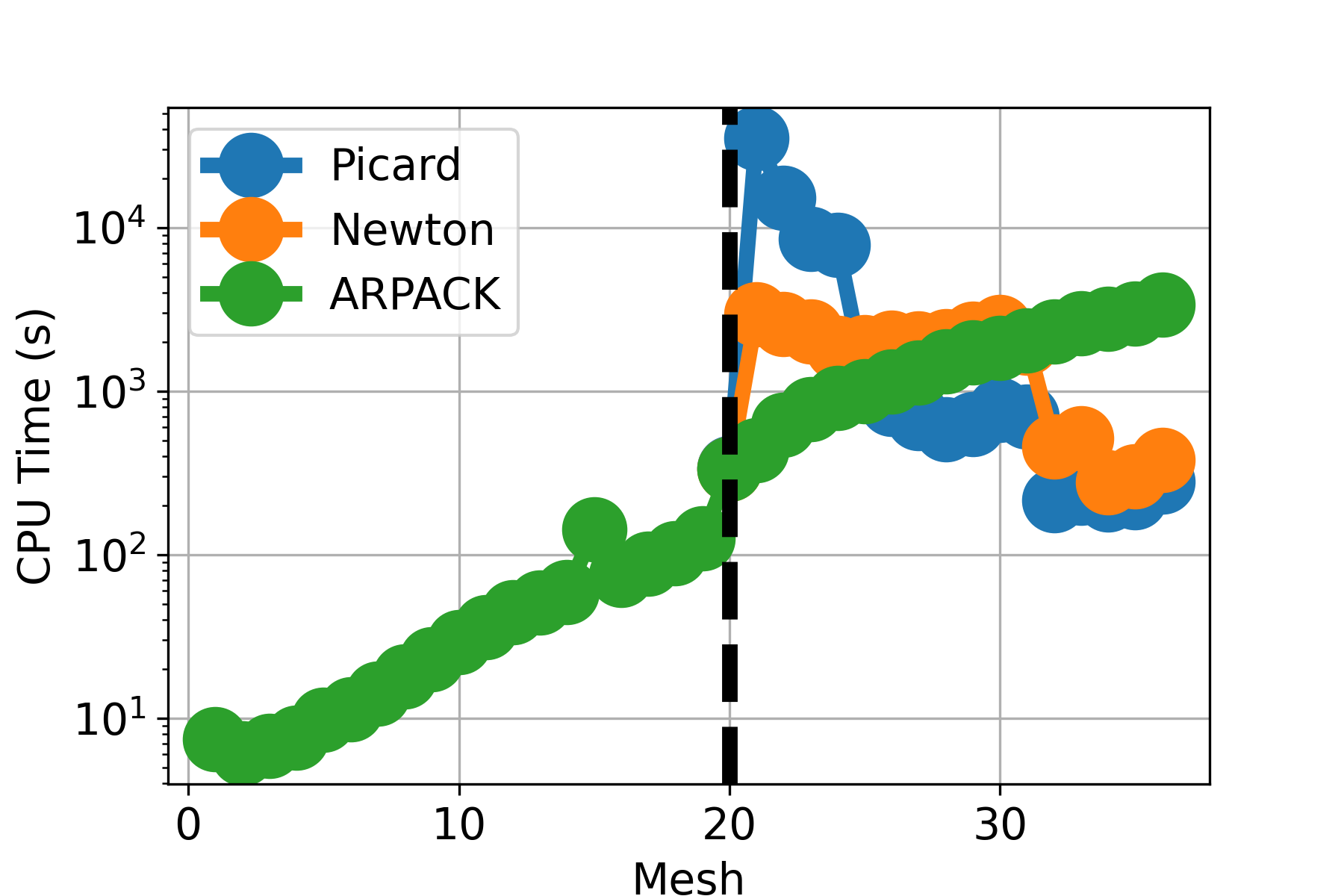}
\caption{Left: Number of not converged eigenpairs on each adapted
  mesh. Right:
CPU time comparison between different eigenvalue solvers} \label{fig:conv_eig}
\end{figure}

In Table~\ref{tab:pause_iterative},
Algorithm~\ref{alg:adapt_landscape_pause} is used to solve the problem
already solved in Example~\ref{SimpleSchrodinger} using different
values for the pause length and using the Picard eigensolver when the
number of not converged eigenpairs is small than 30. The simulation
for $\ell_\mathrm{pause}= 0$ uses only ARPACK and it is equivalent to
using Algorithm~\ref{alg:adapt_landscape} and it is used as a
reference to measure the reduction in CPU time.  Compared to
Table~\ref{tab:pause}, we can see that the use of the iterative solver
reduces further the CPU time.

\begin{table}
\begin{center}
\begin{tabular}{ c c c c c }

  Pause Length   & N Meshes   &  DOF  &  CPU Time &   Reduction(\%)\\
\hline\\
             0 (ARPACK only)    &      36 &  117805  &  30986.7     &      0\\
             1     &     37 & 128366  &  19291.4    &      68.7414\\
             2    &      37 & 128366  &  14276       &     77.0795\\
             3    &      37 & 128366  &  11672.5    &      78.0287\\
             4    &      36 & 117805  &   8940.47   &      80.9712 \\
             5    &      37 & 128366  &   9055.94   &      84.4589

\end{tabular}
\end{center}\caption{Comparison of performances for different values for the pause length and using an iterative solver on the last meshes.}\label{tab:pause_iterative}
\end{table}

%%%%%%%%%%%%%%%%%%%%%%%%%%%%%%%%%%%%%%%%%%%%%%%%%%%%%%%%%%%%%%%%%%%%%%%%%%%
%%%%%%%%%%%%%%%%%%%%%%%%%%%%%%%%%%%%%%%%%%%%%%%%%%%%%%%%%%%%%%%%%%%%%%%%%%%

\section{Conclusions}\label{Conclusions}
We have introduced a new approach for adaptively approximating
multiple eigenpairs of an elliptic operator $\cL$ via finite element
methods, where the adaptivity is driven not by the computed
eigenpairs, but by the computed solution for a single source
problem---the so-called landscape problem $\cL u = 1$.  Although this
approach was motivated by operators that induce strong spatial
localization of many eigenvectors, and for which the landscape
function has been seen to contain an impressive amount of information
about these eigenvectors, we have demonstrated its utility on a
broader range of problems.  It compares favorably with more
traditional approaches, both in terms of CPU time to reach a certain
convergence criterion, and in terms of conceptual
simplicity---\textit{refine based on approximation of a single
  function instead of figuring out how to incorporate information from
  a large collection of functions}.  We also presented a few modifications of
the algorithm that reduce its computational cost even further.
We believe that our landscape refinement approach provides an
attractive alternative to traditional methods, particularly in cases
where large numbers of eigenpairs are sought; and it can also be very
effective in cases where a smaller cluster of eigenpairs is sought, as
a sensible guide for refinement until the finite element space is
``sufficiently rich''.

\section{Declarations}
\paragraph{\bf Funding}
This work was partially supported by the National Science Foundation
through NSF grant DMS-2012285 and NSF RTG grant DMS-2136228.

\paragraph{\bf Competing Interests}
The authors have no competing interests to declare that are relevant
to the content of this article.

\paragraph{\bf Data Availability}
Data will be made available upon reasonable request.

%%%%%%%%%%%%%%%%%%%%%%%%%%%%%%%%%%%%%%%%%%%%%%%%%%%%%%%%%%%%%%%%%%%%%%%%%%%
%%%%%%%%%%%%%%%%%%%%%%%%%%%%%%%%%%%%%%%%%%%%%%%%%%%%%%%%%%%%%%%%%%%%%%%%%%%

%%%% Bibliography


\medskip

\begin{thebibliography}{10}

\bibitem{arnold_unified_2002}
D.~N. Arnold, F.~Brezzi, B.~Cockburn, and L.~D. Marini.
\newblock Unified analysis of discontinuous {G}alerkin methods for elliptic
  problems.
\newblock {\em SIAM J. Numer. Anal.}, 39(5):1749--1779, 2001/02.

\bibitem{Arnold2019}
D.~N. Arnold, G.~David, M.~Filoche, D.~Jerison, and S.~Mayboroda.
\newblock Computing spectra without solving eigenvalue problems.
\newblock {\em SIAM J. Sci. Comput.}, 41(1):B69--B92, 2019.

\bibitem{Arnold2019a}
D.~N. Arnold, G.~David, M.~Filoche, D.~Jerison, and S.~Mayboroda.
\newblock Localization of eigenfunctions via an effective potential.
\newblock {\em Comm. Partial Differential Equations}, 44(11):1186--1216, 2019.

\bibitem{Arnold2016}
D.~N. Arnold, G.~David, D.~Jerison, S.~Mayboroda, and M.~Filoche.
\newblock Effective confining potential of quantum states in disordered media.
\newblock {\em Phys. Rev. Lett.}, 116:056602, Feb 2016.

\bibitem{Bank2013}
R.~E. Bank, L.~Grubi\v{s}i\'c, and J.~S. Ovall.
\newblock A framework for robust eigenvalue and eigenvector error estimation
  and {R}itz value convergence enhancement.
\newblock {\em Appl. Numer. Math.}, 66:1--29, 2013.

\bibitem{Betcke2005}
T.~Betcke and L.~N. Trefethen.
\newblock Reviving the method of particular solutions.
\newblock {\em SIAM Rev.}, 47(3):469--491 (electronic), 2005.

\bibitem{Boffi2017}
D.~Boffi, D.~Gallistl, F.~Gardini, and L.~Gastaldi.
\newblock Optimal convergence of adaptive {FEM} for eigenvalue clusters in
  mixed form.
\newblock {\em Math. Comp.}, 86(307):2213--2237, 2017.

\bibitem{Borsuk2006} M.~Borsuk, V.~A.~Kondratiev.  \newblock {\em
    Elliptic boundary value problems of second order in piecewise
    smooth domains}, volume~60 of {\em North-Holland Mathematical
    Library}.  \newblock Elsevier Science B.V., Amsterdam, 2006.


\bibitem{Cances2020}
E.~Canc\`es, G.~Dusson, Y.~Maday, B.~Stamm, and M.~Vohral\'ik.
\newblock Guaranteed a posteriori bounds for eigenvalues and eigenvectors:
  multiplicities and clusters.
\newblock {\em Math. Comp.}, 89(326):2563--2611, 2020.

\bibitem{Canuto2019}
C.~Canuto.
\newblock Adaptive hp-{FEM} for eigenvalue computations.
\newblock {\em Calcolo}, 56(4):Paper No. 39, 25, 2019.

\bibitem{Canuto2017}
C.~Canuto, R.~H. Nochetto, R.~Stevenson, and M.~Verani.
\newblock Convergence and optimality of {${\bf{hp}}$}-{\bf {afem}}.
\newblock {\em Numer. Math.}, 135(4):1073--1119, 2017.

\bibitem{David2021}
G.~David, M.~Filoche, and S.~Mayboroda.
\newblock The landscape law for the integrated density of states.
\newblock {\em Adv. Math.}, 390:Paper No. 107946, 2021.

\bibitem{ern_discontinuous_2009}
A.~Ern, A.~F. Stephansen, and P.~Zunino.
\newblock A discontinuous {G}alerkin method with weighted averages for
  advection-diffusion equations with locally small and anisotropic diffusivity.
\newblock {\em IMA J. Numer. Anal.}, 29(2):235--256, 2009.

\bibitem{Filoche2012}
M.~Filoche and S.~Mayboroda.
\newblock Universal mechanism for {A}nderson and weak localization.
\newblock {\em Proc. Natl. Acad. Sci. USA}, 109(37):14761--14766, 2012.

\bibitem{Gallistl2015}
D.~Gallistl.
\newblock An optimal adaptive {FEM} for eigenvalue clusters.
\newblock {\em Numer. Math.}, 130(3):467--496, 2015.

\bibitem{giani_reliable_2018}
S.~Giani.
\newblock Reliable anisotropic-adaptive discontinuous {G}alerkin method for
  simplified {$\bf P_N$} approximations of radiative transfer.
\newblock {\em J. Comput. Appl. Math.}, 337:225--243, 2018.

\bibitem{Giani2021}
S.~Giani, L.~Grubi{\v{s}}i{\'{c}}, H.~Hakula, and J.~S. Ovall.
\newblock A posteriori error estimates for elliptic eigenvalue problems using
  auxiliary subspace techniques.
\newblock {\em Journal of Scientific Computing}, 88(3):55(25), jul 2021.

\bibitem{giani_posteriori_2018}
S.~Giani, L.~Grubi\v{s}i\'c, H.~Hakula, and J.~S. Ovall.
\newblock An a posteriori estimator of eigenvalue/eigenvector error for
  penalty-type discontinuous {G}alerkin methods.
\newblock {\em Appl. Math. Comput.}, 319:562--574, 2018.

\bibitem{Giani2016a}
S.~Giani, L.~Grubi\v{s}i\'{c}, and J.~S. Ovall.
\newblock Benchmark results for testing adaptive finite element eigenvalue
  procedures part 2 (conforming eigenvector and eigenvalue estimates).
\newblock {\em Appl. Numer. Math.}, 102:1--16, 2016.

\bibitem{giani_posteriori_2012}
S.~Giani and E.~J.~C. Hall.
\newblock An {\it a posteriori} error estimator for {$hp$}-adaptive
  discontinuous {G}alerkin methods for elliptic eigenvalue problems.
\newblock {\em Math. Models Methods Appl. Sci.}, 22(10):1250030, 35, 2012.

\bibitem{Giani2021a}
S.~Giani and P.~Solin.
\newblock Solving elliptic eigenproblems with adaptive multimesh {$hp$}-{FEM}.
\newblock {\em J. Comput. Appl. Math.}, 394:Paper No. 113528, 15, 2021.

\bibitem{Grisvard1985}
P.~Grisvard.
\newblock {\em Elliptic problems in nonsmooth domains}, volume~24 of {\em
  Monographs and Studies in Mathematics}.
\newblock Pitman (Advanced Publishing Program), Boston, MA, 1985.

\bibitem{Grisvard1992}
P.~Grisvard.
\newblock {\em Singularities in boundary value problems}, volume~22 of {\em
  Recherches en Math\'ematiques Appliqu\'ees [Research in Applied
  Mathematics]}.
\newblock Masson, Paris, 1992.

\bibitem{Grubisic2009}
L.~Grubi\v{s}i\'c and J.~S. Ovall.
\newblock On estimators for eigenvalue/eigenvector approximations.
\newblock {\em Math. Comp.}, 78:739--770, 2009.

\bibitem{houston_energy_2007}
P.~Houston, D.~Sch\"otzau, and T.~P. Wihler.
\newblock Energy norm a posteriori error estimation of {$hp$}-adaptive
  discontinuous {G}alerkin methods for elliptic problems.
\newblock {\em Math. Models Methods Appl. Sci.}, 17(1):33--62, 2007.

\bibitem{houston_note_2005}
P.~Houston and E.~S\"uli.
\newblock A note on the design of {$hp$}-adaptive finite element methods for
  elliptic partial differential equations.
\newblock {\em Comput. Methods Appl. Mech. Engrg.}, 194(2-5):229--243, 2005.

\bibitem{Kellogg1974}
R.~B. Kellogg.
\newblock On the {P}oisson equation with intersecting interfaces.
\newblock {\em Applicable Anal.}, 4:101--129, 1974/75.
\newblock Collection of articles dedicated to Nikolai Ivanovich Muskhelishvili.

\bibitem{Kondratiev1967}
  V.~A.~Kondrat'ev.
  \newblock Boundary value problems for elliptic equations in domains
  with conical or angular points  (Russian).
  \newblock {\em Trudy Moskov. Mat. Obšč.}, 16: 209–-292, 1967.

\bibitem{lehoucq_arpack_1998}
R.~B. Lehoucq, D.~C. Sorensen, and C.~Yang.
\newblock {\em A{RPACK} users' guide}, volume~6 of {\em Software, Environments,
  and Tools}.
\newblock Society for Industrial and Applied Mathematics (SIAM), Philadelphia,
  PA, 1998.
\newblock Solution of large-scale eigenvalue problems with implicitly restarted
  Arnoldi methods.

\bibitem{Liu2022}
X.~Liu and T.~Vejchodsk\'{y}.
\newblock Fully computable a posteriori error bounds for eigenfunctions.
\newblock {\em Numer. Math.}, 152(1):183--221, 2022.

\bibitem{prudhomme_review_2000}
S.~Prudhomme, F.~Pascal, J.~Oden, and A.~Romkes.
\newblock Review of a priori error estimation for discontinuous {Galerkin}
  {Methods}.
\newblock Technical Report TICAM REPORT 00-27, The University of Texas at
  Austin, 2000.

\bibitem{solin_iterative_2012}
P.~Solin and S.~Giani.
\newblock An iterative adaptive finite element method for elliptic eigenvalue
  problems.
\newblock {\em J. Comput. Appl. Math.}, 236(18):4582--4599, 2012.

\bibitem{Trefethen2006}
L.~N. Trefethen and T.~Betcke.
\newblock Computed eigenmodes of planar regions.
\newblock In {\em Recent advances in differential equations and mathematical
  physics}, volume 412 of {\em Contemp. Math.}, pages 297--314. Amer. Math.
  Soc., Providence, RI, 2006.

\bibitem{Wigley1964}  
  N.M.~Wigley.
  \newblock Asymptotic expansions at a corner of solutions of mixed
  boundary value problems.
  \newblock {\em J. Math. Mech.}, 13:549--576, 1964.
  
\end{thebibliography}
\end{document}